\documentclass[a4paper, 12pt, reqno]{amsart}

\pdfoutput=1

\usepackage[utf8]{inputenc}
\usepackage[margin=2cm, headsep=0.75cm, footskip=0.75cm]{geometry}
\usepackage{amsmath, amsthm, amssymb, mathtools, physics, thmtools}
\usepackage[hang, flushmargin]{footmisc}
\usepackage[english]{isodate}
\usepackage{enumitem}
\usepackage[dvipsnames]{xcolor}
\usepackage{tikz}
\usepackage{tikz-cd}
\usepackage[noadjust]{cite}
\usepackage[colorlinks=true, linkcolor=blue, linkbordercolor=blue, citecolor=red, citebordercolor=red, urlcolor=indigo, linktocpage=true]{hyperref}
\usepackage[capitalise, nameinlink, noabbrev, nosort]{cleveref}
\usepackage{doi}

\usetikzlibrary{decorations.markings}

\definecolor{indigo}{HTML}{492DA5}

\allowdisplaybreaks

\setenumerate{listparindent=\parindent}

\makeatletter\g@addto@macro\bfseries{\boldmath}\makeatother

\makeatletter
\let\origsection\section
\renewcommand{\section}{\@ifstar{\starsection}{\nostarsection}}
\newcommand{\sectionspace}{\vspace{0.5ex}}
\newcommand{\nostarsection}[1]{\sectionspace\origsection{#1}\sectionspace}
\newcommand{\starsection}[1]{\sectionspace\origsection*{#1}\sectionspace}
\makeatother

\setlist[enumerate]{font=\normalfont}
\crefname{enumi}{}{}
\crefname{enumii}{}{}

\newcommand\numberthis{\addtocounter{equation}{1}\tag{\theequation}}

\numberwithin{equation}{section}
\crefname{equation}{Equation}{Equations}

\crefname{inequality}{Inequality}{Inequalities}
\creflabelformat{inequality}{#2(#1)#3}

\newenvironment{inequality}{\begin{equation}}{\end{equation}\ignorespacesafterend}

\AddToHook{env/inequality/begin}{\crefalias{equation}{inequality}}

\newtheorem{thmx}{Theorem}

\crefname{thmx}{Theorem}{Theorems}

\newtheorem{theorem}{Theorem}[section]
\newtheorem{thm}[theorem]{Theorem}
\crefname{thm}{Theorem}{Theorems}

\newtheorem{lemma}[theorem]{Lemma}
\crefname{lemma}{Lemma}{Lemmas}

\newtheorem{prop}[theorem]{Proposition}
\crefname{prop}{Proposition}{Propositions}

\newtheorem{cor}[theorem]{Corollary}
\crefname{cor}{Corollary}{Corollaries}

\theoremstyle{definition}

\newtheorem{defn}[theorem]{Definition}
\crefname{defn}{Definition}{Definitions}

\newtheorem{remark}[theorem]{Remark}
\crefname{remark}{Remark}{Remarks}

\newtheorem{example}[theorem]{Example}
\crefname{example}{Example}{Examples}

\renewcommand{\d}{\mathrm{d}}
\newcommand{\C}{\mathbb{C}}
\newcommand{\N}{\mathbb{N}}
\newcommand{\R}{\mathbb{R}}
\newcommand{\Z}{\mathbb{Z}}
\newcommand{\CC}{\mathcal{C}}
\newcommand{\EE}{\mathcal{E}}
\newcommand{\GG}{\mathcal{G}}
\newcommand{\II}{\mathcal{I}}
\newcommand{\JJ}{\mathcal{J}}
\newcommand{\KK}{\mathcal{K}}
\newcommand{\MM}{\mathcal{M}}
\newcommand{\NN}{\mathcal{N}}
\newcommand{\QQ}{\mathcal{Q}}
\newcommand{\WW}{\mathcal{W}}
\newcommand{\GGo}{\GG^{(0)}}
\newcommand{\rmU}{\mathrm{U}}
\newcommand{\SO}[1]{\mathrm{SO}(#1)}
\newcommand{\loc}{\mathrm{loc}}
\newcommand{\Mloc}{\MM_\loc}
\newcommand{\id}{\mathrm{id}}
\newcommand{\supp}{\operatorname{supp}}
\newcommand{\osupp}{\supp^{\mathrm{o}}}
\newcommand{\vecspan}{\operatorname{span}}
\newcommand{\cl}[2][]{{\overline{{#2}}}^{#1}}
\newcommand{\clspan}{\cl{\vecspan}}
\newcommand{\restr}[1]{\ensuremath{\vert_{#1}}}
\newcommand{\dom}{\operatorname{dom}}
\newcommand{\ran}{\operatorname{ran}}
\newcommand{\im}{\operatorname{im}}
\newcommand{\Iso}{\operatorname{Iso}}
\newcommand{\alg}{\mathrm{alg}}
\newcommand{\ess}{\mathrm{ess}}
\newcommand{\full}{\mathrm{max}}
\newcommand{\red}{\mathrm{red}}
\newcommand{\simalpha}{\sim_\alpha}
\newcommand{\rtimesalg}{\rtimes_\alg}
\newcommand{\rtimesess}{\rtimes_\ess}
\newcommand{\rtimesfull}{\rtimes_\full}
\newcommand{\rtimesred}{\rtimes_\red}
\newcommand{\ipmid}{\,{\mid}\,}
\newcommand{\Cu}{\mathrm{Cu}}
\newcommand{\Cuc}{\Cu_c}
\newcommand{\Lsc}{\mathrm{Lsc}}
\newcommand{\EL}{\mathrm{EL}}
\newcommand{\ER}{\mathrm{E}}
\newcommand{\precalpha}{\preceq_\alpha}
\newcommand{\rmpath}{\mathrm{path}}
\newcommand{\pathalpha}{\preceq_{\alpha,\rmpath}}

\hypersetup{pdfauthor = {Becky Armstrong, Lisa Orloff Clark, Astrid an Huef, Diego Mart\'inez, and Ilija Tolich}, pdftitle = {A dichotomy for inverse-semigroup crossed products via dynamical Cuntz semigroups}}

\date{\today}
\title[A dichotomy for inverse-semigroup crossed products]{A dichotomy for inverse-semigroup crossed products via dynamical Cuntz semigroups}
\author[Armstrong]{Becky Armstrong}
\author[Clark]{Lisa Orloff Clark}
\author[an Huef]{Astrid an Huef}
\author[Mart\'inez]{Diego Mart\'inez}
\author[Tolich]{Ilija Tolich}

\address[B.~Armstrong, L.O.~Clark, A.~an Huef, and I.~Tolich]{School of Mathematics and Statistics, Victoria University of Wellington, PO Box 600, Wellington 6140, NEW ZEALAND.}
\email[B.~Armstrong]{\href{mailto:becky.armstrong@vuw.ac.nz}{becky.armstrong@vuw.ac.nz}}
\email[L.O.~Clark]{\href{mailto:lisa.orloffclark@vuw.ac.nz}{lisa.orloffclark@vuw.ac.nz}}
\email[A.~an Huef]{\href{mailto:astrid.anhuef@vuw.ac.nz}{astrid.anhuef@vuw.ac.nz}}
\email[I.~Tolich]{\href{mailto:ilija.tolich@vuw.ac.nz}{ilija.tolich@vuw.ac.nz}}
\address[D.~Mart\'inez]{Department of Mathematics, KU Leuven, Celestijnenlaan 200B, 3001 Leuven, BELGIUM.}
\email[D.~Mart\'inez]{\href{mailto:diego.martinez@kuleuven.be}{diego.martinez@kuleuven.be}}


\thanks{This research was supported by the Marsden Fund of the Royal Society of New Zealand (21\nobreakdash-VUW\nobreakdash-156 and 24\nobreakdash-VUW\nobreakdash-014); the Deutsche Forschungsgemeinschaft (EXC 2044 -- 390685587, Mathematics M\"unster -- Dynamics -- Geometry -- Structure and Project-ID 427320536 -- SFB 1442); the European Research Council (Advanced Grant 834267 -- AMAREC); the Research Foundation Flanders (projects G085020N and 1218726N); and the Engineering and Physical Sciences Research Council of the UK (Grant Number EP/V521929/1). The authors thank the Isaac Newton Institute for Mathematical Sciences for support and hospitality during the programme ``Topological groupoids and their C*-algebras'' where work on this paper was undertaken; thus this work was supported by EPSRC grant no.~EP/Z000580/1 and the Simons Foundation Award SFI-MPS-T-Institutes-00006117.}

\begin{document}

\begin{abstract}
We characterise stable finiteness and pure infiniteness of the essential crossed product of a C*-algebra by an action of an inverse semigroup. Under additional assumptions, we prove a stably finite / purely infinite dichotomy. Our main technique is the development, using an induced action, of a ``dynamical Cuntz semigroup'' that is a subquotient of the usual Cuntz semigroup. We prove that the essential crossed product is stably finite / purely infinite if and only if the dynamical Cuntz semigroup admits / does not admit a nontrivial state. Indeed, a retract of our dynamical Cuntz semigroup suffices to prove the dichotomy. Our results generalise those by Rainone on crossed products of groups acting by automorphisms of a C*-algebra, and we recover results by Kwa\'sniewski--Meyer--Prasad on C*-algebras of non-Hausdorff groupoids.
\end{abstract}

\maketitle
\vspace{-1.5ex}
\tableofcontents
\vspace{-1.5ex}

\section{Introduction}

A set is Dedekind-infinite if it admits a bijection onto a proper subset of itself; otherwise, it is Dedekind-finite. These definitions are intrinsic to the set in question---that is, they do not depend on any extrinsic structure, such as the set of natural numbers. A Dedekind-infinite set consists of the space itself, which can be decomposed into smaller pieces and then glued back together, and some dynamics, which are given by partial bijections and identify certain pieces with others. These ideas have inspired numerous works for more than a century, and in particular motivate many of the ideas in this paper, including our \emph{dynamical Cuntz semigroup}, which is a kind of \emph{type semigroup}.

Type semigroups originally appeared in the work of Banach--Tarski \cite{TW2016}, and encapsulate the idea that ``$2$ units may be equal to $1$ unit''. The prime example of this behaviour is the Banach--Tarski paradox \cite{TW2016}, which involves the decomposition of the unit ball in $\R^3$ (the space) into smaller subsets (the pieces), which are moved via $\SO{3}$ (the dynamics) and then glued together (via taking unions) to form two copies of the unit ball. Tarski's original definition of a type semigroup exploits this phenomenon: the elements of the semigroup are (copies of) subsets of the unit ball, and two subsets are equivalent if one can be decomposed into smaller pieces, these moved via $\SO{3}$, and the images rearranged to get the other set.

Another manifestation of this same idea has appeared in the study of C*-algebras. Given a C*-algebra $A$, we say that projections $p, q \in A$ are Murray--von Neumann equivalent, if $p = v^*v$ and $q = vv^*$ for some $v \in A$. We view the projections in $A$ as the pieces of $A$ that are made to be equivalent via partial isometries (the dynamics). This leads to the definition of the $K_0$-group $K_0(A)$, which may be viewed as a type semigroup (see \cite{Blackadar1998} and references therein). If $A$ contains few or no projections, it is preferable to view the positive elements $a \in A^+$ (of which there are plenty) as the pieces of $A$, and a similar approach to the construction of $K_0(A)$ then yields the \emph{Cuntz semigroup} $\Cu(A)$ of $A$ (see \cite{CEI2008, GardellaPerera2024, ORT2011, Thiel2017}).

Likewise, given a discrete group $\Gamma$ acting on a locally compact Hausdorff space $X$ by homeomorphisms, we may consider open sets $A, B \subseteq X$ to be \emph{pieces} of $X$, and say that $A \precsim B$ if $A$ can be covered by finitely many open sets and these moved via the action of $\Gamma$ to be properly contained in $B$. This leads to the notion of \emph{comparison} of open sets (see \cite{Kerr2020, Ma2022, Rainone2017} and references therein). Other examples of type semigroups appear in \cite{ALM2020, ABBL2020, ABBL2023, BoenickeLi2020, BPWZ2025, KMP2025, Ma2023, Rainone2017, RS2020}. In most of these references, the type semigroups are used to capture a stably finite / purely infinite dichotomy of the respective crossed-product C*-algebras.

This paper grew out of a desire to generalise Rainone's results from \cite{Rainone2017} about crossed products of groups acting by automorphisms of C*-algebras, and analogous results about Hausdorff groupoid C*-algebras due to B\"onicke and Li \cite{BoenickeLi2020} and Rainone and Sims \cite{RS2020}. To achieve this, we consider actions of \emph{inverse semigroups} on C*-algebras; see \cref{subsec: inverse semigroups}. The C*-algebras associated to such actions generalise both the class of groupoid C*-algebras and the class of crossed products by groups acting on (possibly noncommutative) C*-algebras, so our approach should be compared to those of \cite{BoenickeLi2020, BPWZ2025, KMP2025, Ma2022, Rainone2016}.

We introduce a dynamical Cuntz semigroup $\Cu(A)_\alpha$ for an action $\alpha$ of an inverse semigroup $S$ on a (possibly noncommutative) C*-algebra $A$, and we relate the structure of this ``type semigroup'' to a stably finite / purely infinite dichotomy of the essential crossed product $A \rtimesess S$. We use the Cuntz semigroup rather than the $K_0$-group to allow us to discuss actions on C*\nobreakdash-algebras that are not of real rank $0$.

Following this endeavour it quickly becomes apparent that there is an inherent tension between the reduced crossed product $A \rtimesred S$ and the essential crossed product $A \rtimesess S$ \cite{ExelPitts2022, KwasniewskiMeyer2021}. This tension arises because the canonical method of extending an invariant trace on $A$ to $A \rtimesred S$ does not work for $A \rtimesess S$, and even though $A \rtimesess S$ is a quotient of $A \rtimesred S$, not every trace on $A \rtimesred S$ factors through a trace on $A \rtimesess S$. In \cite{KMP2025} Kwa\'sniewski, Meyer, and Prasad resolve this tension by requiring $A \rtimesess S$ to be isomorphic to $A \rtimesred S$, and we follow this strategy here as well. In essence, $A \rtimesred S$ has the ``correct'' functorial properties, while $A \rtimesess S$ has the ``correct'' ideal structure. Thus, the condition that $A \rtimesred S \cong A \rtimesess S$ heuristically means that this crossed product is well-behaved with respect to functoriality (as the Cuntz semigroup should be) and with respect to ideals (which we need to consider because we want to construct faithful traces on the crossed product).

Once we establish a characterisation of when the crossed product is stably finite or purely infinite, the main theorem in the paper is the following dichotomy.

\begin{thmx}[{\cref{thm: dichotomy result}}] \label{thm: main}
Let $\alpha\colon S \curvearrowright A$ be an action of a unital inverse semigroup $S$ on a unital and exact C*-algebra $A$. Suppose that $A \rtimesred S$ is simple, and that the dynamical Cuntz semigroup $\Cu(A)_\alpha$ (as defined in \cref{defn: dynamical cuntz semigroup}) has plain paradoxes. Then $A \rtimesred S$ is either stably finite or purely infinite.
\end{thmx}

In general, the Cuntz semigroup of an arbitrary C*-algebra is difficult to compute. However, we establish that a smaller object---a retract of a $\Cu$-semigroup that is strongly invariant with respect to the induced action (see \cref{defn: retract,defn: invariant retract})---can be used instead. This gives a second form of our dichotomy theorem.

\begin{thmx}[{\cref{thm: retract dichotomy}}] \label{thm: main retract}
Let $\alpha\colon S \curvearrowright A$ be an action of a unital inverse semigroup $S$ on a unital and exact C*-algebra $A$. Suppose that $A \rtimesred S$ is simple, and that $(R,\rho,\sigma)$ is a strongly invariant retract of $\Cu(A)$ such that $[a]_\Cu \in \rho(R)$ for each $a \in A^+$. Let $R_\alpha$ be the retracted dynamical Cuntz semigroup, as defined in \cref{defn: retracted DCS}. If $R_\alpha$ has plain paradoxes, then $A \rtimesred S$ is either stably finite or purely infinite.
\end{thmx}

The plain paradoxes hypothesis in \cref{thm: main,thm: main retract} remains difficult to verify. This condition is common in the type semigroup literature (see \cite{Rainone2017, BoenickeLi2020, KMP2025, Rainone2017, RS2020}). For certain classes of examples, for example self-similar actions on directed graphs for which the associated groupoid is minimal, the associated type semigroup has plain paradoxes and the stably finite / purely infinite dichotomy holds (see \cite[Proposition~7.9 and Theorem~7.15]{KMP2025}).

For every \'etale groupoid $\GG$ with compact Hausdorff unit space $X$, there is an action of an inverse semigroup $S$ on $C(X)$ such that the reduced and essential C*-algebras of $\GG$ are crossed products $C(X) \rtimesred S$ and $C(X) \rtimesess S$; see \cref{sec: groupoid C*-algebras} for the details. We show that the set $\Lsc(X,\bar{\N})$ of lower-semicontinuous functions $F\colon X \to \N \cup \{\infty\}$ is an invariant retract of $\Cu(C(X))$. Thus, as an immediate corollary of \cref{thm: main retract}, we obtain an explicit stably finite / purely infinite dichotomy result for the essential C*-algebra of $\GG$ in \cref{thm: retract dichotomy}. \cref{thm: retract dichotomy} is a special case of the dichotomy results in \cite{KMP2025} and should be compared with \cite{BoenickeLi2020, Rainone2017, RS2020, Ma2022}.

The paper is organised as follows. In \cref{sec: preliminaries} we discuss much of the background necessary for the rest of the paper. In \cref{sec: Cuntz semigroups} we introduce the necessary preliminaries on Cuntz semigroups. In \cref{sec: dynamical Cuntz semigroup} we construct our dynamical Cuntz semigroup, and we discuss how minimality of the inverse-semigroup action can be seen in it. In \cref{sec: stably finite,sec: purely infinite} we present and prove characterisations of stable finiteness and pure infiniteness of the crossed product, and we then combine the two to prove our dichotomy theorem (\cref{thm: dichotomy result}). In \cref{sec: retracts} we consider retracts of Cuntz semigroups and show that these simpler objects retain sufficient structure to state our dichotomy results. Finally, in \cref{sec: groupoid C*-algebras}, we discuss the consequences of our main theorems within the context of non-Hausdorff \'etale groupoid C*-algebras.

\vspace{1ex}

\noindent\textbf{Acknowledgements.} We are grateful to Hannes Thiel and Bartosz Kwa\'sniewski for useful conversations about Cuntz semigroups and similar results in the groupoid setting. We also thank Aaron Tikuisis for bringing \cite{BPWZ2025} to our attention, and we thank Changyuan Gao for alerting us to an error in \cref{thm: homomorphism induces Cu-morphism} in the first version of the paper. Finally, we thank the anonymous referee for their careful reviewing of the paper and for their helpful comments that in particular enabled us to remove all assumptions of separability from our results.

\section{Preliminaries}
\label{sec: preliminaries}

In this section we introduce the background necessary for the rest of the paper.

\subsection{Actions of inverse semigroups}
\label{subsec: inverse semigroups}

For inverse semigroups and actions of inverse semigroups on C*-algebras, we follow \cite{BussExelMeyer2017, BussMartinez2023, Exel2008, Paterson1999, Sieben1997}.

\begin{defn}
An \emph{inverse semigroup} is a semigroup $S$ such that every $s \in S$ has a unique \emph{inverse} $s^* \in S$ satisfying $ss^*s = s$ and $s^*ss^* = s^*$.
\end{defn}

Every inverse semigroup has a natural partial order given by $s \le t \iff s = ss^*t$ for $s, t \in S$. An $e \in S$ is called an \emph{idempotent} if $e^2 = e$. We write $E(S)$ for the collection of idempotents in $S$. For $e, f \in E(S)$, we have $e^* = e$, and hence $e \le f \iff e = ef$, and $E(S) = \{ ss^* : s \in S \}$. Moreover, $E(S)$ is a commutative subsemigroup of $S$, and it follows that $s \le t \iff s = ts^*s$ for $s, t \in S$. The commutativity of $E(S)$ implies that for all $s \in S$ and $e \in E(S)$, we have $ses^*, s^*es \in E(S)$. If $S$ has an identity element, which we denote by $1 \in S$, then $E(S) = \{s \in S : s \le 1\}$.

\begin{defn}[{\cite[Definition~4.2]{Exel2008}}]
Let $X$ be a set. Define \emph{$\II(X)$} to be the set of partial bijections on $X$; that is, bijections $f\colon \dom(f) \to \ran(f)$, where $\dom(f)$ and $\ran(f)$ are subsets of $X$.
\end{defn}

The set $\II(X)$ is an inverse semigroup with respect to composition ``where it makes sense''; that is, $f \circ g$ is defined by $(f \circ g)(x) \coloneqq f(g(x))$ for all $x \in g^{-1}\big(\!\ran(g) \cap \dom(f)\big)$. The inverse of $f \in \II(X)$ is the map $f^*\colon \ran(f) \to \dom(f)$ given by $f^*(x) \coloneqq f^{-1}(x)$. Note that for each $f \in \II(X)$, $f^*f$ is the identity map on $\dom(f)$, and $ff^*$ is the identity map on $\ran(f)$. We may therefore identify $f^*f$ with $\dom(f)$ and $ff^*$ with $\ran(f)$.

\begin{defn}[{\cite[Definition~9.1]{Exel2008}}] \label{defn: C*-algebra action}
An \emph{action} $\alpha\colon S \curvearrowright A$ of an inverse semigroup $S$ on a C*-algebra $A$ is an inverse semigroup homomorphism $\alpha\colon S \to \II(A)$ such that
\begin{enumerate}[label=(\alph*)]
\item \label{item: alpha_s iso} for each $s \in S$, $\alpha_s$ is a C*-isomorphism and its domain $A_{s^*s}$ is an ideal of $A$; and
\item \label{item: ideals dense} the linear span of $\cup_{s \in S} \, A_{s^*s}$ is dense in $A$.
\end{enumerate}
\end{defn}

An action of an inverse semigroup $S$ on a C*-algebra $A$ can be viewed as a generalisation of the natural action of the inverse semigroup of open bisections of an \'etale groupoid. The details of this connection are made explicit in \cref{subsec: groupoid C*s as CPs} and in \cite[Definition~2.14 and Theorem~6.13]{KwasniewskiMeyer2021}.

Below are some immediate consequences of \cref{defn: C*-algebra action} that we will use throughout the paper without necessarily referencing \cref{lemma: action properties}.

\begin{lemma} \label{lemma: action properties}
Let $\alpha\colon S \curvearrowright A$ be an action of an inverse semigroup $S$ on a C*-algebra $A$. Then
\begin{enumerate}[label=(\alph*)]
\item \label{item: alpha_s inverse and idempotent} $\alpha_{s^*} = \alpha_s^{-1}$ and $\alpha_e = \id_{A_e}$ for all $s \in S$ and $e \in E(S)$;
\item \label{item: idempotent ideals} $A_e \cap A_f = A_{ef}$ for all idempotents $e, f \in E(S)$; and
\item \label{item: alpha_s ideal containment} if $s, t \in S$ satisfy $s \le t$, then $A_{s^*s} \subseteq A_{t^*t}$ and $A_{ss^*} \subseteq A_{tt^*}$.
\end{enumerate}
Moreover, if $S$ has a unit $1$, then $A_1 = A$.
\end{lemma}

\begin{proof}
Parts \crefrange{item: alpha_s inverse and idempotent}{item: alpha_s ideal containment} follow from \cref{defn: C*-algebra action}\cref{item: alpha_s iso} and the paragraph preceding it. The final statement follows from \cref{defn: C*-algebra action}\cref{item: ideals dense} using that $E(S) = \{s \in S : s \le 1\}$.
\end{proof}

\subsection{Ideals}

In this subsection we introduce ideals of a C*-algebra $A$ that are invariant with respect to a given action of an inverse semigroup $S$, and we study certain special ideals of $A$ whose elements are fixed by the actions of particular elements of $S$.

\begin{defn}
Let $\alpha\colon S \curvearrowright A$ be an action of an inverse semigroup $S$ on a C*-algebra $A$. We say that a subset $I \subseteq A$ is \emph{invariant} under $\alpha$ if, for all $s \in S$, $\alpha_s(I \cap A_{s^*s}) \subseteq I$. We say that $\alpha$ is \emph{minimal} if the only invariant ideals of $A$ are $\{0\}$ and $A$ itself.
\end{defn}

Following \cite[Definition~2.12]{BussMartinez2023} (see also \cite{BussExelMeyer2017}), for each pair $(s,t) \in S \times S$, we define an ideal $I_{s,t}$ on which $\alpha_s$ and $\alpha_t$ agree.

\begin{defn}
Let $\alpha\colon S \curvearrowright A$ be an action of an inverse semigroup $S$ on a C*-algebra $A$, and fix $s, t \in S$. Define
\[
I_{s,t} \coloneqq \clspan\{ a \in A : a \in A_{v^*v} \text{ for some } v \in S \text{ with } v \le s, t \}.
\]
\end{defn}

\begin{lemma} \label{lemma: I_st properties}
Let $\alpha\colon S \curvearrowright A$ be an action of an inverse semigroup $S$ on a C*-algebra $A$. Fix $s, t \in S$. Then $I_{s,t}$ is an ideal of $A$, and $I_{s,t} = I_{t,s} \subseteq A_{s^*s} \cap A_{t^*t}$.
\end{lemma}

\begin{proof}
Since $A_{v^*v}$ is an ideal of $A$ for each $v \in S$ with $v \le s,t$, a routine argument shows that the set
\[
\vecspan\{ a \in A : a \in A_{v^*v} \text{ for some } v \in S \text{ with } v \le s, t \}
\]
is an algebraic ideal of $A$, and from this it follows that $I_{s,t}$ is an ideal of $A$. For each $v \in S$ with $v \le s, t$, we have $v = sv^*v$, and so by \cref{lemma: action properties}\cref{item: idempotent ideals},
\[
A_{v^*v} = A_{(v^*vs^*)(sv^*v)} = A_{v^*v} \cap A_{s^*s} \subseteq A_{s^*s}.
\]
Similarly, we have $v = tv^*v$, and so $A_{v^*v} \subseteq A_{t^*t}$. Thus $I_{s,t} \subseteq A_{s^*s} \cap A_{t^*t}$, and we have $I_{s,t} = I_{t,s}$ by construction.
\end{proof}

\Cref{lemma: I_st properties} shows that $I_{s,t}$ is contained in the domains of both $\alpha_s$ and $\alpha_t$. We now show that $\alpha_s$ and $\alpha_t$ agree on $I_{s,t}$.

\begin{lemma} \label{lemma: alpha_s is alpha_t on I_st}
Let $\alpha\colon S \curvearrowright A$ be an action of an inverse semigroup $S$ on a C*-algebra $A$, and fix $s, t \in S$. Then $\alpha_s(a) = \alpha_t(a)$ for all $a \in I_{s,t}$.
\end{lemma}

\begin{proof}
Fix $a \in I_{s,t}$. Since $\alpha_s$ and $\alpha_t$ are bounded linear maps and every element of $I_{s,t}$ can be approximated arbitrarily well by a finite sum $\sum_{v \le s, t} \, a_v$, where $a_v \in A_{v^*v}$, it suffices to assume that $a \in A_{v^*v}$ for some $v \le s,t$ and to then show that $\alpha_s(a) = \alpha_t(a)$. So assume this, and note that since $v^*v \in E(S)$, we have $\alpha_{v^*v}(a) = a$ by \cref{lemma: action properties}\cref{item: alpha_s inverse and idempotent}. Since $v \le s,t$, we have $sv^*v = v = tv^*v$, and so
\[
\alpha_s(a) = \alpha_s\big(\alpha_{v^*v}(a)\big) = \alpha_{sv^*v}(a) = \alpha_v(a) = \alpha_{tv^*v}(a) = \alpha_t\big(\alpha_{v^*v}(a)\big) = \alpha_t(a). \qedhere
\]
\end{proof}

\begin{remark} \label{rem: I_s1 properties} Suppose that $S$ has a unit $1$, and fix $s \in S$. For all $v \in S$ with $v \le s,1$, we have $v^*v = v \in E(S)$, and thus
\[
I_{s,1} = \clspan\{ a \in A : a \in A_v \text{ for some } v \in E(S) \text{ with } v \le s \}.
\]
For all $v \in E(S)$, we have $v \le s \iff v \le s^*$, and hence it follows from \cref{lemma: I_st properties} that
\[
I_{s,1} = I_{s^*,1} \subseteq A_{s^*s} \cap A_{ss^*}.
\]
\end{remark}

\Cref{rem: I_s1 properties} shows that $I_{s,1}$ is contained in the domains of both $\alpha_s$ and $\alpha_{s^*}$. We now show that $\alpha_s$ and $\alpha_{s^*}$ are the identity on $I_{s,1}$.

\begin{cor} \label{cor: alpha_s is alpha_s* on I_1s}
Let $\alpha\colon S \curvearrowright A$ be an action of a unital inverse semigroup $S$ on a C*\nobreakdash-algebra $A$. Fix $s \in S$. Then $\alpha_s(a) = a = \alpha_{s^*}(a)$ for all $a \in I_{s,1}$.
\end{cor}

\begin{proof}
Fix $a \in I_{s,1}$. Then $\alpha_s(a) = \alpha_1(a) = a$ by \cref{lemma: alpha_s is alpha_t on I_st}. It follows by \cref{lemma: action properties}\cref{item: alpha_s inverse and idempotent} that $\alpha_{s^*}(a) = \alpha_{s^*}\big(\alpha_s(a)\big) = \alpha_{s^*s}(a) = a$.
\end{proof}

The following lemma shows the relationship between the ideals $I_{st,1}$ and $I_{ts,1}$ for $s, t \in S$. An analogous result in the setting of Fell bundles is presented within the proof of \cite[Proposition~2.17]{BussMartinez2023}, but for convenience we reprove the result in our setting.

\begin{lemma} \label{lemma: I_1st iso to I_1ts}
Let $\alpha\colon S \curvearrowright A$ be an action of a unital inverse semigroup $S$ on a C*-algebra $A$. Fix $s, t \in S$. Then
\begin{enumerate}[label=(\alph*)]
\item \label{item: alpha_s* = alpha_t} $\alpha_{s^*}(a) = \alpha_t(a)$ for all $a \in I_{st,1}$; and
\item \label{item: alpha_s* image} $\alpha_{s^*}(I_{st,1}) = I_{ts,1}$.
\end{enumerate}
In other words, $\alpha_{s^*}\restr{I_{st,1}} = \alpha_t\restr{I_{st,1}}$ as isomorphisms from $I_{st,1}$ onto $I_{ts,1}$.
\end{lemma}

\begin{proof}
For part~\cref{item: alpha_s* = alpha_t}, fix $a \in I_{st,1}$. By \cref{cor: alpha_s is alpha_s* on I_1s}, $a = \alpha_{st}(a) \in \ran(\alpha_s) = \dom(\alpha_{s^*})$, and hence $\alpha_{s^*}(a) = \alpha_{s^*}\big(\alpha_{st}(a)\big) = \alpha_{s^*s}\big(\alpha_t(a)\big) = \alpha_t(a)$.

For part~\cref{item: alpha_s* image}, fix $a_e \in A_e$ for some $e \in S$ satisfying $e \le st, 1$. Then
\[
\alpha_{s^*}(a_e) = \alpha_{s^*}\big(\alpha_e(a_e)\big) \in \ran(\alpha_{s^*e}) = A_{s^*e(s^*e)^*} = A_{s^*es}.
\]
To see that $\alpha_{s^*}(a_e) \in I_{ts,1}$, it suffices to show that $s^*es \le ts, 1$, because then $A_{s^*es}$ is a generating ideal of $I_{ts,1}$. Since $e \in E(S)$, we have $s^*es \in E(S)$, and hence $s^*es \le 1$. Since $e \le st$, we have $e = est$, and hence
\[
(s^*es)(s^*es)^*ts = (s^*es)ts = s^*(est)s = s^*es.
\]
Thus $\alpha_{s^*}(a_e) \in A_{s^*es} \subseteq I_{ts,1}$. Since $\alpha_{s^*}$ is a bounded linear map and $A_e$ is a generating ideal of $I_{st,1}$, it follows that $\alpha_{s^*}(I_{st,1}) \subseteq I_{ts,1}$. For the reverse containment, note that by symmetry, we have $\alpha_{t^*}(I_{ts,1}) \subseteq I_{st,1}$, and also part~\cref{item: alpha_s* = alpha_t} implies that $\alpha_{t^*}(a) = \alpha_s(a)$ for all $a \in I_{ts,1}$, so $\alpha_{t^*}(I_{ts,1}) = \alpha_s(I_{ts,1})$. Therefore,
\[
I_{ts,1} = \alpha_{s^*s}(I_{ts,1}) = \alpha_{s^*}\big(\alpha_s(I_{ts,1})\big) = \alpha_{s^*}\big(\alpha_{t^*}(I_{ts,1})\big) \subseteq \alpha_{s^*}(I_{st,1}) \subseteq I_{ts,1}.
\]
Thus we have equality throughout, so $\alpha_{s^*}(I_{st,1}) = I_{ts,1}$.
\end{proof}

\subsection{Enveloping von Neumann algebras and local multiplier algebras}

In this subsection we introduce essential ideals and complements of ideals, as well as the enveloping von Neumann algebra $A''$ and the local multiplier algebra $\Mloc(A)$ of a C*-algebra $A$ (see \cite{AraMathieu2003}).

Let $I$ be an ideal of a C*-algebra $A$. We define the \emph{complement of $I$ in $A$} to be the set
\[
I^\perp \coloneqq \{ x \in A : xI = Ix = \{0\}\}.
\]
The complement of $I$ in $A$ is the largest ideal $K \trianglelefteq A$ such that $I \cap K = \{0\}$. We say that $I$ is \emph{essential} if $I \cap J \ne \{0\}$ for every nonzero ideal of $J$. Thus $I$ is essential if and only if $I^\perp = \{0\}$. Moreover, $I$ is essential if and only if, for $a \in A$, we have $aI = \{0\} \implies a = 0$ . It follows that if $I$ and $J$ are essential ideals of $A$ such that $I \subseteq J$, then $I$ is also an essential ideal of $J$. We say that $I$ is \emph{complemented in $A$} if the internal direct sum $I \oplus I^\perp$ is isomorphic to $A$. Regardless of whether $I$ is complemented in $A$, the ideal $I \oplus I^\perp$ is always essential.

We now define the enveloping von Neumann algebra of a C*-algebra, and we introduce some of its properties.

\begin{defn}
Let $A$ be a C*-algebra, and let $\pi_\rmU\colon A \to B(H_\rmU)$ be the universal representation of $A$ (that is, the direct sum of the GNS representations corresponding to all states of $A$, which is faithful). The \emph{enveloping von Neumann algebra} $A''$ of $A$ is the double commutant
\[
A'' \coloneqq \big(\pi_\rmU (A)\big)'' \subseteq B(H_\rmU).
\]
\end{defn}

Then $A''$ is a unital C*-algebra and is isomorphic as a Banach space to the double dual of $A$ (see, for instance, \cite[Proposition~3.7.8]{Pedersen2018}). Moreover, it follows from von Neumann's double commutant theorem \cite[Theorem~2.2.2]{Pedersen2018} that $A''$ is equal to the strong closure of $\pi_\rmU(A)$ in $B(H_\rmU)$, so $\pi_\rmU$ is an inclusion of $A$ in $A''$. The local multiplier algebra of $A$ is a certain subquotient of $A''$ that contains a copy of the multiplier algebra $\MM(A)$. The following results are tailored towards constructing particular inclusion maps of multiplier algebras of essential ideals of $A$, and studying the behaviour of the associated directed system of multiplier algebras.

We will implicitly use the following result without necessarily referencing \cref{lemma: inclusions of ideals}.

\begin{lemma} \label{lemma: inclusions of ideals}
Let $I$ be an ideal of a C*-algebra $A$. Then $I''$ is an ideal of $A''$, and the natural inclusion $I \hookrightarrow A\hookrightarrow A''$ extends to inclusions $I'' \hookrightarrow A''$ and $\MM(I) \hookrightarrow A''$.
\end{lemma}

\begin{proof}
The existence of the inclusion $I'' \hookrightarrow A''$ and that $I''$ is an ideal of $A''$ both follow from \cite[Corollary~3.7.9]{Pedersen2018}. Since $I$ is an essential ideal of $\MM(I)$ and the canonical inclusion $\pi_\rmU\colon I \hookrightarrow I'' = \MM(I'')$ is a nondegenerate injective homomorphism, \cite[Proposition~2.50]{RaeburnWilliams1998} implies that $\pi_\rmU$ extends uniquely to an inclusion $\MM(I) \hookrightarrow I''$. Composing this with the inclusion $I'' \hookrightarrow A''$ gives an inclusion $\MM(I) \hookrightarrow A''$.
\end{proof}

\begin{prop} \label{lemma: compatible maps between multiplier algebras}
Let $A$ be a C*-algebra and let $I$ be an essential ideal of $A$. Then there is a unique injective homomorphism $\phi\colon A \to \MM(I)$ that extends the natural inclusion $I \hookrightarrow \MM(I)$. The map $\phi$ extends uniquely to an injective homomorphism $\bar{\phi}\colon \MM(A) \to \MM(I)$. Furthermore, if $I$ and $J$ are essential ideals of $A$ such that $I \subseteq J$, then there is a unique injective homomorphism $\phi_{I,J}\colon \MM(J) \to \MM(I)$ that extends the canonical inclusion $I \hookrightarrow \MM(I)$.
\end{prop}

\begin{proof}
The existence and uniqueness of $\phi$ is due to \cite[Theorem~3.1.8]{Murphy1990}. (That $I$ is an essential ideal of $A$ is needed for injectivity of $\phi$.) The existence and uniqueness of $\bar{\phi}$ follows from \cite[Corollary~2.51]{RaeburnWilliams1998}. The final statement then follows immediately because $I$ is also an essential ideal of $J$.
\end{proof}

The maps $(\phi_{I,J})_{I,J}$ given in \cref{lemma: compatible maps between multiplier algebras} form a directed system, and the direct limit of this system is a distinguished C*-algebra $\Mloc(A)$.

\begin{defn}
The \emph{local multiplier algebra} of a C*-algebra $A$ is the C*-algebra
\[
\Mloc(A) \coloneqq \varinjlim \MM(J),
\]
where the limit is taken over all essential ideals $J \subseteq A$ with connecting maps $\phi_{I,J}\colon \MM(J) \to \MM(I)$ given in \cref{lemma: compatible maps between multiplier algebras}.
\end{defn}

\subsection{C*-algebras of actions of inverse semigroups}

In this subsection we recall various constructions of C*-algebras associated to an action $\alpha\colon S \curvearrowright A$ of a unital inverse semigroup $S$ on a C*-algebra $A$. We refer the reader to \cite{Sieben1997, Exel2008, BussExelMeyer2017, BussMartinez2023, KwasniewskiMeyer2021} and references therein for more details.

We begin by considering the collection of functions
\[
C_c(S,A)\coloneqq \{ f\colon S \to A \,:\, f(s) \in A_{ss^*} \text{ and } f(s) = 0 \text{ for all but finitely many } s \in S \}.
\]
(We use the notation $C_c$ to denote continuous compactly supported functions, but since the domain $S$ is discrete, these are actually finitely supported functions.) We equip $C_c(S,A)$ with pointwise addition, scalar multiplication, and multiplication and adjoint given by convolution and involution, respectively: for $f, g \in C_c(S,A)$ and $s \in S$,
\[
(f * g)(s) \coloneqq \sum_{\substack{t,u \in S, \\ s = tu}} \alpha_t\big(\alpha_{t^*}(f(t)) \, g(u) \big) \quad \text{ and } \quad f^*(s) \coloneqq \alpha_s\big(f(s^*)^*\big).
\]
For $s \in S$ and $a_s \in A_{ss^*}$, we define $a_s\delta_s\colon S \to A$ by
\[
a_s\delta_s(t) \coloneqq \begin{cases} a_s & \text{if } t = s \\ 0 & \text{if } t \ne s. \end{cases}
\]
Then for $s, t \in S$, $a_s \in A_{ss^*}$, and $b_t \in A_{tt^*}$, we have
\[
(a_s\delta_s) * (b_t\delta_t) = \alpha_s\big(\alpha_{s^*}(a_s) \, b_t\big) \delta_{st} \quad \text{ and } \quad (a_s\delta_s)^* = \alpha_{s^*}(a_s^*) \delta_{s^*}.
\]
Since each $f = \sum_{s \in S} f(s) \delta_s \in C_c(S,A)$ is finitely supported, the map $f \mapsto \oplus_{s \in S} \, f(s)$ is an isomorphism of $C_c(S, A)$ onto $\oplus_{s \in S} \, A_{ss^*}$.

Fix $s, t \in S$ with $s \le t$, and let $a \in A_{ss^*}$. Then $s^* \le t^*$, and so $A_{ss^*} \subseteq A_{tt^*}$. Thus $a\delta_s$ and $a\delta_t$ are both elements of $C_c(S,A)$. Following \cite[Lemma~4.5]{Sieben1997}, we would like to construct a quotient of $C_c(S,A)$ in which $a\delta_s$ is identified with $a\delta_t$. There are two algebraic ideals of $C_c(S,A)$ that are defined in the literature for this purpose:
\begin{align*}
\II_\alpha &\coloneqq \vecspan \{ a\delta_v - a\delta_r : v, r \in S, \, v \le r, \, a \in A_{vv^*} \} \quad \text{and} \\
\JJ_\alpha &\coloneqq \vecspan \{ a\delta_s - a\delta_t : s, t \in S, \, a \in I_{s^*,t^*} \}.
\end{align*}
For all $v, r \in S$ with $v \le r$, we have $v^* \le r^*$, so $A_{vv^*} \subseteq I_{v^*,r^*}$, and it follows that $\II_\alpha\subseteq \JJ_\alpha$.

\begin{defn} \label{defn: algebraic crossed product} Let $\alpha\colon S \curvearrowright A$ be an action of a unital inverse semigroup $S$ on a C*\nobreakdash-algebra $A$. We define the \emph{algebraic crossed product} $A \rtimesalg S$ to be the $*$-algebra $C_c(S,A) / \JJ_\alpha$. We define the \emph{maximal crossed product} $A \rtimesfull S$ to be the universal C*-algebra obtained by completing $A \rtimesalg S$ with respect to the norm given by
\begin{equation} \label{eqn: A rtimesalg S norm}
\norm{x} \coloneqq \sup \{ \norm{\pi(x)} : \pi \text{ is a $*$-representation of } A \rtimesalg S \}
\end{equation}
for $x \in A \rtimesalg S$.
\end{defn}

\begin{remark}
Let $\alpha\colon S \curvearrowright A$ be an action of a unital inverse semigroup $S$ on a C*-algebra $A$. It is possible to construct the maximal crossed product $A \rtimesfull S$ without any reference to the ideal $\JJ_\alpha$. For this, let $\QQ_\alpha \coloneqq C_c(S,A) / \II_\alpha$, and define
\[
\varsigma(x) \coloneqq \sup \{ \norm{\pi(x)} : \pi \text{ is a $*$-representation of } \QQ_\alpha \}
\]
for $x \in \QQ_\alpha$. Then $\varsigma$ is a seminorm on $\QQ_\alpha$, and $A \rtimesfull S$ is the Hausdorff completion of $\QQ_\alpha$ with respect to $\varsigma$. (See \cite[Definition~3.11]{Exel2011} for details.)
\end{remark}

We now show that the action $\alpha\colon S \curvearrowright A$ is implemented by conjugation in $A \rtimesalg S$.

\begin{lemma} \label{lemma: alpha implemented by conjugation}
Let $\alpha\colon S \curvearrowright A$ be an action of a unital inverse semigroup $S$ on a C*-algebra $A$. Fix $s \in S$, and let $\{e_\lambda\}_\lambda$ be an approximate identity for $A_{ss^*}$. Then for all $a \in A_{s^*s}$, we have
\[
a_s(a) \delta_{ss^*} = \lim_\lambda \big( e_\lambda \delta_s * a \delta_1 * \alpha_{s^*}(e_\lambda) \delta_{s^*} \big).
\]
\end{lemma}

\begin{proof}
Fix $a \in A_{s^*s}$. Then
\begin{align*}
\lim_\lambda \big( e_\lambda \delta_s * a \delta_1 * \alpha_{s^*}(e_\lambda) \delta_{s^*} \big) &= \lim_\lambda \big( e_\lambda \delta_s * a \, \alpha_{s^*}(e_\lambda) \, \delta_{s^*} \big) \\
&= \lim_\lambda \, \alpha_s\big( \alpha_{s^*}(e_\lambda) \, a \, \alpha_{s^*}(e_\lambda) \big) \delta_{ss^*} \\
&= \lim_\lambda \, e_\lambda \, \alpha_s(a) \, e_\lambda \, \delta_{ss^*} \\
&= \alpha_s(a) \delta_{ss^*}. \qedhere
\end{align*}
\end{proof}

We now construct the reduced and essential crossed products $A \rtimesred S$ and $A \rtimesess S$. We build these in parallel to highlight the similarities between the constructions. We first define some special elements of $I_{s,1}''$ and $\MM\big(I_{s,1} \oplus I_{s,1}^\perp\big)$ that will be used in these constructions.

\begin{defn} \label{defn: identity for expectations}
Let $\alpha\colon S \curvearrowright A$ be an action of a unital inverse semigroup $S$ on a C*-algebra $A$. Fix $s \in S$.
\begin{enumerate}[label=(\alph*)]
\item \label{item: identity for ER} We write $1_s$ for the identity of $I_{s,1}''$, where $I_{s,1}''$ is viewed as an ideal of $A''$ by \cref{lemma: inclusions of ideals}.
\item \label{item: identity for EL} We define $i_s \in \MM(I_{s,1}\oplus I_{s,1}^\perp)$ by $i_s(a\oplus b) \coloneqq a$ for $a \in I_{s,1}$ and $b \in I_{s,1}^\perp$.
\end{enumerate}
\end{defn}

\begin{remark} \label{rem: 1_s and i_s properties} \leavevmode
\begin{enumerate}[label=(\alph*), ref={\cref{rem: 1_s and i_s properties}(\alph*)}]
\item \label{item: 1_s is central} Fix $s \in S$, and observe that $1_s$ is in the centre of $A''$. To see this, fix $x \in A''$, and note that $1_s \, x, \, x \, 1_s \in I_{s,1}''$ because $I_{s,1}''$ is an ideal of $A''$. Thus, since $1_s$ is the identity of $I_{s,1}''$, we have $1_s \, x = (1_s \, x) 1_s = 1_s (x \, 1_s) = x \, 1_s$.
\item A straightforward calculation shows that any approximate identity $\{e_{\lambda, s}\}_\lambda$ for $I_{s,1}$ converges strictly to $i_s$ in $\MM\big(I_{s,1} \oplus I_{s,1}^\perp\big)$.
\end{enumerate}
\end{remark}

We now define a map $\ER\colon A \rtimesfull S \to A''$ that is referred to in \cite[Definition~3.1]{KwasniewskiMeyer2021} as a \emph{weak conditional expectation}. We will then quotient the maximal crossed product by the nucleus of this map $\ER$ to obtain the reduced crossed product $A \rtimesred S$ (see \cref{defn: reduced and essential crossed products}).

\begin{thm}[{see \cite[Lemma~4.5]{BussExelMeyer2017} or \cite[Proposition~3.16]{KwasniewskiMeyer2021}}] \label{thm: ER existence and properties}
Let $\alpha\colon S \curvearrowright A$ be an action of a unital inverse semigroup $S$ on a C*-algebra $A$. Let $1_s \in I_{s,1}'' \subseteq A''$ be as in \cref{defn: identity for expectations}\cref{item: identity for ER}. There exists a completely positive, linear, contractive map $\ER\colon A \rtimesfull S \to A''$ such that $\ER(a_s\delta_s) = a_s 1_s$ for all $s \in S$ and $a_s \in A_{ss^*}$.
\end{thm}

The proof of \cref{thm: ER existence and properties} can be found (with different notation) in \cite[Lemma~4.5]{BussExelMeyer2017} and \cite[Proposition~3.16]{KwasniewskiMeyer2021}, and it is also similar to the proof of \cref{thm: EL existence and properties} below, so we omit it.

\begin{thm}[{see \cite[Proposition~4.3]{KwasniewskiMeyer2021}}] \label{thm: EL existence and properties}
Let $\alpha\colon S \curvearrowright A$ be an action of a unital inverse semigroup $S$ on a C*-algebra $A$. Let $i_s \in \MM\big(I_{s,1} \oplus I_{s,1}^\perp\big) \subseteq \Mloc(A)$ be as in \cref{defn: identity for expectations}\cref{item: identity for EL}. There exists a completely positive, linear, contractive map $\EL\colon A \rtimesfull S \to \Mloc(A)$ such that $\EL(a_s \delta_s) = a_s i_s$ for all $s \in S$ and $a_s \in A_{ss^*}$.
\end{thm}

\begin{proof}
Let $\EL\colon A \rtimesfull S \to \Mloc(A)$ be the linear, completely positive, contractive map defined in \cite[Proposition~4.3]{KwasniewskiMeyer2021}. We will show that this map satisfies $\EL(a_s \delta_s) = a_s i_s$ for all $s \in S$ and $a_s \in A_{ss^*}$. In order to do this, we must first translate our setting to that of \cite{KwasniewskiMeyer2021}.

The action $\alpha$ induces an action of $S$ on $A$ by Hilbert $A$-modules, as defined in \cite[Definition~2.8]{KwasniewskiMeyer2021}, as follows. For $s \in S$, we set $\EE_s \coloneqq A_{ss^*}$. We equip $\EE_s$ with left and right actions of $A$, and left and right inner products with values in $A$, given by
\[
a \cdot \xi = a \xi, \quad \xi \cdot b = \alpha_s\big(\alpha_{s^*}(\xi) \, b\big), \quad {}_L\langle \xi, \, \eta \rangle = \xi\eta^*, \,\ \text{ and } \ \ \langle \xi, \, \eta \rangle_R = \alpha_{s^*}(\xi^*\eta)
\]
for $a, b \in A$ and $\xi, \eta \in \EE_s$. With this structure, $\EE_s$ is a Hilbert $A$-module; the ranges of the inner products are $r(\EE_s) \coloneqq A_{ss^*}$ and $s(\EE_s) \coloneqq A_{s^*s}$, respectively. Together with the bimodule isomorphisms $\mu_{s,t}\colon \EE_s \otimes_A \EE_t \to \EE_{st}$ given by $\mu_{s,t}(\xi \otimes \eta) = \alpha_s\big(\alpha_{s^*}(\xi) \, \eta\big)$, this gives an action of $S$ on $A$ by Hilbert $A$-modules. (See also \cite[Example~2.3]{BussExelMeyer2017}, where $\EE_s = A_{s^*s}$ is used instead.)

To continue translating our setting to that of \cite{KwasniewskiMeyer2021}, we now identify $\xi \coloneqq a_s\delta_s$ with $a_s \in A_{ss^*}$ (via the isomorphism $C_c(S,A) \to \oplus_{s \in S} \, A_{ss^*}$), and trace through the maps used to define $\EL(\xi)$ in the discussion preceding \cite[Proposition~4.3]{KwasniewskiMeyer2021}. Given $v \le s$, we have $\EE_v \subseteq \EE_s$ by \cref{lemma: action properties}\cref{item: alpha_s ideal containment}. As in \cite[Section~2.2]{KwasniewskiMeyer2021}, let $j_{s,v}\colon \EE_v \to \EE_s$ be the inclusion map. In \cite[Section~3.2]{KwasniewskiMeyer2021} it is noted that $j_{s,v}$ may be viewed as an isomorphism onto its range $r(\EE_v) \cdot \EE_s = \EE_s \cdot s(\EE_v)$, and in our setting this just amounts to the observation that $j_{s,v}\colon \EE_v \to \EE_v$ is an isomorphism because it is the identity map. By \cite[Equation~3.2]{KwasniewskiMeyer2021}, there is a unique Hilbert-bimodule isomorphism
\[
\vartheta_{1,s}\colon \EE_s \cdot I_{s,1} \to \EE_1 \cdot I_{s,1}
\]
that restricts to $j_{1,v} \circ j_{s,v}^{-1} = \id_{\EE_v}$ on $\EE_v$ for any $v \in E(S)$ satisfying $v \le s$. By \cref{cor: alpha_s is alpha_s* on I_1s}, $\alpha_s$ and $\alpha_{s^*}$ both restrict to the identity map on $I_{s,1}$, and it follows that $\EE_s \cdot I_{s,1} = I_{s,1} = \EE_1 \cdot I_{s,1}$. Thus $\vartheta_{1,s}$ is the identity map on $I_{s,1}$.

Fix $a_1 \in I_{s,1}$ and $a_2 \in I_{s,1}^\perp$, and write $a \coloneqq a_1 \oplus a_2$. By the discussion preceding \cite[Prop\-osition~4.3]{KwasniewskiMeyer2021}, we have
\begin{equation} \label{eqn: EL(xi)(a)}
\EL(\xi)(a) = i_s\big((\vartheta_{1,s} \oplus \id)(\xi \cdot a)\big) = i_s\big(\vartheta_{1,s}(\xi \cdot a_1) \oplus (\xi \cdot a_2)\big) = \vartheta_{1,s}(\xi \cdot a_1).
\end{equation}
Since $\alpha_{s^*}(\xi), a_1 \in A_{s^*s}$, we have
\[
\xi \cdot a_1 = \alpha_s\big(\alpha_{s^*}(\xi) \, a_1\big) = \alpha_{ss^*}(\xi) \, \alpha_s(a_1) = \xi \alpha_s(a_1).
\]
Since $a_1 \in I_{s,1}$, we have $\alpha_s(a_1) = a_1$ by \cref{cor: alpha_s is alpha_s* on I_1s}, and thus $\xi \cdot a_1 = \xi \alpha_s(a_1) = \xi a_1 \in I_{s,1}$. Since $\vartheta_{1,s}\colon I_{s,1} \to I_{s,1}$ is the identity map, \cref{eqn: EL(xi)(a)} implies that
\[
\EL(\xi)(a) = \vartheta_{1,s}(\xi \cdot a_1) = \vartheta_{1,s}(\xi a_1) = \xi a_1 = \xi \, i_s(a).
\]
It follows that $\EL(\xi) = \xi i_s$, as claimed.
\end{proof}

We now define the reduced and essential crossed products $A \rtimesred S$ and $A \rtimesess S$. Let $\alpha\colon S \curvearrowright A$ be an action of a unital inverse semigroup $S$ on a C*-algebra $A$. Let $\ER\colon A \rtimesfull S \to A''$ and $\EL\colon A \rtimesfull S \to \Mloc(A)$ be the completely positive linear contractions from \cref{thm: ER existence and properties,thm: EL existence and properties}, respectively. Set
\[
\NN_\ER \coloneqq \{a \in A \rtimesfull S : \ER(a^*a) = 0\} \quad \text{ and } \quad \NN_\EL \coloneqq \{a \in A \rtimesfull S : \EL(a^*a) = 0\}.
\]
We call $\NN_\ER$ the \emph{nucleus} of $\ER$ and $\NN_\EL$ the \emph{nucleus} of $\EL$. By \cite[Remark~4.5]{KwasniewskiMeyer2021}, $\NN_\ER$ and $\NN_\EL$ are ideals of $A \rtimesfull S$ such that $\NN_\ER \subseteq \NN_\EL$.

\begin{defn} \label{defn: reduced and essential crossed products}
Let $\alpha\colon S \curvearrowright A$ be an action of a unital inverse semigroup on a C*-algebra $A$. We define the \emph{reduced crossed product of $(A, S, \alpha)$} and the \emph{essential crossed product of $(A, S, \alpha)$} by
\[
A \rtimesred S \coloneqq (A \rtimesfull S) / \NN_\ER \quad \text{ and } \quad A \rtimesess S \coloneqq (A \rtimesfull S) / \NN_\EL,
\]
respectively. We write $\Lambda\colon A \rtimesred S \to (A \rtimesred S) / \NN_\EL \cong A \rtimesess S$ for the quotient map, and we call $\ker(\Lambda)$ the \emph{singular ideal} of $A \rtimesred S$.
\end{defn}

The map $A \owns a \mapsto a\delta_1 \in C_c(S,A)$ induces embeddings
\[
\iota_\red\colon A \to A \rtimesred S \quad \text{ and } \quad \iota_\ess\colon A \to A \rtimesess S
\]
(see \cite[Lemma~2.28]{BussMartinez2023} for the reduced case, and \cite[Corollary~4.12]{KwasniewskiMeyer2021} for the essential case). It follows that $\ker(\Lambda) \cap A = \{0\}$.

\begin{remark}
We can identify precisely when $\ER$ takes values in $A$ rather than in $A''$. Let $\alpha\colon S \curvearrowright A$ be an action of a unital inverse semigroup $S$ on a C*-algebra $A$. Following the naming convention established in \cite[Definitions~2.5 and 2.18]{KwasniewskiMeyer2021}, we say that $\alpha$ is \emph{closed} if $I_{s,1}$ is complemented in $A_{s^*s}$ for each $s \in S$. The map $\ER\colon A \rtimesfull S \to A''$ takes values in $A$ if and only if $\alpha$ is closed (see \cite[Proposition~3.20]{KwasniewskiMeyer2021} or \cite[Corollary~2.29]{BussMartinez2023}). In this case, $\ER = \EL$, and the quotient map $\Lambda\colon A \rtimesred S \to A \rtimesess S$ is injective (see \cite[Remark 4.6]{KwasniewskiMeyer2021}).

When $S = G$ is a group, $I_{g,1}^\perp = A = I_{1,1}$ for all $g \in G {\setminus} \{1\}$, and hence $I_{g,1}$ is trivially complemented in $A_{g^*g} = A$ for all $g \in G$. Thus every group action $G\curvearrowright A$ is closed, and we have $A\rtimesred G = A\rtimesess G$.

When $\alpha$ is an action of an inverse semigroup on a commutative C*-algebra $C_0(X)$ (as in \cref{sec: groupoid C*-algebras}), the ideals $I_{s,1}$ are the subalgebras $C_0\big(\bigcup_{v \le s, 1} O_{v^*v}\big)$, where $O_{v^*v} \subseteq X$ are the open sets such that $C_0(O_{v^*v}) = A_{v^*v}$. The ideal $I_{s,1}$ is complemented in $A_{ss^*}$ precisely when $\bigcup_{v \le s, 1} O_{v^*v}$ is closed in $O_{s^*s}$. In this case, the transformation groupoid (sometimes called the groupoid of germs) of the action $\alpha$ has a closed unit space; that is, the groupoid is Hausdorff.
\end{remark}

\begin{example}
Let $S \coloneqq \Z_2 \sqcup \{0\}$ be the inverse semigroup formed by adjoining a $0$ to the group $\Z_2=\{1,s\}$ with $2$ elements where $s^2 = 1$ and $1 \cdot 0 = s \cdot 0 = 0 = 0 \cdot s = 0 \cdot 1$. Let $A$ be a C*-algebra with a nontrivial proper ideal $J \subseteq A$. Let $\alpha\colon S \curvearrowright A$ be the trivial action, where $A_0 \coloneqq J$ and $A_1 = A_{ss^*} = A$. The only relation in the inverse semigroup is $0 \le 1,s$, and hence
\[
\JJ_\alpha = \vecspan\!\big\{ i\delta_1 - i\delta_0, \, j\delta_s - j\delta_0 : i, j \in J \big\}.
\]
A straightforward calculation shows that the map $\Psi\colon C_c(S,A) \to A \oplus (A / J)$ given by
\[
\Psi(a\delta_1 + b\delta_s + c\delta_0) \coloneqq (a + b + c, \, a - b + J)
\]
for $a, b \in $ and $c \in J$ is a surjective homomorphism. We have $\ker(\Psi) = \JJ_\alpha$, and hence $\Psi$ induces an isomorphism $\Phi\colon A \rtimesalg S = C_c(S,A) / \JJ_\alpha \to A \oplus (A / J)$. To find $A\rtimesfull S$ we must complete $A \rtimesalg S$ with respect to the norm given in \cref{eqn: A rtimesalg S norm}, but this is trivial, so we have
\[
A \rtimesfull S \cong A \rtimesred S \cong A \rtimesalg S \cong A \oplus (A/J).
\]

Now we will make an additional assumption that the ideal $J$ is essential in $A$ and show that in this case, the essential crossed product $A \rtimesess S$ differs from $A \rtimesred S$. We have $I_{1,1} = A_1 = A$ and $I_{s,1} = I_{0,1} = A_0 = J$. Recall that since $J$ and $A$ are essential ideals of $A$, $J^\perp = A^\perp = \{0\}$. Thus
\[
\MM\big(I_{1,1} \oplus I_{1,1}^\perp\big) = \MM(A) \quad \text{ and } \quad \MM\big(I_{0,1} \oplus I_{0,1}^\perp\big) = \MM\big(I_{s,1} \oplus I_{s,1}^\perp\big) = \MM(J),
\]
and so $i_1 = 1_{\MM(A)}$ and $i_s = i_0 = 1_{\MM(J)}$. Thus, for all $a, b \in A$ and $c \in J$,
\[
\EL\big(a\delta_1 + b\delta_s + c\delta_0 + \JJ_\alpha\big) = a i_1 + b i_s + c i_0 = a + b + c \in A \subseteq \Mloc(A).
\]
Therefore, $\EL$ is $A$-valued, and so it is a ``genuine'' conditional expectation even though $\ER$ does not always take values in $A$. A straightforward computation shows that
\[
\NN_\EL = \big\{ a\delta_1 + b\delta_s + c\delta_0 + \JJ_\alpha \in A \rtimesalg S : a + b + c = 0 \big\},
\]
and it follows that the map $\widetilde{\Psi}\colon A \rtimesess S = (A \rtimesfull S) / \NN_\EL \to A$ given by
\[
\widetilde{\Psi}\big(a\delta_1 + b\delta_s + c\delta_0 + \NN_\EL\big) \coloneqq a + b + c
\]
for $a, b \in A$ and $c \in J$ is an isomorphism.
\end{example}

\section{Preliminaries on Cuntz semigroups}
\label{sec: Cuntz semigroups}

\subsection{Preordered abelian monoids}

In this subsection we introduce the necessary background on preordered abelian monoids. (These are called \emph{positively ordered monoids} in \cite[Definition~1.1]{Wehrung1992}.) We write $\N$ for the set of nonnegative integers.

\begin{defn}
A \emph{preorder} is a reflexive and transitive binary relation $\le$ on some set. A \emph{monoid} is a semigroup with an identity element. If $x \le y$, then we also write $y \ge x$. We say that an abelian monoid $(M,+)$ is \emph{preordered} if it has a preorder $\le$ such that
\begin{enumerate}[label=(\alph*)]
\item $z \ge 0$ for all $z \in M$; and
\item for all $x, x', y, y' \in M$ with $x \le x'$ and $y \le y'$, we have $x + y \le x' + y'$.
\end{enumerate}
\end{defn}

\begin{defn} \label{defn: infinite monoid}
Let $M$ be a preordered abelian monoid.
\begin{enumerate}[label=(\alph*), ref={\cref{defn: infinite monoid}(\alph*)}]
\item \label{item: infinite element of monoid} An element $x \in M$ is \emph{infinite} if there is some $y \in M {\setminus} \{0\}$ such that $x + y \le x$.
\item \label{item: properly infinite element of monoid} An element $x \in M {\setminus} \{0\}$ is \emph{properly infinite} if $2x \le x$.
\item The monoid $M$ is \emph{purely infinite} if $M \ne \{0\}$ and every $x \in M {\setminus} \{0\}$ is properly infinite.
\item The monoid $M$ is \emph{almost unperforated} if, whenever $x, y \in M$ satisfy $(n+1)x \le ny$ for some $n \in \N$, we have $x \le y$.
\item The monoid $M$ has \emph{plain paradoxes} if, whenever $x \in M$ satisfies $(n+1)x \le nx$ for some $n \in \N$, we have $2x \le x$.
\end{enumerate}
\end{defn}

It is straightforward to show that if $M$ is almost unperforated then $M$ has plain paradoxes. In \cite{KMP2025} the authors note that having plain paradoxes is in fact \emph{strictly} weaker than being almost unperforated, and so we use the plain paradoxes condition as it is easier to check and sufficient for our results.

\begin{defn}[{see~\cite[Definition~3.1]{Wehrung1992}}] \label{defn: injective}
Let $M$ be a preordered abelian monoid. We say that $M$ is \emph{injective} if, whenever $A$ is a submonoid of a preordered abelian monoid $B$, every order-preserving homomorphism $A \to M$ extends to an order-preserving homomorphism $B \to M$.
\end{defn}

We remark that $M$ is injective in the sense of \cref{defn: injective} if it is injective as an object in the category of preordered abelian monoids. Wehrung proved in \cite[Example~3.10]{Wehrung1992} that the set of extended positive reals $[0,\infty]$ (with the usual addition and preorder) is an injective preordered abelian monoid. This result allows Tarski's theorem (regarding abelian monoids with the algebraic preorder) to be strengthened to apply to all preordered abelian monoids.

\begin{thm}[{Tarski's theorem; see \cite[Page~43]{Wehrung1992} and \cite[Theorem~A.5]{KMP2025}}] \label{thm: Tarski-Wehrung paradoxicality thm}
Let $M$ be a preordered abelian monoid, and fix $x \in M {\setminus} \{0\}$. The following are equivalent:
\begin{enumerate}[label=(\arabic*)]
\item $(n+1)x \not\le nx$ for all $n \in \N$; and
\item there is an order-preserving homomorphism $\nu\colon M \to [0,\infty]$ such that $\nu(x) = 1$.
\end{enumerate}
\end{thm}

\begin{proof}
Fix $x \in M {\setminus} \{0\}$. Suppose that $(n+1)x \not\le nx$ for all $n \in \N$. We claim that if $mx \le nx$ for some $m, n \le \N$, then $m \le n$. To see this, suppose for contradiction that $mx \le nx$ for some $m, n \in \N$ with $m > n$. Then $m = n + k$ for some $k > 0$, and so $nx \ge (n+k)x = nx + kx$, which implies that $kx \le 0$. Hence either $k \le 0$ or $x \le 0$; both are contradictions, so we have proved the claim. Now let $\langle x \rangle$ denote the submonoid of $M$ generated by $x$; that is, $\langle x \rangle \coloneqq \{nx : n \in \N\}$. Then the map $\nu_x\colon \langle x\rangle \to [0,\infty]$ given by $\nu_x(nx) \coloneqq n$ is an order-preserving homomorphism. Since $[0,\infty]$ is an injective preordered abelian monoid (see \cite[Example~3.10]{Wehrung1992}), $\nu_x$ extends to an order-preserving homomorphism $\nu\colon M \to [0,\infty]$ such that $\nu(x) = \nu_x(x) = 1$.

For the converse, suppose there exists an order-preserving homomorphism $\nu\colon M \to [0,\infty]$ such that $\nu(x) = 1$, and suppose for contradiction that $(n+1)x \le nx$ for some $n \in \N$. Then $\nu\big((n+1)x\big) = (n+1) \, \nu(x) = n+1 > n = n \, \nu(x) = \nu(nx)$, which contradicts our assumption that $\nu$ is order-preserving. Hence $(n+1)x \not\le nx$ for all $n \in \N$.
\end{proof}

\subsection{\texorpdfstring{$\Cu$}{Cu}-semigroups}

Before getting to the definition of the Cuntz semigroup of a C*-algebra, it will be useful to consider the more general notion of a $\Cu$-semigroup. We first define another relation on a preordered abelian monoid.

\begin{defn}[{see \cite[Definition~4.1]{GardellaPerera2024}}] \label{defn: swb ll}
Let $M$ be a preordered abelian monoid, and fix $x, y \in M$.
\begin{enumerate}[label=(\alph*), ref={\cref{defn: swb ll}(\alph*)}]
\item \label{item: swb ll} We say that $x$ is \emph{sequentially way below} $y$, and we write $x \ll y$, if, given any increasing sequence $\{d_n\}_{n \in \N} \subseteq M$ with supremum $d \in M$ satisfying $y \le d$, there is some $n \in \N$ such that $x \le d_n$.
\item We say that $x$ is \emph{compact} if $x \ll x$.
\end{enumerate}
\end{defn}

It is immediate that $x \ll y$ implies $x \le y$, and also that $0 \ll x$ for all $x \in M$. The next lemma is also relatively straightforward and will be useful in some of our calculations.

\begin{lemma} \label{lemma: ll le properties}
Let $M$ be a preordered abelian monoid, and fix $x, y, z \in M$. If $x \ll y \le z$ or $x \le y \ll z$, then $x \ll z$. Moreover, $\ll$ is transitive.
\end{lemma}

\begin{proof}
Let $\{d_n\}_{n \in \N}$ be a sequence in $M$ with supremum $d$ such that $z \le d$. Suppose that $x \ll y \le z$. Then $y \le z \le d$, so since $x \ll y$, there exists $n \in \N$ such that $x \le d_n$. Thus $x \ll z$. Suppose instead that $x \le y \ll z$. Then since $y \ll z$, there exists $n \in \N$ such that $y \le d_n$. So $x \le y \le d_n$, and thus $x \ll z$. Finally, if $x \ll y$ and $y \ll z$, then $x \le y$ and $y \le z$, so the preceding argument implies that $x \ll z$, and thus $\ll$ is transitive.
\end{proof}

\begin{defn}[{see \cite[Definition~4.5]{GardellaPerera2024}}] \label{defn: Cu-semigroup}
Let $M$ be a preordered abelian monoid with respect to a preorder $\le$. We say that $M$ is a \emph{$\Cu$-semigroup} if it satisfies the following axioms.
\begin{enumerate}[label=(O\arabic*)]
\item Every sequence in $M$ that is increasing with respect to $\le$ has a supremum.
\item \label{item: O2} For every $x \in M$, there is a sequence $\{x_n\}_{n \in \N}$ in $M$ such that $x_n \ll x_{n+1}$ for all $n \in \N$ and $\displaystyle x = \sup_{n \in \N} (x_n)$.
\item \label{item: O3} Let $x, x',y, y' \in M$. If $x \ll y$ and $x' \ll y'$, then $x + x' \ll y + y'$.
\item \label{item: O4} For all sequences $\{x_n\}_{n \in \N}$ and $\{y_n\}_{n \in \N}$ in $M$ that are increasing with respect to $\le$,
\[
\sup_{n \in \N} (x_n + y_n) = \sup_{n \in \N} (x_n) + \sup_{n \in \N} (y_n).
\]
\end{enumerate}
\end{defn}

\begin{example}
The injective preordered abelian monoid $[0,\infty]$ is a $\Cu$-semigroup, where ``sequentially way below'' is the usual $<$ relation.
\end{example}

\begin{remark} \label{rem: nonzero sequentially way below}
Let $M$ be a $\Cu$-semigroup. It follows from axiom~\cref{item: O2} of \cref{defn: Cu-semigroup} that for any $x \in M {\setminus} \{0\}$, there exists $z \in M {\setminus} \{0\}$ such that $z \ll x$.
\end{remark}

\begin{defn}[{see \cite[Definition~1.7]{Wehrung1992} and \cite[Definition~5.1]{GardellaPerera2024}}] \label{defn: PAM and Cu-semigroup ideals} \leavevmode
\begin{enumerate}[label=(\alph*)]
\item Let $M$ be a preordered abelian monoid with respect to a preorder $\le$. An \emph{ideal} of $M$ is a submonoid $I \subseteq M$ that is hereditary, in the sense that whenever $x \in M$ and $y \in I$ satisfy $x \le y$, we have $x \in I$.
\item Let $M$ be a $\Cu$-semigroup. An \emph{ideal} of $M$ is a hereditary submonoid $I \subseteq M$ that is closed under suprema of increasing sequences.
\end{enumerate}
\end{defn}

\begin{defn} \label{defn: Cu-semigroup action}
An \emph{action} $\beta\colon S \curvearrowright M$ of a unital inverse semigroup $S$ on a $\Cu$-semigroup $M$ consists of collections of ideals $\{M_{s^*s}\}_{s \in S}$ and $\Cu$-morphisms $\{ \beta_s\colon M_{s^*s} \to M_{ss^*}\}_{s \in S}$ such that
\begin{enumerate}[label=(\alph*), ref={\cref{defn: Cu-semigroup action}(\alph*)}]
\item for each $s \in S$, $\beta_s$ is a bijective $\Cu$-morphism;
\item \label{item: Cu-semigroup action composition} for all $s, t \in S$, $M_{(st)^*st} = \beta_t^{-1}(M_{s^*s} \cap M_{tt^*})$ and $\beta_{st} = \beta_s \circ \beta_t\restr{M_{(st)^*st}}$; and
\item $M_1 = M$.
\end{enumerate}
We say that a subset $I \subseteq M$ is \emph{invariant} under $\beta$ if, for all $s \in S$, $\beta_s(I \cap M_{s^*s}) \subseteq I$.
\end{defn}

\begin{remark} \label{rem: Cu-semigroup action properties}
It is straightforward to show that if $\beta\colon S \curvearrowright M$ is an action of a unital inverse semigroup on a $\Cu$\nobreakdash-semigroup $M$, then $\beta_e = \id_{M_e}$ for all $e \in E(S)$, and $\beta_{s^*} = \beta_s^{-1}$ for all $s \in S$.
\end{remark}

The following is a useful interpolation property that emerges from the $\Cu$-semigroup axioms, which we will use in some of our proofs.

\begin{lemma} \label{lemma: sequentially way-below interpolation}
Let $M$ be a $\Cu$-semigroup, with ideals $J_1, \dotsc, J_n \trianglelefteq M$. Suppose that $x \in M$ and $\big\{ y_i \in J_i : i \in \{1, \dotsc, n\} \big\}$ satisfy
\[
x \,\ll\, \sum_{i=1}^n y_i.
\]
Then there exist $\big\{ y_i' \in J_i : i \in \{1, \dotsc, n\} \big\}$ such that $y_i' \ll y_i$ for each $i \in \{1, \dotsc, n\}$, and
\[
x \,\ll\, \sum_{i=1}^n y_i' \,\ll\, \sum_{i=1}^n y_i.
\]
\end{lemma}

\begin{proof}
Suppose that $x \in M$ and $\big\{ y_i \in J_i : i \in \{1, \dotsc, n\} \big\}$ satisfy $x \,\ll\, \sum_{i=1}^n y_i$. By axiom~\cref{item: O2} of \cref{defn: Cu-semigroup}, for each $i \in \{1, \dotsc, n\}$, there is a sequence $\{y_{i,k}\}_{k \in \N} \subseteq J_i$ such that $y_{i,k} \ll y_{i,k+1}$ for all $k \in \N$ and $\sup_{k \in \N} (y_{i,k}) = y_i$. For each $k \in \N$, define $z_k \coloneqq \sum_{i=1}^n y_{i,k}$. Then by axiom~\cref{item: O3}, we have $z_k \ll z_{k+1}$ for all $k \in \N$, and by axiom~\cref{item: O4}, we have
\[
\sup_{k \in \N} (z_k) = \sum_{i=1}^n \sup_{k \in \N} (y_{i,k}) = \sum_{i=1}^n y_i.
\]
So since $x \ll \sum_{i=1}^n y_i = \sup_{k \in \N} (z_k)$, \cref{item: swb ll} implies that there exists $l \in \N$ such that $x \le z_l$. Let $y_i' \coloneqq y_{i,l+1} \in J_i$ for each $i \in \{1, \dotsc, n\}$. Then $\sum_{i=1}^n y_i' = z_{l+1}$. Since $x \le z_l \ll z_{l+1}$, \cref{lemma: ll le properties} implies that $x \ll z_{l+1} = \sum_{i=1}^n y_i'$. And since $z_{l+1} \ll z_{l+2} \le \sup_{k \in \N} (z_k)$, \cref{lemma: ll le properties} implies that $z_{l+1} \ll \sup_{k \in \N} (z_k) = \sum_{i=1}^n y_i$. Thus, since $\ll$ is transitive, we have
\[
x \ll \sum_{i=1}^n y_i' \ll \sum_{i=1}^n y_i. \qedhere
\]
\end{proof}

\begin{defn}
Let $M$ and $N$ be $\Cu$-semigroups. We call a map $f\colon M \to N$ a \emph{generalised $\Cu$-morphism} if it preserves addition, zero, the partial order $\le$, and suprema of increasing sequences with respect to $\le$. We call $f\colon M \to N$ a \emph{$\Cu$-morphism} if $f$ is a generalised $\Cu$-morphism that also preserves the sequentially way-below relation $\ll$.
\end{defn}

\begin{defn}
Let $M$ be a $\Cu$-semigroup. A \emph{functional} on $M$ is a generalised $\Cu$-morphism $f\colon M \to [0,\infty]$. We say that $f$ is
\begin{enumerate}[label=(\alph*)]
\item \emph{nontrivial} if $f(x) \in (0,\infty)$ for some $x \in M$;
\item \emph{finite} if $f(x) < \infty$ for all $x \in M$; and
\item \emph{faithful} if $f(x) \ne 0$ for all $x \in M {\setminus} \{0\}$.
\end{enumerate}
\end{defn}

Note that a functional on a $\Cu$-semigroup need not preserve the sequentially way-below relation $\ll$.

\subsection{The Cuntz semigroup}

In this paper we use the Cuntz semigroup to detect when C*\nobreakdash-algebras are stably finite or purely infinite. In this subsection we recall the key results on Cuntz semigroups that we will use, and we explain how the action of an inverse semigroup $S$ on a C*-algebra $A$ and the induced crossed products $A \rtimesred S$ and $A \rtimesess S$ interact with the Cuntz semigroup of $A$. For a survey of the broader Cuntz semigroup literature, see \cite{GardellaPerera2024}, and for some helpful notes on Cuntz semigroups, see \cite{Thiel2017}.

Recall that, given a C*-algebra $A$, there is an embedding $A \owns a \mapsto a \otimes e_{1,1} \in A \otimes \KK$. We will sometimes view $A$ as a subalgebra of $A \otimes \KK$ via this embedding.

\begin{defn}[{\cite[Definition~2.1]{GardellaPerera2024}}]
Let $A$ be a C*-algebra, and fix $a, b \in A^+$. We say that $a$ is \emph{Cuntz subequivalent} to $b$, and we write $a \preceq b$, if, for each $\varepsilon > 0$, there exists $d_\varepsilon \in A$ such that $\norm{d_\varepsilon \,b\, d_\varepsilon^* - a} < \varepsilon$. (Equivalently, $a \preceq b$ if and only if there is a net $\{r_\lambda\}_\lambda \subseteq A$ such that $r_\lambda \,b\, r_\lambda^* \to a$.) We say that $a$ and $b$ are \emph{Cuntz equivalent}, and we write $a \sim b$, if $a \preceq b$ and $b \preceq a$.

The \emph{Cuntz semigroup} of a C*-algebra $A$ is defined as a set to be
\[
\Cu(A) \coloneqq (A \otimes \KK)^+ / \sim,
\]
where $\sim$ is Cuntz equivalence. We denote the Cuntz-equivalence class of an element $a \in (A \otimes \KK)^+$ by $[a]_{\Cu(A)}$, or just $[a]_\Cu$ if there is no confusion. (We will later introduce another type of equivalence class and the subscript will distinguish the two.) There is a commutative addition on $\Cu(A)$ given by
\[
[a]_\Cu + [b]_\Cu \coloneqq \left[\begin{pmatrix} a & 0 \\ 0 & b \end{pmatrix}\right]_\Cu = [a \oplus b]_\Cu.
\]
There is also a natural partial order on $\Cu(A)$ given by $[a]_\Cu \le [b]_\Cu$ if $a \preceq b$.
\end{defn}

It follows immediately from the definition of Cuntz subequivalence that $[xax^*]_\Cu \le [a]_\Cu$ for all $a, x \in (A \otimes \KK)^+$. By \cite[Proposition~2.17]{Thiel2017}, if $a, b \in (A \otimes \KK)^+$ satisfy $a \le b$, then $a \preceq b$, so $[a]_\Cu \le [b]_\Cu$. It follows that for any increasing convergent sequence $\{a_n\}_{n \in \N} \subseteq (A \otimes \KK)^+$, the sequence $\big\{[a_n]_\Cu\big\}_{n \in \N} \subseteq \Cu(A)$ is increasing and has supremum $\sup_{n \in \N} \big([a_n]_\Cu\big) = \big[\!\sup_{n \in \N} (a_n)\big]_\Cu$ (see \cite[Lemma~2.57]{Thiel2017}). In fact, even when $\{a_n\}_{n \in \N} \subseteq (A \otimes \KK)^+$ doesn't have a supremum, the sequence $\big\{[a_n]_\Cu\big\}_{n \in \N} \subseteq \Cu(A)$ still always has a supremum (see \cite[Theorem~3.8]{GardellaPerera2024}).

For all $x \in (A \otimes \KK)^+$, we have $0 \preceq x$, and $x \sim 0 \iff x = 0$, so the semigroup $\Cu(A)$ has a minimal element $0 = [0]_\Cu$ with respect to the partial order $\le$. Moreover, the partial order $\le$ is compatible with addition, in the sense that if $[a_1]_\Cu \le [b_1]_\Cu$ and $[a_2]_\Cu \le [b_2]_\Cu$ for $a_1, a_2, b_1, b_2 \in (A \otimes \KK)^+$, then $[a_1]_\Cu + [a_2]_\Cu \le [b_1]_\Cu + [b_2]_\Cu$. Thus $\Cu(A)$ is a preordered abelian monoid.

The collections
\[
\Cu(A)^\ll \coloneqq \{ x \in \Cu(A) : x \ll y \text{ for some } y \in \Cu(A) \} \ \text{ and } \ \Cuc(A) \coloneqq \{ x \in \Cu(A) : x \ll x \}
\]
are submonoids of $\Cu(A)$ such that $\Cuc(A) \subseteq \Cu(A)^\ll$.

If $A$ is stably finite, then $\Cuc(A)$ is precisely the set of Cuntz-equivalence classes of projections in $(A \otimes \KK)^+$, and hence $\Cuc(A)$ is isomorphic to $V(A)$, the Murray--von Neumann semigroup of $A$ (see \cite[Corollary~3.3]{BrownCiuperca2009} and \cite[Remark~3.5]{GardellaPerera2024}).

\begin{thm}[{\cite[Theorem~4.6]{GardellaPerera2024}}] \label{thm: Cu(A) is Cu-semigroup}
Let $A$ be a C*-algebra. Then $\Cu(A)$ is a $\Cu$-semigroup.
\end{thm}

\begin{remark} \label{rem: [p]_Cu ll [p]_Cu}
Let $A$ be a C*-algebra. Use the functional calculus to write $\phi_t(a) = (a - t)_+$ for $a \in (A \otimes \KK)^+$ and $t \in [0,\infty)$, where $\phi_t\colon [0,\infty) \to [0,\infty)$ is given by $\phi_t(y) \coloneqq \max\{0, y - t\}$. Then by \cite[Remark~4.2]{GardellaPerera2024},
\begin{inequality} \label{ineq: [(a - t)_+]_Cu ll [a]_Cu}
[(a - t)_+]_\Cu \ll [a]_\Cu \ \text{ for all } a \in (A \otimes \KK)^+ \text{ and } t \in (0,\infty).
\end{inequality}
Suppose that $p \in (A \otimes \KK)^+$ is a projection. Then for all $t \in [0,1]$ and $y \in \sigma(p) \subseteq \{0,1\}$, we have $\phi_t(y) = (1 - t)y$, so $(p - t)_+ = \phi_t(p) = (1 - t)p$. Since $(1 - t)p$ is Cuntz equivalent to $p$ for all $t \in [0,1)$, it follows from \cref{ineq: [(a - t)_+]_Cu ll [a]_Cu} that $[p]_\Cu \ll [p]_\Cu$. In particular, if $A$ is unital, then the identity element $1_A \in A$ may be viewed as a projection in $(A \otimes \KK)^+$ via the embedding $A \owns a \mapsto a \otimes e_{1,1} \in A \otimes \KK$, and hence $[1_A]_\Cu \ll [1_A]_\Cu$, so $[1_A]_\Cu \in \Cuc(A) {\setminus} \{0\}$.
\end{remark}

\begin{defn}[{\cite[Definition~3.2 and Theorem~4.16]{KR2000} and \cite[Definitions~III.1.3.1 and V.2.1.5]{Blackadar2006}}] \label{defn: SF PI C*s}
Let $A$ be a unital C*-algebra.
\begin{enumerate}[label=(\alph*), ref={\cref{defn: SF PI C*s}(\alph*)}]
\item \label{item: infinite element of C*-algebra} An element $a \in (A \otimes \KK)^+$ is \emph{infinite} if $[a]_\Cu + [b]_\Cu \le [a]_\Cu$ for some $b \in (A \otimes \KK)^+ {\setminus} \{0\}$. Equivalently, $a \in (A \otimes \KK)^+$ is infinite if $[a]_\Cu$ is an infinite element of $\Cu(A)$, in the sense of \cref{item: infinite element of monoid}.
\item An element $a \in (A \otimes \KK)^+ {\setminus} \{0\}$ is \emph{properly infinite} if $2[a]_\Cu \le [a]_\Cu$. Equivalently, $a \in (A \otimes \KK)^+ {\setminus} \{0\}$ is properly infinite if $[a]_\Cu$ is a properly infinite element of $\Cu(A)$, in the sense of \cref{item: properly infinite element of monoid}.
\item An element $a \in (A \otimes \KK)^+$ is \emph{finite} if it is not infinite.
\item The C*-algebra $A$ is \emph{finite} if $1_A$ is finite.
\item \label{item: stably finite C*} The C*-algebra $A$ is \emph{stably finite} if the matrix algebras $M_n(A)$ are finite for all $n \in \N$. Equivalently, $A$ is stably finite if every projection in $A \otimes \KK$ is finite.
\item The C*-algebra $A$ is \emph{purely infinite} if every element of $A^+ {\setminus} \{0\}$ is properly infinite.
\end{enumerate}
\end{defn}

The next result shows that isomorphisms between C*-algebras induce bijective $\Cu$-morphisms between their Cuntz semigroups. Note that an injective homomorphism between C*-algebras induces a $\Cu$-morphism with trivial kernel, but a $\Cu$-morphism with trivial kernel need not be injective.

\begin{thm} \label{thm: homomorphism induces Cu-morphism}
Let $\phi\colon A \to B$ be a homomorphism of C*-algebras and let $\phi \otimes \id$ be the standard amplification of $\phi$ to $A \otimes \KK$. There is an induced $\Cu$-morphism $\hat{\phi}\colon \Cu(A) \to \Cu(B)$ such that $\hat{\phi}\big([x]_{\Cu(A)}\big) = [(\phi \otimes \id)(x)]_{\Cu(B)}$ for all $x \in (A \otimes \KK)^+$. We have $\hat{\phi}\big(\Cuc(A)\big) \subseteq \Cuc(B)$, and if $\phi$ is surjective, then $\hat{\phi}$ is surjective. Moreover, if $\phi$ is injective, then $\ker(\hat{\phi}) = \{[0]_{\Cu(A)}\}$, and if $\phi$ is bijective, then $\hat{\phi}$ is bijective.
\end{thm}

\begin{proof}
By \cite[Theorem~4.35]{AraPereraToms2011}, the homomorphism $\phi\colon A \to B$ induces a $\Cu$-morphism $\hat{\phi} = \Cu(\phi)\colon \Cu(A) \to \Cu(B)$ such that $\hat{\phi}\big([x]_{\Cu(A)}\big) = [(\phi \otimes \id)(x)]_{\Cu(B)}$ for all $x \in (A \otimes \KK)^+$. To see that $\hat{\phi}\big(\Cuc(A)\big) \subseteq \Cuc(B)$, fix $x \in \Cuc(A)$. Then $x \ll x$, and since $\Cu$-morphisms preserve the sequentially way-below relation $\ll$, we have $\hat{\phi}(x) \ll \hat{\phi}(x)$. Thus $\hat{\phi}(x) \in \Cuc(B)$, and so $\hat{\phi}\big(\Cuc(A)\big) \subseteq \Cuc(B)$. If $\phi\colon A \to B$ is surjective, then so is $\phi \otimes \id\colon A \otimes \KK \to B \otimes \KK$, and since $\phi \otimes \id$ is a homomorphism, it follows immediately that $\hat{\phi}$ is surjective. Suppose now that $\phi\colon A \to B$ is injective. Then $\phi \otimes \id$ is injective by \cite[Proposition~B.13]{RaeburnWilliams1998}, and so for all $x \in (A \otimes \KK)^+$, we have
\[
\hat{\phi}([x]_{\Cu(A)}) = [(\phi \otimes \id)(x)]_{\Cu(B)} = 0 \iff (\phi \otimes \id)(x) = 0 \iff x = 0 \iff [x]_{\Cu(A)} = 0.
\]
Thus $\ker(\hat{\phi}) = \{[0]_{\Cu(A)}\}$.

Finally, suppose that $\phi\colon A \to B$ is bijective. Then $\hat{\phi}$ is surjective. To see that $\hat{\phi}$ is injective, suppose that $\hat{\phi}([x]_{\Cu(A)}) = \hat{\phi}([y]_{\Cu(A)})$ for some $x, y \in (A \otimes \KK)^+$. Then $(\phi \otimes \id)(x) \sim (\phi \otimes \id)(y)$ in $(B \otimes \KK)^+$. Since $\phi \otimes \id\colon A \otimes \KK \to B \otimes \KK$ is an isomorphism by \cite[Proposition~B.13]{RaeburnWilliams1998}, it follows using a routine argument that $x \sim y$ in $(A \otimes \KK)^+$. Thus $[x]_{\Cu(A)} = [y]_{\Cu(A)}$, and $\hat{\phi}$ is injective.
\end{proof}

We now explain how an action $\alpha\colon S \curvearrowright A$ of a unital inverse semigroup $S$ on a C*-algebra $A$ induces an action $\hat{\alpha}$ of $S$ on $\Cu(A)$, as defined in \cref{defn: Cu-semigroup action}. Given any ideal $J$ of $A$, the inclusion map $\iota\colon J \hookrightarrow A$ induces an injective $\Cu$-morphism $\hat{\iota}\colon \Cu(J) \to \Cu(A)$, and its image $\{ [x]_\Cu : x \in (J \otimes \KK)^+ \}$ is an ideal of $\Cu(A)$ (see \cref{defn: PAM and Cu-semigroup ideals,thm: homomorphism induces Cu-morphism} and \cite[Lemma~5.2]{GardellaPerera2024}). \Cref{thm: homomorphism induces Cu-morphism} implies that for each $s \in S$, the isomorphism $\alpha_s\colon A_{s^*s} \to A_{ss^*}$ induces a bijective $\Cu$-morphism
\[
\hat{\alpha}_s\colon \Cu(A_{s^*s}) \to \Cu(A_{ss^*}) \label{page: alpha hat s}
\]
between the ideals $\Cu(A_{s^*s})$ and $\Cu(A_{ss^*})$ of $\Cu(A)$. We have $\Cu(A_1) = \Cu(A)$, and routine calculations show that for any $s, t \in S$, we have $\Cu(A_{(st)^*(st)}) = \Cu\big(\alpha_t^{-1}(A_{s^*s} \cap A_{tt^*})\big) = \hat{\alpha}_t^{-1}\big(\Cu(A_{s^*s}) \cap \Cu(A_{tt^*})\big)$ and $\hat{\alpha}_{st} = \hat{\alpha}_s \circ \hat{\alpha}_t\restr{\Cu(A_{(st)^*(st)})}$. Thus $\hat{\alpha}$ is an action of $S$ on $\Cu(A)$.

\begin{defn}
Let $\alpha\colon S \curvearrowright A$ be an action of an inverse semigroup $S$ on a C*-algebra $A$, and let $M$ be a $\Cu$-semigroup. We say that a generalised $\Cu$-morphism $f\colon \Cu(A) \to M$ is \emph{invariant} if, for all $x \in \Cu(A_{s^*s})$ and all $s \in S$, we have $f(x) = f(\hat{\alpha}_s(x))$.
\end{defn}

\begin{lemma} \label{lemma: Cu-embedding in crossed product is invariant}
Let $\alpha\colon S \curvearrowright A$ be an action of a unital inverse semigroup $S$ on a C*-algebra $A$. Let $\iota_\red\colon A \to A \rtimesred S$ and $\iota_\ess\colon A \to A \rtimesess S$ be the canonical embeddings $a \mapsto a\delta_1$. Then the induced $\Cu$-morphisms
\[
\hat{\iota}_\red\colon \Cu(A) \to \Cu(A \rtimesred S) \quad \text{ and } \quad \hat{\iota}_\ess\colon \Cu(A) \to \Cu(A \rtimesess S)
\]
are invariant and have trivial kernel.
\end{lemma}

\begin{proof}
The proofs for $\iota_\red$ for $\iota_\ess$ are identical, so let $\iota$ be one of them. Since $\iota$ is injective, $\hat{\iota}$ is a $\Cu$-morphism with trivial kernel by \cref{thm: homomorphism induces Cu-morphism}. We claim that $(\iota \otimes \id)\big(\alpha_s \otimes \id(a)\big)$ is Cuntz equivalent to $(\iota \otimes \id)(a)$ for all $s \in S$ and $a \in (A_{s^*s} \otimes \KK)^+$. This will give the lemma as follows: fix $s \in S$ and $x \in \Cu(A_{s^*s})$; say $x = [a]_\Cu$ for some $a \in (A_{s^*s} \otimes \KK)^+$. Recall that $\hat{\alpha}_s(x) = \hat{\alpha}_s([a]_\Cu) = [(\alpha_s \otimes \id)(a)]_\Cu$. Then
\[
\hat{\iota}(\hat{\alpha}_s(x)) = [(\iota \otimes \id)\big((\alpha_s \otimes \id)(a)\big)]_\Cu = [(\iota \otimes \id)(a)]_\Cu = \hat{\iota}(x),
\]
using the claim.

To prove the claim, fix $s \in S$, $b \in A_{s^*s}$, and $c \in \KK$. Let $\{e_{s,\lambda}\}_\lambda$ be an approximate identity for $A_{ss^*}$, and let $\{f_\lambda\}_\lambda$ be an approximate identity for $\KK$. Then
\begin{align*}
\lim_\lambda \,(e_{s,\lambda}\delta_s \otimes f_\lambda) \big((\iota \otimes \id)(b \otimes c)\big) (e_{s,\lambda}\delta_s \otimes f_\lambda)^* &= \lim_\lambda \,(e_{s,\lambda}\delta_s \otimes f_\lambda)(b\delta_1 \otimes c)(e_{s,\lambda}\delta_s \otimes f_\lambda)^* \\
&= \lim_\lambda \big(\big(e_{s,\lambda}\delta_s(b\delta_1)\alpha_{s^*}(e_{s,\lambda}^*)\delta_{s^*}\big) \otimes f_\lambda c f_\lambda^*\big) \\
&= \lim_\lambda \big(e_{s,\lambda}\alpha_s(b)e_{s,\lambda}\delta_{ss^*} \otimes f_\lambda c f_\lambda^*\big) \\
&= \alpha_s(b) \delta_{ss^*} \otimes c \\
&= \alpha_s(b) \delta_1 \otimes c \qquad \text{ (since $ss^* \le 1$)} \\
&= (\iota \otimes \id)\big(\alpha_s(b) \otimes c\big) \\
&= (\iota \otimes \id)\big((\alpha_s \otimes \id)(b \otimes c)\big).
\end{align*}
It follows that $(\iota \otimes \id)\big((\alpha_s \otimes \id)(a)\big)$ is Cuntz subequivalent to $(\iota \otimes \id)(a)$ for all $s \in S$ and $a \in (A_{s^*s} \otimes \KK)^+$. A similar calculation shows that $(\iota \otimes \id)(a) = (\iota \otimes \id)\big((\alpha_{s^*s} \otimes \id)(a)\big)$ is Cuntz subequivalent to $(\iota \otimes \id)\big((\alpha_s \otimes \id)(a)\big)$. Thus $(\iota \otimes \id)(a)$ and $(\iota \otimes \id)\big((\alpha_s \otimes \id)(a)\big)$ are Cuntz equivalent.
\end{proof}

Let $\alpha\colon S \curvearrowright A$ be an action of an inverse semigroup $S$ on a C*-algebra $A$. Every functional $\beta$ on $\Cu(A)$ induces a quasitrace $\tau_\beta$ on $A$: here we give the relevant background, and show that if $\beta$ is invariant, then so is $\tau_\beta$. In particular, if $A$ is unital and exact, then every invariant, faithful, and finite functional on $\Cu(A)$ induces an invariant and faithful trace on $A$ by \cite[Theorem~5.11]{Haagerup2014}.

\begin{defn}[{see \cite[Definition~6.7]{Thiel2017}}]
Let $A$ be a C*-algebra. A \emph{$1$-quasitrace} on $A$ is a function $\tau\colon A^+ \to [0,\infty]$ such that
\begin{enumerate}[label=(\alph*)]
\item $\tau(ra) = r\tau(a)$ for all $a \in A^+$ and $r \in [0,\infty)$;
\item if $a, b \in A^+$ commute, then $\tau(a + b) = \tau(a) + \tau(b)$; and
\item $\tau(a^*a) = \tau(aa^*) \ge 0$ for all $a \in A$.
\end{enumerate}
For $n \ge 2$, an \emph{$n$-quasitrace} on $A$ is a $1$-quasitrace $\tau$ on $A$ that induces a $1$-quasitrace $\tilde{\tau}$ on $A \otimes M_n(\C)$ such that $\tilde{\tau}(a \otimes e_{1,1}) = \tau(a)$ for all $a \in A^+$.
\end{defn}

Every $2$-quasitrace on $A$ is an $n$-quasitrace for all $n \in \N$ by \cite[Remark~2.27(viii)]{BlanchardKirchberg2004}. If $\tau$ is a $2$-quasitrace, then we call $\tau$ a \emph{quasitrace}, and $\tau$ extends to $A \otimes \KK$. Quasitraces on $A$ are in one-to-one correspondence with quasitraces on $A \otimes \KK$.

We say that a quasitrace $\tau$ on $A$ is \emph{lower-semicontinuous} if, for every $t \in (0,\infty)$, the set $\tau^{-1}\big((t,\infty]\big)$ is open, and we say that $\tau$ is \emph{finite} if $\sup\{ \tau(a) : a \in A_+, \, \norm{a} = 1 \} < \infty$. If $A$ is unital, then we say that a quasitrace $\tau$ on $A$ is \emph{normalised} if $\tau(1_A) = 1$.

By \cite[Proposition~4.2]{ERS2011}, every functional $\beta$ on $\Cu(A)$ gives rise to a quasitrace $\tau_\beta$ on $A \otimes \KK$ and vice versa. We provide some of the details in \cref{prop: quasitraces and functionals} below. We also connect properties of $\beta$ with properties of $\tau_\beta$.

\begin{prop} \label{prop: quasitraces and functionals}
Let $A$ be a C*-algebra and let $\beta\colon \Cu(A) \to [0,\infty]$ be a functional. The function $\tau_\beta\colon (A \otimes \KK)^+ \to [0,\infty]$ given by
\[
\tau_\beta(a)\coloneqq \int_0^\infty \beta\big([(a-t)_+]_\Cu\big) \, \d t
\]
is a lower-semicontinuous quasitrace on $A \otimes \KK$. The map $\beta \mapsto \tau_\beta$ is a bijection between functionals on $\Cu(A)$ and lower-semicontinuous quasitraces on $A \otimes \KK$, with inverse $\tau \mapsto \beta_\tau$, where
\[
\beta_\tau([b]_\Cu) \coloneqq \sup_{n \in \N} \big(\tau\big(b^{\frac{1}{n}}\big)\big)
\]
for $b \in (A \otimes \KK)^+$. Furthermore,
\begin{enumerate}[label=(\alph*)]
\item \label{item: qtf invariance} if $\alpha\colon S \curvearrowright A$ is an action of an inverse semigroup $S$ on a unital C*-algebra $A$, then $\beta$ is invariant if and only if $\tau_\beta$ is invariant;
\item \label{item: qtf faithfulness} if $\beta$ is faithful, then $\tau_\beta$ is faithful;
\item \label{item: qtf finiteness} if $\beta$ is finite, then $\tau_\beta$ is finite; and
\item \label{item: qtf projections} $\tau_\beta(p) = \beta([p]_\Cu)$ for every projection $p \in (A \otimes \KK)^+$.
\end{enumerate}
\end{prop}

\begin{proof}
That $\tau_\beta$ is a lower-semicontinuous quasitrace on $A \otimes \KK$ follows from \cite[Proposition~4.2]{ERS2011}, as does the statement about $\tau \mapsto \beta_\tau$ being the inverse of $\beta \mapsto \tau_\beta$. For the rest of the proof, it is helpful to use the functional calculus to write $\phi_t(a) = (a - t)_+$ for $a \in (A \otimes \KK)^+$ and $t \in [0,\infty)$, where $\phi_t\colon [0,\infty) \to [0,\infty)$ is given by $\phi_t(y) \coloneqq \max\{0, y - t\}$. Then
\[
\tau_\beta(a) = \int_0^\infty \beta\big([\phi_t(a)]_\Cu\big) \, \d t \qquad \text{ for all } a \in (A \otimes \KK)^+.
\]
For part~\cref{item: qtf invariance}, first note that if $\tau$ is invariant, then it is clear from the formula for $\beta_\tau$ that it is also invariant. Now suppose that $\beta$ is invariant. Fix $s \in S$, and define $\tilde{\alpha}_s \coloneqq \alpha_s \otimes \id_\KK$. Let $a \in (A_{s^*s} \otimes \KK)^+$. Since $\tilde{\alpha}_s$ is a homomorphism, we have $\tilde{\alpha}_s\big(\phi_t(a)\big) = \phi_t\big(\tilde{\alpha}_s(a)\big)$. Thus $\tau_\beta$ is invariant, because
\begin{align*}
\tau_\beta\big(\tilde{\alpha}_s(a)\big) &= \int_0^\infty \beta\big(\big[\phi_t\big(\tilde{\alpha}_s(a)\big)\big]_\Cu\big) \, \d t = \int_0^\infty \beta\big(\big[\tilde{\alpha}_s\big(\phi_t(a)\big)\big]_\Cu\big) \, \d t \\
&= \int_0^\infty \beta\big(\hat{\alpha}_s\big([\phi_t(a)]_\Cu\big)\big) \, \d t = \int_0^\infty \beta\big([\phi_t(a)]_\Cu\big) \, \d t = \tau_\beta(a).
\end{align*}

For part~\cref{item: qtf faithfulness}, suppose that $\beta$ is faithful. Fix $a \in (A \otimes \KK)^+ {\setminus} \{0\}$. Then there exists $\varepsilon > 0$ such that $\phi_\varepsilon(a) > 0$. Since $\beta$ is faithful and order-preserving, $\beta\big([\phi_\varepsilon(a)]_\Cu\big) > 0$. Since $t \mapsto \phi_t(a)$ is a decreasing function,
\[
\tau_\beta(a) = \int_0^\infty \beta\big([\phi_t(a)]_\Cu\big) \, \d t \ge \int_0^\varepsilon \beta\big([\phi_t(a)]_\Cu\big) \, \d t > \int_0^\varepsilon \beta\big([\phi_\varepsilon(a)]_\Cu\big) \, \d t = \beta\big([\phi_\varepsilon(a)]_\Cu\big) \cdot \varepsilon > 0,
\]
which implies that $\tau_\beta$ is faithful.

For part~\cref{item: qtf finiteness}, suppose that $\beta$ is finite, and fix $a \in (A \otimes \KK)^+$. For all $t \in [0,\norm{a}]$, we have
\[
a = \phi_0(a) \ge\phi_t(a) \ge \phi_{\norm{a}}(a) = 0,
\]
and it follows by \cite[Proposition~2.17]{Thiel2017} that $[a]_\Cu \ge [\phi_t(a)]_\Cu \ge 0$. Since $\beta$ is order-preserving, we have $\beta\big([a]_\Cu\big) \ge \beta\big([\phi_t(a)]_\Cu\big) \ge 0$ for all $t \in (0,\norm{a})$. Since $\beta$ is finite, we have $\beta\big([a]_\Cu\big) < \infty$, and hence
\[
\tau_\beta(a) = \int_0^\infty \beta\big([\phi_t(a)]_\Cu\big) \, \d t = \int_0^{\norm{a}} \beta\big([\phi_t(a)]_\Cu\big) \, \d t \le \norm{a} \, \beta\big([a]_\Cu\big) < \infty.
\]

For part~\cref{item: qtf projections}, fix a projection $p \in (A \otimes \KK)^+$. Then for all $t \in [0,1]$ and $y \in \sigma(p) \subseteq \{0,1\}$, we have $\phi_t(y) = \max\{0, y - t\} = (1 - t)y$, and so $\phi_t(p) = (1 - t)p$. Since $(1 - t)p$ is Cuntz equivalent to $p$ for $t \in [0,1)$, we have
\[
\tau_\beta(p) = \int_0^\infty \beta\big([\phi_t(p)]_\Cu\big) \, \d t = \int_0^1 \beta\big([\phi_t(p)]_\Cu\big) \, \d t = \int_0^1 \beta([p]_\Cu) \, \d t = \beta\big([p]_\Cu\big). \qedhere
\]
\end{proof}

\section{The dynamical Cuntz semigroup}
\label{sec: dynamical Cuntz semigroup}

Let $\alpha\colon S \curvearrowright A$ be an action of a unital inverse semigroup $S$ on a C*-algebra $A$. In this section we consider how the action $\alpha$ interacts with the Cuntz semigroup of $A$ via relations we denote by $\precalpha$ and $\pathalpha$. We then construct the \emph{dynamical Cuntz semigroup} of the action $\alpha$, which is similar to the analogous construction for actions of groups presented by Bosa--Perera--Wu--Zacharias in \cite{BPWZ2025}. The key idea is to ``complete'' a reflexive binary relation into a preorder by adding the missing transitive pairs. Our approach is based on \cite[Definition~4.3]{Ma2022} and \cite{Rainone2017}, where a K-theoretic approach is used to construct the semigroup.

We show that order-preserving homomorphisms on the dynamical Cuntz semigroup induce functionals on the Cuntz semigroup that ``see'' non-paradoxical elements. We start by naively defining a new order relation $\precalpha$, called \emph{$\alpha$-below}, which is induced by the action $\alpha$. We observe that $\precalpha$ is reflexive and behaves well with respect to addition and functionals. However, $\precalpha$ will not necessarily be transitive unless $\Cu(A)$ has certain refinement properties. We force transitivity by completing $\precalpha$ to obtain a transitive preorder $\pathalpha$, called \emph{$\alpha$-path-below}. (For a formal discussion about this completion, see \cite[Section~2.1]{BPWZ2025}.)

The following definition is inspired by \cite[Definition~4.3]{Ma2022} and \cite[Definition~4.16]{KMP2025}. We make this connection explicit in \cref{lemma: theta^Lsc-below characterisation} when we apply this construction to groupoid C*-algebras.

\begin{defn} \label{defn: alpha-below}
Let $\alpha\colon S \curvearrowright A$ be an action of a unital inverse semigroup $S$ on a C*-algebra $A$. For $x, y \in \Cu(A)$, we say that $x$ is \emph{$\alpha$-below} $y$, and we write $x \precalpha y$, if, for all $f \in \Cu(A)$ such that $f \ll x$, there exist $\{s_i\}_{i=1}^m \subseteq S$ and $\big\{ u_i \in \Cu(A_{s_i^*s_i}) : i \in \{1, \dotsc, m\} \big\}$ such that
\[
f \ll \sum_{i=1}^m u_i \quad \text{ and } \quad \sum_{i=1}^m \hat{\alpha}_{s_i}(u_i) \ll y.
\]
\end{defn}

The following alternative characterisation of the $\alpha$-below relation is a noncommutative analogue of \cite[Remark~4.17 and Proposition~4.19]{KMP2025} and will be useful in many of our calculations.

\begin{lemma} \label{lemma: alpha-below characterisation}
Let $\alpha\colon S \curvearrowright A$ be an action of a unital inverse semigroup $S$ on a C*-algebra $A$. Fix $x, y \in \Cu(A)$. Then $x \precalpha y$ if and only if, for all $g \in \Cu(A)$ such that $g \ll x$, there exist $\{s_i\}_{i=1}^m \subseteq S$ and $\big\{ v_i \in \Cu(A_{s_i^*s_i}) : i \in \{1, \dotsc, m\} \big\}$ such that
\[
g \le \sum_{i=1}^m v_i \quad \text{ and } \quad \sum_{i=1}^m \hat{\alpha}_{s_i}(v_i) \le y.
\]
\end{lemma}

\begin{proof}
If $x \precalpha y$, then the result follows immediately from \cref{defn: alpha-below} because $a \ll b \implies a \le b$ for all $a, b \in \Cu(A)$. For the reverse implication, we assume that the hypothesis holds, and then show that $x \precalpha y$. Fix $f \in \Cu(A)$ such that $f \ll x$. Then by \cref{lemma: sequentially way-below interpolation}, there exists $g \in \Cu(A)$ such that $f \ll g \ll x$. Thus, by our assumption, there exist $\{s_i\}_{i=1}^m \subseteq S$ and $\big\{ v_i \in \Cu(A_{s_i^*s_i}) : i \in \{1, \dotsc, m\} \big\}$ such that
\[
g \le \sum_{i=1}^m v_i \quad \text{ and } \quad \sum_{i=1}^m \hat{\alpha}_{s_i}(v_i) \le y.
\]
Since $f \ll g$, \cref{lemma: ll le properties} implies that
\[
f \ll \sum_{i=1}^m v_i.
\]
Thus, by \cref{lemma: sequentially way-below interpolation}, there exist $\big\{ u_i \in \Cu(A_{s_i^*s_i}) : i \in \{1, \dotsc, m\} \big\}$ such that $u_i \ll v_i$ for each $i \in \{1, \dotsc, m\}$, and
\begin{inequality} \label{ineq: alpha-below sum 1}
f \ll \sum_{i=1}^m u_i \ll \sum_{i=1}^m v_i.
\end{inequality}
For each $i \in \{1, \dotsc, m\}$, $\hat{\alpha}_{s_i}$ is a $\Cu$-morphism, and hence $\hat{\alpha}_{s_i}(u_i) \ll \hat{\alpha}_{s_i}(v_i)$. Thus, by axiom~\cref{item: O3} of \cref{defn: Cu-semigroup}, we have
\begin{inequality} \label{ineq: alpha-below sum 2}
\sum_{i=1}^m \hat{\alpha}_{s_i}(u_i) \ll \sum_{i=1}^m \hat{\alpha}_{s_i}(v_i) \le y.
\end{inequality}
It follows from \cref{lemma: ll le properties,ineq: alpha-below sum 1,ineq: alpha-below sum 2} that
\[
f \ll \sum_{i=1}^m u_i \quad \text{ and } \quad \sum_{i=1}^m \hat{\alpha}_{s_i}(u_i) \ll y,
\]
and hence $x \precalpha y$.
\end{proof}

\begin{cor} \label{cor: precalpha properties}
Let $\alpha\colon S \curvearrowright A$ be an action of a unital inverse semigroup $S$ on a C*-algebra $A$. The $\alpha$-below relation $\precalpha$ has the following properties.
\begin{enumerate}[label=(\alph*)]
\item \label{item: le implies precalpha} For all $x, y \in \Cu(A)$, we have $x \le y \implies x \precalpha y$. In particular, $\precalpha$ is reflexive, and $0 \precalpha x$ for all $x \in \Cu(A)$.
\item \label{item: x precalpha alpha_s(x)} For all $s \in S$ and $x \in \Cu(A_{s^*s})$, we have $x \precalpha \hat{\alpha}_s(x)$ and $\hat{\alpha}_s(x) \precalpha x$.
\end{enumerate}
\end{cor}

\begin{proof}
For part~\cref{item: le implies precalpha}, fix $x, y \in \Cu(A)$ such that $x \le y$. To see that $x \precalpha y$, apply \cref{lemma: alpha-below characterisation} with $m = 1$, $s_1 = 1$, and $v_1 = x$. Since $x \le x$ and $0 \le x$, it follows immediately that $x \precalpha x$ and $0 \precalpha x$. For part~\cref{item: x precalpha alpha_s(x)}, fix $s \in S$ and $x \in \Cu(A_{s^*s})$. To see that $x \precalpha \hat{\alpha}_s(x)$, apply \cref{lemma: alpha-below characterisation} with $m = 1$, $s_1 = s$, and $v_1 = x$. To see that $\hat{\alpha}_s(x) \precalpha x$, apply \cref{lemma: alpha-below characterisation} with $m = 1$, $s_1 = s^*$, and $v_1 = \hat{\alpha}_s(x)$.
\end{proof}

\begin{lemma} \label{lemma: Cuntz semigroup addition is well-defined}
Let $\alpha\colon S \curvearrowright A$ be an action of a unital inverse semigroup $S$ on a C*-algebra $A$. Fix $x, y, z, w \in \Cu(A)$. If $x \precalpha z$ and $y \precalpha w$, then $x + y \precalpha z + w$.
\end{lemma}

\begin{proof}
Suppose that $x \precalpha z$ and $y \precalpha w$. Fix $f \ll x + y$. Since $\Cu(A)$ is a $\Cu$-semigroup by \cref{thm: Cu(A) is Cu-semigroup}, \cref{lemma: sequentially way-below interpolation} implies that there exist $x',y' \in \Cu(A)$ such that $f \ll x' + y' \ll x + y$ and $x' \ll x$ and $y' \ll y$. Since $x \precalpha z$ and $y \precalpha w$, there exist $\{s_i\}_{i=1}^m, \{t_j\}_{j=1}^n \subseteq S$ and collections $\big\{ u_i \in \Cu(A_{s_i^*s_i}) : i \in \{1, \dotsc, m\} \big\}$ and $\big\{ v_j \in \Cu(A_{t_j^*t_j}) : j \in \{1, \dotsc, n\} \big\}$ such that
\[
x' \ll \sum_{i=1}^m u_i, \,\quad \sum_{i=1}^m \hat{\alpha}_{s_i}(u_i) \ll z, \,\quad y' \ll \sum_{j=1}^n v_j, \quad \text{ and } \,\quad \sum_{j=1}^m \hat{\alpha}_{t_j}(v_j) \ll w.
\]
Using axiom~\cref{item: O3} of \cref{defn: Cu-semigroup}, we see that
\[
f \ll x' + y' \ll \sum_{i=1}^m u_i + \sum_{j=1}^n v_j \quad \text{ and } \quad \sum_{i=1}^m \hat{\alpha}_{s_i}(u_i) + \sum_{j=1}^n \hat{\alpha}_{t_j}(v_j) \ll z + w,
\]
so $x + y \precalpha z + w$.
\end{proof}

We now complete $\precalpha$ to obtain a transitive preorder $\pathalpha$.

\begin{defn}
Let $\alpha\colon S \curvearrowright A$ be an action of a unital inverse semigroup $S$ on a C*-algebra $A$. For $x, y \in \Cu(A)$, we say that $x$ is \emph{$\alpha$-path-below} $y$, and we write $x \pathalpha y$, if there exist $z_1, \dotsc, z_n \in \Cu(A)$ such that $x \precalpha z_1 \precalpha z_2 \precalpha \dotsb \precalpha z_n \precalpha y$.
\end{defn}

It is clear that $\pathalpha$ is transitive. Moreover, it follows from the reflexivity of $\precalpha$ (see \cref{cor: precalpha properties}\cref{item: le implies precalpha}) that $\pathalpha$ contains $\precalpha$ (and hence is also reflexive): for all $x, y \in \Cu(A)$ with $x \precalpha y$, we have $x \precalpha x \precalpha y$, so $x \pathalpha y$. Indeed, $\pathalpha$ is the smallest transitive relation containing $\precalpha$.

We now show that $\pathalpha$ is compatible with the addition in the Cuntz semigroup.

\begin{lemma} \label{lemma: path order is compatible with addition}
Let $\alpha\colon S \curvearrowright A$ be an action of a unital inverse semigroup $S$ on a C*-algebra $A$. Fix $x, y, z, w \in \Cu(A)$. If $x \pathalpha z$ and $y \pathalpha w$, then $x + y \pathalpha z + w$.
\end{lemma}

\begin{proof}
Suppose that $x \pathalpha z$ and $y \pathalpha w$. By trivially extending paths if necessary (using that $\precalpha$ is reflexive), we see that for some $n \in \N$ there exist $\{e_i\}_{i=1}^n, \{f_i\}_{i=1}^n \subseteq \Cu(A)$ such that
\[
x \precalpha e_1 \precalpha \dotsb \precalpha e_n \precalpha z
\quad \text{ and } \quad y \precalpha f_1 \precalpha \dotsb \precalpha f_n \precalpha w.
\]
By \cref{lemma: Cuntz semigroup addition is well-defined}, we have
\[
x + y \precalpha e_1 + f_1 \precalpha \dotsb \precalpha e_n + f_n \precalpha z + w,
\]
and thus $x + y \pathalpha z + w$.
\end{proof}

\begin{lemma} \label{lemma: functionals preserve precalpha}
Let $\alpha\colon S \curvearrowright A$ be an action of a unital inverse semigroup $S$ on a C*-algebra $A$, and let $\beta\colon \Cu(A) \to [0,\infty]$ be an invariant functional. Fix $x, y \in \Cu(A)$. If $x \precalpha y$ or $x \pathalpha y$, then $\beta(x) \le \beta(y)$.
\end{lemma}

\begin{proof}
Suppose that $x \precalpha y$. Since $\Cu(A)$ is a $\Cu$-semigroup, axiom~\cref{item: O2} of \cref{defn: Cu-semigroup} implies that there is a sequence $\{x_i\}_{i \in \N} \subseteq \Cu(A)$ such that $x_i \ll x_{i+1}$ for all $i \in \N$ and $x = \sup_{i \in \N} (x_i)$. Fix $i \in \N$. Since $x_i \ll x_{i+1} \le x$, \cref{lemma: ll le properties} implies that $x_i \ll x$. Thus, since $x \precalpha y$, there exist $\{s_{i,j}\}_{j=1}^{n_i} \subseteq S$ and $\big \{ u_{i,j} \in \Cu(A_{s_{i,j}^*s_{i,j}}) : j \in \{1, \dotsc, n_i\} \big\}$ such that
\[
x_i \ll \sum_{j=1}^{n_i} u_{i,j} \quad \text{ and } \quad \sum_{j=1}^{n_i} \hat{\alpha}_{s_{i,j}}(u_{i,j}) \ll y.
\]
Since $\beta$ is an invariant functional and $a \ll b \implies a \le b$ for all $a, b \in \Cu(A)$, it follows that
\[
\beta(x_i) \le \beta\Big(\sum_{j=1}^{n_i} u_{i,j}\Big) = \sum_{j=1}^{n_i} \beta(u_{i,j}) = \sum_{j=1}^{n_i} \beta\big(\hat{\alpha}_{s_{i,j}}(u_{i,j})\big) = \beta\Big(\sum_{j=1}^{n_i} \hat{\alpha}_{s_{i,j}}(u_{i,j})\Big) \le \beta(y).
\]
Thus $\beta(x_i) \le \beta(y)$ for each $i \in \N$. Since functionals preserve suprema of increasing sequences, we have
\[
\beta(x) = \beta\big(\!\sup_{i \in \N} (x_i)\big) = \sup_{i \in \N} \!\big(\beta(x_i)\big) \le \beta(y).
\]
Therefore, $x \precalpha y \implies \beta(x) \le \beta(y)$. Now, since $\le$ is transitive, it follows immediately by definition of the $\pathalpha$ relation that $x \pathalpha y \implies \beta(x) \le \beta(y)$.
\end{proof}

\begin{lemma} \label{cor: iota hat preserves precalpha}
Let $\alpha\colon S \curvearrowright A$ be an action of a unital inverse semigroup $S$ on a C*-algebra $A$. Let $\iota\colon A \to A \rtimesess S$ be the canonical embedding $a \mapsto a\delta_1$, and let $\hat{\iota}\colon\Cu(A) \to \Cu(A \rtimesess S)$ be the induced $\Cu$-morphism from \cref{lemma: Cu-embedding in crossed product is invariant}. Fix $x, y \in \Cu(A)$. If $x \precalpha y$ or $x \pathalpha y$, then $\hat{\iota}(x) \le \hat{\iota}(y)$.
\end{lemma}

\begin{proof}
Suppose that $x \precalpha y$. Since $\Cu(A)$ is a $\Cu$-semigroup, axiom~\cref{item: O2} of \cref{defn: Cu-semigroup} implies that there is a sequence $\{x_i\}_{i \in \N} \subseteq \Cu(A)$ such that $x_i \ll x_{i+1}$ for all $i \in \N$ and $x = \sup_{i \in \N} (x_i)$. Fix $i \in \N$. Since $x_i \ll x_{i+1} \le x$, \cref{lemma: ll le properties} implies that $x_i \ll x \precalpha y$. So there exist $\{s_{i,j}\}_{j=1}^{n_i} \subseteq S$ and $\big \{ u_{i,j} \in \Cu(A_{s_{i,j}^*s_{i,j}}) : j \in \{1, \dotsc, n_i\} \big\}$ such that
\[
x_i \ll \sum_{j=1}^{n_i} u_{i,j} \quad \text{ and } \quad \sum_{j=1}^{n_i} \hat{\alpha}_{s_{i,j}}(u_{i,j}) \ll y.
\]
Since $\hat{\iota}$ is a $\Cu$-morphism, it preserves addition and the sequentially way-below relation $\ll$. Moreover, by \cref{lemma: Cu-embedding in crossed product is invariant}, $\hat{\iota}$ is $\hat{\alpha}$-invariant. Thus, for each $i \in \N$, we have
\[
\hat{\iota}(x_i) \ll \hat{\iota}\Big( \sum_{j=1}^{n_i} u_{i,j} \Big) = \sum_{j=1}^{n_i} \hat{\iota}(u_{i,j}) = \sum_{j=1}^{n_i} \hat{\iota}\big( \hat{\alpha}_{s_{i,j}}(u_{i,j}) \big) = \hat{\iota}\Big( \sum_{j=1}^{n_i} \hat{\alpha}_{s_{i,j}}(u_{i,j}) \Big) \ll \hat{\iota}(y).
\]
Since $\hat{\iota}$ is a $\Cu$-morphism, it preserves suprema of increasing sequences, so it follows that
\[
\hat{\iota}(x) = \hat{\iota}\big(\!\sup_{i \in \N} (x_i)\big) = \sup_{i \in \N} \!\big(\hat{\iota}(x_i)\big) \le \hat{\iota}(y).
\]
Therefore, $x \precalpha y \implies \hat{\iota}(x) \le \hat{\iota}(y)$. Now, since $\le$ is transitive, it follows immediately by definition of the $\pathalpha$ relation that $x \pathalpha y \implies \hat{\iota}(x) \le \hat{\iota}(y)$.
\end{proof}

We now introduce a refinement property for $\Cu(A)$ that forces $\precalpha$ to be transitive (see \cref{lemma: almost refinement implies precalpha transitive}).

\begin{defn}[cf.~{\cite[Theorem~4.1.1]{Robert2013a}}]
Let $A$ be a C*-algebra. We say that $\Cu(A)$ has \emph{almost refinement} if, for all
\[
\{ f_i \}_{i=1}^m, \{ x_i \}_{i=1}^m, \{ y_j \}_{j=1}^n \subseteq \Cu(A)
\]
satisfying
\[
f_i \ll x_i \ \text{ for each } i \in \{1, \dotsc, m\} \quad \text{and} \quad \sum_{i=1}^m x_i \le \sum_{j=1}^n y_j,
\]
there exists a collection $\big\{ u_{i,j} \in \Cu(A) : (i,j) \in \{1, \dotsc, m\} \times \{1, \dotsc, n\} \big\}$ such that
\[
f_i \ll \sum_{j=1}^n u_{i,j} \ll x_i \ \text{ for each } i \in \{1, \dotsc, m\} \quad \text{and} \quad \sum_{i=1}^m u_{i,j} \le y_j \ \text{ for each } j \in \{1, \dotsc, n\}.
\]
\end{defn}

Almost refinement of $\Cu(A)$ as defined above is weaker than $\Cu(A)$ being a refinement monoid satisfying the Riesz decomposition property and the Riesz interpolation property \cite[Page~3]{Perera1997}. Thus, by applying \cite[Theorem~2.13]{Perera1997}, we see that if $A$ is a $\sigma$-unital C*-algebra with real rank~$0$ and stable rank~$1$, then $\Cu(A)$ has the almost-refinement property.

The almost-refinement property also appears in \cite[Theorem~4.1.1]{Robert2013a}, although it is not named there. We chose the name ``almost refinement'' to align with \cite[Paragraph~7.8]{BPWZ2025}. By \cite[Theorem~4.1.1 and Lemma~5.1.3]{Robert2013a}, $\Cu(A)$ has almost refinement if $A$ is $\WW$-stable (that is, $A \cong A \otimes \WW$, where $\WW$ is the stably projectionless Jacelon--Razak algebra of \cite{Jacelon2013}). A related property for open subsets of a Hausdorff space $X$ appears in \cite[Lemma~4.13]{KMP2025}, which we use in \cref{sec: groupoid C*-algebras} to show that $\Lsc(X,\bar{\N})$ has almost refinement.

\begin{lemma} \label{lemma: almost refinement implies precalpha transitive}
Let $\alpha\colon S \curvearrowright A$ be an action of a unital inverse semigroup $S$ on a C*-algebra $A$. If $\Cu(A)$ has almost refinement, then $\precalpha$ and $\pathalpha$ coincide.
\end{lemma}

\begin{proof}
Suppose that $\Cu(A)$ has almost refinement. Since $\pathalpha$ is the smallest transitive relation containing $\precalpha$, it suffices to show that $\precalpha$ is transitive. Fix $x, y, z \in \Cu(A)$, and suppose that $x \precalpha y$ and $y \precalpha z$. We want to show that $x \precalpha z$. Fix $f \in \Cu(A)$ such that $f \ll x$. Since $x \precalpha y$, there exist $\{s_i\}_{i=1}^m \subseteq S$ and $\big\{ u_i \in \Cu(A_{s_i^*s_i}) : i \in \{1, \dotsc, m\} \big\}$ such that
\[
f \ll \sum_{i=1}^m u_i \quad \text{ and } \quad \sum_{i=1}^m \hat{\alpha}_{s_i}(u_i) \ll y.
\]
Thus, since $y \precalpha z$, there exist $\{t_j\}_{j=1}^n \subseteq S$ and $\big\{ w_j \in \Cu(A_{t_j^*t_j}) : j \in \{1, \dotsc, n\} \big\}$ such that
\[
\sum_{i=1}^m \hat{\alpha}_{s_i}(u_i) \ll \sum_{j=1}^n w_j \quad \text{ and } \quad \sum_{j=1}^n \hat{\alpha}_{t_j}(w_j) \ll z.
\]
By \cref{lemma: sequentially way-below interpolation} there is a collection $\big\{ u_i' \in \Cu(A_{s_i^*s_i}) : i \in \{1, \dotsc, m\} \big\}$ such that $u_i' \ll u_i$ for each $i \in \{1, \dotsc, m\}$, and
\begin{inequality} \label{ineq: f ll sum u_i'}
f \ll \sum_{i=1}^m u_i' \ll \sum_{i=1}^m u_i.
\end{inequality}
Since $\hat{\alpha}$ is a $\Cu$-morphism, it preserves the sequentially way-below relation $\ll$, so we have $\hat{\alpha}_{s_i}(u_i') \ll \hat{\alpha}_{s_i}(u_i)$ for each $i \in \{1, \dotsc m\}$. Now we can apply almost refinement to
\[
\hat{\alpha}_{s_i}(u_i') \ll \hat{\alpha}_{s_i}(u_i) \ \text{ for each } i \in \{1, \dotsc, m\} \quad \text{and} \quad \sum_{i=1}^m \hat{\alpha}_{s_i}(u_i) \le \sum_{j=1}^n w_j
\]
to obtain a collection $\big\{ v_{i,j} \in \Cu(A) : (i,j) \in \{1, \dotsc, m\} \times \{1, \dotsc, n\} \big\}$ such that
\begin{inequality} \label{ineq: u_i' v_ij}
\hat{\alpha}_{s_i}(u_i') \ll \sum_{j=1}^n v_{i,j} \ll \hat{\alpha}_{s_i}(u_i) \ \text{ for each } i \in \{1, \dotsc, m\}
\end{inequality}
and
\begin{inequality} \label{ineq: v_ij w_j}
\sum_{i=1}^m v_{i,j} \le w_j \ \text{ for each } j \in \{1, \dotsc, n\}.
\end{inequality}
It follows from \cref{ineq: u_i' v_ij} that for each $i \in \{1, \dotsc, m\}$ and $k \in \{1, \dotsc, n\}$,
\[
v_{i,k} \le \sum_{j=1}^n v_{i,j} \ll \hat{\alpha}_{s_i}(u_i) \in \Cu(A_{s_is_i^*}),
\]
so since $\Cu(A_{s_is_i^*})$ is hereditary, $v_{i,k} \in \Cu(A_{s_is_i^*})$. Similarly, it follows from \cref{ineq: v_ij w_j} that for each $\ell \in \{1, \dotsc, m\}$ and $j \in \{1, \dotsc, n\}$,
\[
v_{\ell,j} \le \sum_{i=1}^m v_{i,j} \le w_j \in \Cu(A_{t_j^*t_j}),
\]
so since $\Cu(A_{t_j^*t_j})$ is hereditary, $v_{\ell,j} \in \Cu(A_{t_j^*t_j})$. For each $(i,j) \in \{1, \dotsc, m\} \times \{1, \dotsc, n\}$, let
\[
v_{i,j}' \coloneqq \hat{\alpha}_{s_i^*}(v_{i,j}) \in \Cu(A_{s_i^*s_i}).
\]
Then by \cref{ineq: f ll sum u_i',ineq: u_i' v_ij}, and using that each $\hat{\alpha}_{s_i^*}$ is a $\Cu$-morphism, we obtain
\[
f \ll \sum_{i=1}^m u_i' = \sum_{i=1}^m \hat{\alpha}_{s_i^*}\big(\hat{\alpha}_{s_i}(u_i')\big) \ll \sum_{i=1}^m \hat{\alpha}_{s_i^*}\Big( \sum_{j=1}^n v_{i,j} \Big) = \sum_{i=1}^m \sum_{j=1}^n \hat{\alpha}_{s_i^*}(v_{i,j}) = \sum_{i=1}^m \sum_{j=1}^n v_{i,j}'.
\]
By \cref{ineq: v_ij w_j}, and also using that each $\hat{\alpha}_{t_j}$ is a $\Cu$-morphism, we obtain
\[
\sum_{i=1}^m \sum_{j=1}^n \hat{\alpha}_{t_js_i}(v_{i,j}') = \sum_{i=1}^m \sum_{j=1}^n \hat{\alpha}_{t_j}(v_{i,j}) = \sum_{j=1}^n \sum_{i=1}^m \hat{\alpha}_{t_j}(v_{i,j}) = \sum_{j=1}^n \hat{\alpha}_{t_j}\Big( \sum_{i=1}^m v_{i,j} \Big) \le \sum_{j=1}^n \hat{\alpha}_{t_j}(w_j) \ll z.
\]
So \cref{lemma: ll le properties} implies that
\[
f \ll \sum_{i=1}^m \sum_{j=1}^n v_{i,j}' \quad \text{ and } \quad \sum_{i=1}^m \sum_{j=1}^n \hat{\alpha}_{t_js_i}(v_{i,j}') \ll z.
\]
Thus $x \precalpha z$, so $\precalpha$ is transitive.
\end{proof}

We do not know whether almost refinement is a necessary condition for $\precalpha$ to be transitive. However, since $\pathalpha$ is always transitive (and reflexive), we can use $\pathalpha$ to define an equivalence relation on $\Cu(A)$, as follows.

\begin{defn} \label{defn: dynamical cuntz semigroup}
Let $\alpha\colon S \curvearrowright A$ be an action of a unital inverse semigroup $S$ on a C*-algebra $A$. Let $\simalpha$ denote the relation on $\Cu(A)^\ll$ given by
\[
x \simalpha y \iff x \pathalpha y \,\text{ and }\, y \pathalpha x
\]
for $x, y \in \Cu(A)^\ll$. Since $\precalpha$ is reflexive and transitive, $\simalpha$ is an equivalence relation. We denote the $\simalpha$-equivalence class of $x \in \Cu(A)^\ll$ by $[x]_\alpha$. For all $x, y \in \Cu(A)^\ll$, we have $x + y \in \Cu(A)^\ll$, and so \cref{lemma: path order is compatible with addition} implies that the operation
\[
[x]_\alpha + [y]_\alpha \coloneqq [x + y]_\alpha
\]
is well-defined. Under this operation, $\Cu(A)_\alpha \coloneqq \Cu(A)^\ll / \!\!\simalpha$ is an abelian monoid with identity $0 = [0]_\alpha$. Moreover, \cref{lemma: path order is compatible with addition} implies that $\Cu(A)_\alpha$ is preordered with respect to the relation
\[
[x]_\alpha \le [y]_\alpha \iff x \pathalpha y.
\]
We call $\Cu(A)_\alpha$ the \emph{dynamical Cuntz semigroup} of $\alpha\colon S \curvearrowright A$.
\end{defn}

\begin{remark} \label{rem: alpha class properties} \leavevmode
\begin{enumerate}[label=(\alph*), ref={\cref{rem: alpha class properties}(\alph*)}]
\item \label{item: [0]_alpha singleton} A routine argument using \cref{rem: nonzero sequentially way below} shows that $[x]_\alpha \ne [0]_\alpha$ for all $x \in \Cu(A)^\ll {\setminus} \{0\}$. In particular, if $A$ is unital, then $[1_A]_\Cu \in \Cu(A)^\ll$ by \cref{rem: [p]_Cu ll [p]_Cu}, and so $\Cu(A)_\alpha \ne \{0\}$.
\item \label{item: alpha classes invariant under alpha} Since $\pathalpha$ contains $\precalpha$ and since each $\hat{\alpha}_s$ preserves the sequentially way-below relation $\ll$, it follows from \cref{cor: precalpha properties}\cref{item: x precalpha alpha_s(x)} that $[\hat{\alpha}_s(x)]_\alpha = [x]_\alpha$ for all $s \in S$ and $x \in \Cu(A_{s^*s})^\ll$.
\end{enumerate}
\end{remark}

We chose the name ``dynamical Cuntz semigroup'' due to the similarity between $\Cu(A)_\alpha$ and the dynamical Cuntz semigroup defined by Bosa--Perera--Wu--Zacharias in \cite[Paragraph~8.2]{BPWZ2025} in the setting where the inverse semigroup $S$ is a group. Other analogous objects are called ``type semigroups'' \cite{Rainone2017, RS2020} or ``generalised type semigroups'' \cite{Ma2021, Ma2023}.

Our next result, \cref{prop: Cu minimality}, mimics \cite[Proposition~4.7]{Rainone2017}, with the modifications needed for a partially defined action. When $A$ is commutative (see \cref{sec: groupoid C*-algebras}), the analogue is using an open subset of the unit space $\GGo$ of a minimal groupoid $\GG$ and the action of the inverse semigroup of open bisections of $\GG$ to cover any given compact subset of the unit space $\GGo$.

\begin{prop} \label{prop: Cu minimality}
Let $\alpha\colon S \curvearrowright A$ be a minimal action of a unital inverse semigroup $S$ on a C*-algebra $A$. Fix $x \in \Cu(A)^\ll$ and $y \in \Cu(A)^\ll {\setminus} \{0\}$. Then there exist $\{s_i\}_{i=1}^m \subseteq S$ and $\big\{ z_i \in \Cu(A_{s_i^*s_i})^\ll : i \in \{1, \dotsc, m\} \big\}$ such that $z_i \le y$ for each $i \in \{1, \dotsc, m\}$, and
\[
x \le \sum_{i=1}^m \hat{\alpha}_{s_i}(z_i).
\]
Moreover, $[x]_\alpha \le m[y]_\alpha$.
\end{prop}

In order to prove \cref{prop: Cu minimality}, we need the following lemma.

\begin{lemma} \label{lemma: J_b is dense in A otimes K}
Let $\alpha\colon S \curvearrowright A$ be a minimal action of a unital inverse semigroup $S$ on a C*\nobreakdash-algebra $A$. For each $s \in S$, let $\alpha^\KK_s \coloneqq \alpha_s \otimes \id$ be the standard amplification of $\alpha_s$ to $A \otimes \KK$. Fix $b \in (A \otimes \KK)^+ {\setminus} \{0\}$, and define
\[
J_b \coloneqq \vecspan\!\big\{ c \,\alpha^\KK_s(rbr^*)\, d \,:\, s \in S, \, c, d, r \in A \otimes \KK \text{ such that } rbr^* \in A_{s^*s} \otimes \KK \big\}.
\]
Then $J_b$ is a two-sided algebraic $*$-ideal of $A \otimes \KK$, and $\overline{J_b} = A \otimes \KK$.
\end{lemma}

\begin{proof}
Since $b^* = b \ne 0$, $J_b$ is a nonzero two-sided algebraic $*$-ideal of $A \otimes \KK$. Since $\alpha$ is minimal and the only nonzero ideals of $A \otimes \KK$ have the form $I \otimes \KK$ for $\{0\} \ne I \trianglelefteq A$, the action $\alpha^\KK\colon S \curvearrowright A \otimes \KK$ is minimal. Hence the only nonzero invariant ideal of $A \otimes \KK$ is $A \otimes \KK$ itself, and so to see that $\overline{J_b} = A \otimes \KK$ it suffices to show that $\overline{J_b}$ is invariant.

For this, fix $t \in S$, and suppose that $w \in J_b \cap (A_{t^*t} \otimes \KK)$. Then there exist $\{s_i\}_{i=1}^m \subseteq S$ and $\{c_i\}_{i=1}^m, \{d_i\}_{i=1}^m, \{r_i\}_{i=1}^m \subseteq A \otimes \KK$ such that $r_i b r_i^* \in A_{s_i^*s_i} \otimes \KK$ for each $i \in \{1, \dotsc, m\}$, and
\[
w = \sum_{i=1}^m c_i \,\alpha^\KK_{s_i}(r_i b r_i^*)\, d_i^*.
\]
We will show that $\alpha^\KK_t(w) \in \overline{J_b}$. Let $\{e_\lambda\}_\lambda$ be an approximate identity for $A_{t^*t} \otimes \KK$. Then $e_\lambda \, c_i \,\alpha^\KK_{s_i}(r_i b r_i^*)\, d_i^* \, e_\lambda \in A_{t^*t} \otimes \KK$ for each $i \in \{1, \dotsc, m\}$ and each $\lambda$, so since $\alpha^\KK_t$ is continuous and additive, we have
\[
\alpha^\KK_t(w) = \alpha^\KK_t\big( \!\lim_\lambda \, e_\lambda w e_\lambda \big) = \lim_\lambda \, \alpha^\KK_t(e_\lambda w e_\lambda) = \lim_\lambda \sum_{i=1}^m \alpha^\KK_t \big( e_\lambda \, c_i \,\alpha^\KK_{s_i}(r_i b r_i^*)\, d_i^* \, e_\lambda \big).
\]
Since $e_\lambda \, c_i, \, d_i^* \, e_\lambda \in A_{t^*t} \otimes \KK$ for each $i \in \{1, \dotsc, m\}$ and each $\lambda$, to see that $\alpha^\KK_t(w) \in \overline{J_b}$, it suffices to fix $c, d \in A_{t^*t} \otimes \KK$, $s \in S$, and $r \in A \otimes \KK$ such that $rbr^* \in A_{s^*s} \otimes \KK$, and show that for $h \coloneqq c \, \alpha^\KK_s(rbr^*) \, d$, we have $\alpha^\KK_t(h) \in \overline{J_b}$.

For this, let $\{f_\mu\}_\mu$ be an approximate identity for $A_{(ts)^*ts} \otimes \KK$. Then $\big\{\alpha^\KK_s(f_\mu)\big\}_\mu$ is an approximate identity for $(A_{t^*t} \cap A_{ss^*}) \otimes \KK$. Since $h = c \, \alpha^\KK_s(rbr^*) \, d \in (A_{t^*t} \cap A_{ss^*}) \otimes \KK$, we have
\begin{equation} \label{eqn: h limit}
h = \lim_\mu \, \alpha^\KK_s(f_\mu) \, h \, \alpha^\KK_s(f_\mu) = \lim_\mu \, \alpha^\KK_s(f_\mu) \, c \, \alpha^\KK_s(rbr^*) \, d \, \alpha^\KK_s(f_\mu).
\end{equation}
Since $\alpha^\KK_s(f_\mu) \, c, \, d \, \alpha^\KK_s(f_\mu) \in (A_{t^*t} \cap A_{ss^*}) \otimes \KK$ for each $\mu$, we have
\begin{align*}
\alpha^\KK_s(f_\mu) \, c \, \alpha^\KK_s(rbr^*) \, d \, \alpha^\KK_s(f_\mu) &= \lim_\gamma \, \alpha^\KK_s(f_\mu) \, c \, \alpha^\KK_s(f_\gamma) \, \alpha^\KK_s(rbr^*) \, \alpha^\KK_s(f_\gamma^*) \, d \, \alpha^\KK_s(f_\mu) \\
&= \lim_\gamma \, \alpha^\KK_s(f_\mu) \, c \, \alpha^\KK_s\big((f_\gamma r) b (f_\gamma r)^*\big) \, d \, \alpha^\KK_s(f_\mu) \numberthis \label{eqn: expanded h limit}
\end{align*}
for each $\mu$. Since $(f_\gamma r) b (f_\gamma r)^* \in A_{(ts)^*ts} \otimes \KK$ for each $\gamma$, we have
\[
\alpha^\KK_s(f_\mu) \, c, \, \alpha^\KK_s\big((f_\gamma r) b (f_\gamma r)^*\big), \, d \, \alpha^\KK_s(f_\mu) \in A_{t^*t} \otimes \KK
\]
for each $\mu$ and $\gamma$. Thus, since $\alpha^\KK_t$ is a continuous homomorphism on $A_{t^*t} \otimes \KK$, it follows from \cref{eqn: h limit,eqn: expanded h limit} that
\begin{align*}
\alpha^\KK_t(h) &= \alpha^\KK_t\Big( \!\lim_\mu \, \lim_\gamma \, \alpha^\KK_s(f_\mu) \, c \, \alpha^\KK_s\big((f_\gamma r) b (f_\gamma r)^*\big) \, d \, \alpha^\KK_s(f_\mu) \Big) \\
&= \lim_\mu \, \lim_\gamma \, \alpha^\KK_t\Big( \alpha^\KK_s(f_\mu) \, c \, \alpha^\KK_s\big((f_\gamma r) b (f_\gamma r)^*\big) \, d \, \alpha^\KK_s(f_\mu) \Big) \\
&= \lim_\mu \, \lim_\gamma \, \alpha^\KK_t\big( \alpha^\KK_s(f_\mu) \, c \big) \, \alpha^\KK_t\big( \alpha^\KK_s\big((f_\gamma r) b (f_\gamma r)^*\big) \big) \, \alpha^\KK_t\big( d \, \alpha^\KK_s(f_\mu) \big) \\
&= \lim_\mu \, \lim_\gamma \, \alpha^\KK_t\big( \alpha^\KK_s(f_\mu) \, c \big) \, \alpha^\KK_{ts}\big( (f_\gamma r) b (f_\gamma r)^* \big) \, \alpha^\KK_t\big( d \, \alpha^\KK_s(f_\mu) \big),
\end{align*}
and so $\alpha^\KK_t(h) \in \overline{J_b}$.

Therefore, $\alpha^\KK_t(w) \in \overline{J_b}$, and since $\alpha^\KK_t$ is continuous, it follows that $\alpha^\KK_t\big( \overline{J_b} \cap (A_{t^*t} \otimes \KK)\big) \subseteq \overline{J_b}$. So $\overline{J_b}$ is a nonzero invariant ideal of $A \otimes \KK$, and hence $\overline{J_b} = A \otimes \KK$ by minimality of $\alpha^\KK$.
\end{proof}

\begin{proof}[Proof of \cref{prop: Cu minimality}]
Since $x \in \Cu(A)^\ll$, there exists $w \in \Cu(A)$ such that $x \ll w$. If $w = 0$, then $x = 0$ and the proof is trivial, so assume that $w \ne 0$. Then since $y \ne 0$ as well, there exist $a, b \in (A \otimes \KK)^+ {\setminus} \{0\}$ such that $w = [a]_\Cu$ and $y = [b]_\Cu$. For each $s \in S$, let $\alpha^\KK_s \coloneqq \alpha_s \otimes \id$ be the standard amplification of $\alpha_s$ to $A \otimes \KK$. Then $\hat{\alpha}_s([z]_\Cu) = [\alpha^\KK_s(z)]_\Cu$ for all $s \in S$ and $z \in (A_{s^*s} \otimes \KK)^+$. Define
\[
J_b \coloneqq \vecspan\!\big\{ c \,\alpha^\KK_s(rbr^*)\, d \,:\, s \in S, \, c, d, r \in A \otimes \KK \text{ such that } rbr^* \in A_{s^*s} \otimes \KK \big\}.
\]
We have $b \ne 0$ because $y \ne 0$, so \cref{lemma: J_b is dense in A otimes K}, implies that $J_b$ is a dense ideal of $A \otimes \KK$. Therefore, by \cite[Theorem~5.6.1]{Pedersen2018}, $J_b$ contains the Pedersen ideal. For each $t \in (0,\norm{a}]$, let $\phi_t\colon (0,\infty) \to [0,\infty)$ be the continuous piecewise linear function given by
\[
\phi_t(r) \coloneqq \begin{cases}
0 & \text{ if } 0 < r \le t \\
r - t & \text{ if } t < r \le \norm{a} \\
\big(\norm{a} - t\big) \big(\norm{a} + 1 - r\big) & \text{ if } \norm{a} < r \le \norm{a} + 1 \\
0 & \text{ if } r > \norm{a} + 1.
\end{cases}
\]
Then for $a$ as above, we have $\phi_t\restr{\sigma(a)}(r) = \max\{0, r - t\}$. Since $\supp(\phi_t) = [t, \norm{a} + 1]$ is a compact subset of $(0,\infty)$, $\phi_t(a)$ is in the Pedersen ideal for each $t \in (0,\norm{a}]$. For each $n \in \N$, define
\[
a_n \coloneqq \begin{cases}
0 & \text{ if } \frac{1}{n} > \norm{a} \\[0.25ex]
\phi_{\frac{1}{n}}(a) & \text{ if } \frac{1}{n} \le \norm{a}.
\end{cases}
\]
Then $\{a_n\}_{n \in \N} \subseteq J_b$ since $J_b$ contains the Pedersen ideal. Since $t \mapsto \phi_t(r) = \max\{0, r - t\}$ is a decreasing function for each $r \in \sigma(a)$, we have $0 \le a_n \le a_{n+1} \le a$ for each $n \in \N$. So $\{a_n\}_{n \in \N}$ is an increasing sequence with supremum $a$, and hence the sequence $\big\{[a_n]_\Cu\big\}_{n \in \N} \subseteq \Cu(A)$ is increasing and has supremum
\[
\sup_{n \in \N} \big([a_n]_\Cu\big) = \Big[\!\sup_{n \in \N} (a_n)\Big]_\Cu = [a]_\Cu = w.
\]
Since $x \ll w$, it follows that $x \le [a_p]_\Cu$ for some $p \in \N$. Since $a_p \in J_b$, there exist $\{s_i\}_{i=1}^m \subseteq S$ and $\{c_i\}_{i=1}^m, \{d_i\}_{i=1}^m, \{r_i\}_{i=1}^m \subseteq A \otimes \KK$ such that $r_i b r_i^* \in A_{s_i^*s_i} \otimes \KK$ for each $i \in \{1, \dotsc, m\}$, and
\[
a_p = \sum_{i=1}^m c_i \,\alpha^\KK_{s_i}(r_i b r_i^*)\, d_i^*.
\]
Since $a_p, b \ge 0$, we have
\begin{equation} \label{eqn: a_p star}
a_p = a_p^* = \sum_{i=1}^m d_i \,\alpha^\KK_{s_i}(r_i b r_i^*)^*\, c_i^* = \sum_{i=1}^m d_i \,\alpha^\KK_{s_i}(r_i b r_i^*)\, c_i^*.
\end{equation}
For each $i \in \{1, \dotsc, m\}$, set
\[
h_i \coloneqq c_i + d_i, \quad C_i \coloneqq c_i \,\alpha^\KK_{s_i}(r_i b r_i^*)\, c_i^*, \quad \text{and } \quad D_i \coloneqq d_i \,\alpha^\KK_{s_i}(r_i b r_i^*)\, d_i^*.
\]
Then by \cref{eqn: a_p star},
\begin{align*}
\sum_{i=1}^m h_i \,\alpha^\KK_{s_i}(r_i b r_i^*)\, h_i^* \,&=\, \sum_{i=1}^m c_i \,\alpha^\KK_{s_i}(r_i b r_i^*)\, d_i^* \,+\, \sum_{i=1}^m d_i \,\alpha^\KK_{s_i}(r_i b r_i^*)\, c_i^* \,+\, \sum_{i=1}^m C_i \,+\, \sum_{i=1}^m D_i \\
&=\, 2a_p \,+\, \sum_{i=1}^m (C_i + D_i).
\end{align*}
Since $C_i + D_i \ge 0$ for each $i \in \{1, \dotsc, m\}$, it follows that
\[
a_p \le 2a_p \le 2a_p + \sum_{i=1}^m (C_i + D_i) = \sum_{i=1}^m h_i \,\alpha^\KK_{s_i}(r_i b r_i^*)\, h_i^*,
\]
and hence, by \cite[Proposition~2.17]{Thiel2017},
\[
a_p \preceq \sum_{i=1}^m h_i \,\alpha^\KK_{s_i}(r_i b r_i^*)\, h_i^*.
\]
Since $h_i \,\alpha^\KK_{s_i}(r_i b r_i^*)\, h_i^* \preceq \alpha^\KK_{s_i}(r_i b r_i^*)$ for each $i \in \{1, \dots, m\}$, we now have
\begin{align*}
x &\le [a_p]_\Cu \le \Big[ \sum_{i=1}^m h_i \,\alpha^\KK_{s_i}(r_i b r_i^*)\, h_i^* \Big]_\Cu = \sum_{i=1}^m \big[ h_i \,\alpha^\KK_{s_i}(r_i b r_i^*)\, h_i^* \big]_\Cu \\
&\le \sum_{i=1}^m \big[ \alpha^\KK_{s_i}(r_i b r_i^*) \big]_\Cu = \sum_{i=1}^m \hat{\alpha}_{s_i}\big([r_i b r_i^*]_\Cu\big).
\end{align*}
For each $i \in \{1, \dotsc, m\}$, set $z_i \coloneqq [r_i b r_i^*]_\Cu \in \Cu(A_{s_i^*s_i})$. Then
\[
x \le \sum_{i=1}^m \hat{\alpha}_{s_i}(z_i).
\]
For each $i \in \{1, \dotsc, m\}$, we have $r_i b r_i^* \preceq b$, so $z_i = [r_i b r_i^*]_\Cu \le [b]_\Cu = y \in \Cu(A)^\ll$, and thus $z_i \in \Cu(A_{s_i^*s_i})^\ll$, as required. By \cref{cor: precalpha properties}, it follows that
\[
[x]_\alpha \le \Big[\sum_{i=1}^m \hat{\alpha}_{s_i}(z_i)\Big]_\alpha = \sum_{i=1}^m \big[\hat{\alpha}_{s_i}(z_i)\big]_\alpha = \sum_{i=1}^m [z_i]_\alpha \le \sum_{i=1}^m [y]_\alpha = m[y]_\alpha. \qedhere
\]
\end{proof}

\begin{defn}[{\cite[Definition~5.1]{KMP2025}}] \label{defn: regular ideal}
Let $\alpha\colon S \curvearrowright A$ be an action of a unital inverse semigroup on a C*-algebra $A$. We say that an ideal $I$ of the preordered abelian monoid $\Cu(A)_\alpha$ is \emph{regular} if, for all $x \in \Cu(A)^\ll$, the following condition holds:
\[
[f]_\alpha \in I \text{ for all } f \in \Cu(A) \text{ with } f \ll x \implies [x]_\alpha \in I.
\]
\end{defn}

\begin{prop} \label{prop: minimal means DCS is simple}
Let $\alpha\colon S \curvearrowright A$ be a minimal action of a unital inverse semigroup $S$ on a C*-algebra $A$. Then $\Cu(A)_\alpha$ has no nontrivial regular ideals.
\end{prop}

\begin{proof}
Suppose that $J$ is a nonzero regular ideal of $\Cu(A)_\alpha$, and fix $[y]_\alpha \in J {\setminus} \{0\}$. Then $y \in \Cu(A)^\ll {\setminus} \{0\}$. We will show that $[x]_\alpha \in J$ for all $x \in \Cu(A)^\ll$. Fix $x \in \Cu(A)^\ll$, and suppose that $f \in \Cu(A)$ satisfies $f \ll x$. By \cref{prop: Cu minimality}, there exist $\{s_i\}_{i=1}^m \subseteq S$, and $\big\{ z_i \in \Cu(A_{s_i^*s_i})^\ll : i \in \{1, \dotsc, m\} \big\}$ such that $z_i \le y$ for each $i \in \{1, \dotsc, m\}$, and
\[
f \le \sum_{i=1}^m \hat{\alpha}_{s_i}(z_i).
\]
We claim that $f \precalpha my$. To see this, fix $g \in \Cu(A)$ such that $g \ll f$. Then
\[
g \le f \le \sum_{i=1}^m \hat{\alpha}_{s_i}(z_i) \quad \text{ and } \quad \sum_{i=1}^m \hat{\alpha}_{s_i^*}\big(\hat{\alpha}_{s_i}(z_i)\big) = \sum_{i=1}^m z_i \le \sum_{i=1}^m y = my,
\]
so by \cref{lemma: alpha-below characterisation}, $f \precalpha my$. Hence $[f]_\alpha \le m[y]_\alpha$, and since ideals are hereditary submonoids and $[y]_\alpha \in J$, we have $[f]_\alpha \in J$. Since $J$ is regular, it follows that $[x]_\alpha \in J$. Therefore, $J = \Cu(A)_\alpha$, so $\Cu(A)_\alpha$ has no nontrivial regular ideals.
\end{proof}

Using the transitive relation $\pathalpha$, we can define a notion of paradoxicality with respect to our dynamics. This is similar to definitions in \cite{Rainone2017, RS2020, Ma2023}. After the definition, we investigate how non-paradoxical elements give rise to certain functionals on the Cuntz semigroup.

\begin{defn}
Let $\alpha\colon S \curvearrowright A$ be an action of a unital inverse semigroup $S$ on a C*-algebra $A$. Let $M$ be a submonoid of $\Cu(A)$.
\begin{enumerate}[label=(\alph*)]
\item Fix $x \in \Cu(A)$.
\begin{enumerate}[label=(\roman*)]
\item For integers $k > l > 0$, we say that $x$ is \emph{$(k,l)$-paradoxical} if $kx \pathalpha lx$.
\item We say that $x$ is \emph{completely non-paradoxical} if there are no integers $k > l > 0$ such that $x$ is $(k,l)$-paradoxical.
\end{enumerate}
\item We say that the submonoid $M$ is \emph{completely non-paradoxical} if each $x \in M {\setminus} \{0\}$ is completely non-paradoxical as an element of $\Cu(A)$.
\end{enumerate}
\end{defn}

\begin{remark} \label{rem: Cu-paradoxicality}
Let $\alpha\colon S \curvearrowright A$ be an action of a unital inverse semigroup $S$ on a C*-algebra $A$.
\begin{enumerate}[label=(\alph*), ref={\cref{rem: Cu-paradoxicality}(\alph*)}]
\item In terms of the partial order on $\Cu(A)_\alpha$, given integers $k > l > 0$, $x \in \Cu(A)^\ll$ is $(k,l)$-paradoxical if and only if $k[x]_\alpha \le l[x]_\alpha$.
\item \label{item: DCS CNP} If $x \in \Cu(A)^\ll$ is completely non-paradoxical, then $(n+1)[x]_\alpha \not\le n[x]_\alpha$ for all $n \in \N$.
\end{enumerate}
\end{remark}

\begin{thm} \label{thm: OP hom on DCS induces functional}
Let $\alpha\colon S \curvearrowright A$ be an action of a unital inverse semigroup $S$ on a C*-algebra $A$. Every nontrivial order-preserving homomorphism $\nu\colon \Cu(A)_\alpha \to [0,\infty]$ induces an invariant functional $\beta\colon \Cu(A) \to [0,\infty]$ satisfying the following conditions.
\begin{enumerate}[label=(\alph*)]
\item \label{item: finite functional} If $x \in \Cu(A)^\ll$ satisfies $\nu([x]_\alpha) < \infty$, then $\beta(x) < \infty$.
\item \label{item: faithful functional} If $\alpha$ is minimal, then $\beta(x) < \infty$ for each $x \in \Cu(A)^\ll$, and $\beta$ is faithful. Moreover, if $A$ is unital, then $\beta$ can be chosen so that $\beta([1_A]_\Cu) = 1$.
\end{enumerate}
\end{thm}

\begin{proof}
Let $\nu\colon \Cu(A)_\alpha \to [0,\infty]$ be a nontrivial order-preserving homomorphism. Since $\nu$ is nontrivial and order-preserving, $\nu([0]_\alpha) < \infty$, and thus, since $\nu$ is a homomorphism, $\nu([0]_\alpha) = 0$. Fix $x \in \Cu(A)$, and define $\Cu(A)^{\ll x} \coloneqq \{ z \in \Cu(A) : z \ll x \}$. Then $0 \in \Cu(A)^{\ll x}$, so $\Cu(A)^{\ll x}$ is nonempty. Define
\[
\beta(x) \coloneqq \sup\!\big\{ \nu([z]_\alpha) : z \in \Cu(A)^{\ll x} \big\},
\]
and note that it is possible that $\beta(x) = \infty$. We first show that $\beta\colon \Cu(A) \to [0,\infty]$ is a functional. Since $\Cu(A)^{\ll 0} = \{0\}$ and $\nu([0]_\alpha) = 0$, we have $\beta(0) = 0$.

To see that $\beta$ is order-preserving, fix $x, y \in \Cu(A)$ and suppose that $x \le y$. Since $x \le y$, \cref{lemma: ll le properties} implies that $\Cu(A)^{\ll x} \subseteq \Cu(A)^{\ll y}$, and hence $\beta(x) \le \beta(y)$ because the supremum for the latter is over a larger set.

To see that $\beta$ is additive, fix $x, y \in \Cu(A)$. For all $z_1 \in \Cu(A)^{\ll x}$ and $z_2 \in \Cu(A)^{\ll y}$, we have $z_1 + z_2 \in \Cu(A)^{\ll x + y}$ by axiom~\cref{item: O3} of \cref{defn: Cu-semigroup} and $\nu([z_1]_\alpha) + \nu([z_2]_\alpha) = \nu([z_1 + z_2]_\alpha)$, so
\begin{align*}
\beta(x) + \beta(y) &= \sup\!\big\{ \nu([z_1]_\alpha) : z_1 \in \Cu(A)^{\ll x} \big\} \,+\, \sup\!\big\{ \nu([z_2]_\alpha) : z_2 \in \Cu(A)^{\ll y} \big\} \\
&= \sup\!\big\{ \nu([z_1 + z_2]_\alpha) : z_1 \in \Cu(A)^{\ll x}, \, z_2 \in \Cu(A)^{\ll y} \big\} \\
&\le \sup\!\big\{ \nu([z]_\alpha) : z \in \Cu(A)^{\ll x + y} \big\} \\
&= \beta(x + y).
\end{align*}
To see that $\beta(x + y) \le \beta(x) + \beta(y)$, it suffices to show that for all $z \in \Cu(A)^{\ll x + y}$,
\[
\nu([z]_\alpha) \le \sup\!\big\{ \nu([z_1 + z_2]_\alpha) : z_1 \in \Cu(A)^{\ll x}, \, z_2 \in \Cu(A)^{\ll y} \big\}.
\]
Fix $z \in \Cu(A)^{\ll x + y}$. Then $z \ll x + y$, so by \cref{lemma: sequentially way-below interpolation}, there exist $x' \in \Cu(A)^{\ll x}$ and $y' \in \Cu(A)^{\ll y}$ such that $z \ll x' + y' \ll x + y$. Thus, since $\nu$ is order-preserving, we have
\[
\nu([z]_\alpha) \le \nu([x' + y']_\alpha) \le \sup\!\big\{ \nu([z_1 + z_2]_\alpha) : z_1 \in \Cu(A)^{\ll x}, \, z_2 \in \Cu(A)^{\ll y} \big\}.
\]
Hence $\beta(x + y) \le \beta(x) + \beta(y)$, and so $\beta$ is additive.

To see that $\beta$ preserves suprema of increasing sequences, suppose that $\{x_n\}_{n \in \N} \subseteq \Cu(A)$ is an increasing sequence with supremum $x \coloneqq \sup_{n \in \N} (x_n) \in \Cu(A)$. Since $\beta$ is order-preserving, $\{\beta(x_n)\}_{n \in \N} \subseteq [0,\infty]$ is an increasing sequence with upper bound $\beta(x)$, and hence $\sup_{n \in \N} (\beta(x_n)) \le \beta(x)$. For the reverse inequality, it suffices to show that for all $z \in \Cu(A)^{\ll x}$, $\nu([z]_\alpha) \le \sup_{n \in \N} (\beta(x_n))$. Fix $z \in \Cu(A)^{\ll x}$. Then $z \ll x$, so by \cref{lemma: sequentially way-below interpolation}, there exists $u \in \Cu(A)^{\ll x}$ such that $z \ll u \ll x$. Since $u \ll x = \sup_{n \in \N} (x_n)$, there exists $m_u \in \N$ such that $u \le x_{m_u}$. And since $z \in \Cu(A)^{\ll u}$, we have $\nu([z]_\alpha) \le \beta(u)$ by the definition of $\beta$. Thus, since $\beta$ is order-preserving, we have
\[
\nu([z]_\alpha) \le \beta(u) \le \beta(x_{m_u}) \le \sup_{n \in \N} (\beta(x_n)),
\]
and so $\beta(x) = \sup_{n \in \N} (\beta(x_n))$. Therefore, $\beta$ is a functional.

We now show that $\beta$ is invariant. Fix $s \in S$ and $x \in \Cu(A_{s^*s})$. To see that $\beta(x) \le \beta\big(\hat{\alpha}_s(x)\big)$, it suffices to show that $\nu([z]_\alpha) \le \beta\big(\hat{\alpha}_s(x)\big)$ for all $z \in \Cu(A)^{\ll x}$. Fix $z \in \Cu(A)^{\ll x}$. Then $z \le x$, so $z \in \Cu(A_{s^*s})$ because $\Cu(A_{s^*s})$ is hereditary. Since $z \ll x$ and $\hat{\alpha}_s$ is a $\Cu$-morphism, we have $\hat{\alpha}_s(z) \in \Cu(A)^{\ll \hat{\alpha}_s(x)}$. By \cref{item: alpha classes invariant under alpha}, we have $\nu([z]_\alpha) = \nu\big([\hat{\alpha}_s(z)]_\alpha\big) \le \beta\big(\hat{\alpha}_s(x)\big)$, and it follows that $\beta(x) \le \beta\big(\hat{\alpha}_s(x)\big)$. Since $s \in S$ and $x \in \Cu(A_{s^*s})$ were arbitrary, we also have $\beta\big(\hat{\alpha}_s(x)\big) \le \beta\big(\hat{\alpha}_{s^*}\big(\hat{\alpha}_s(x)\big)\big) = \beta(x)$, so $\beta$ is invariant.

For part~\cref{item: finite functional}, suppose that $x \in \Cu(A)^\ll$ satisfies $\nu([x]_\alpha) < \infty$. Since $\nu$ is order-preserving and $z \le x$ for all $z \in \Cu(A)^{\ll x}$, we have
\[
\beta(x) = \sup\{ \nu([z]_\alpha) : z \in \Cu(A)^{\ll x} \} \le \nu([x]_\alpha) < \infty.
\]

For part~\cref{item: faithful functional}, suppose that $\alpha$ is minimal. Since $\nu$ is nontrivial, there exists $y \in \Cu(A)^\ll$ such that $0 < \nu([y]_\alpha) < \infty$. In particular, $y \ne 0$. Fix $x \in \Cu(A)^\ll$. By \cref{prop: Cu minimality}, there exists $m \in \N$ such that $[x]_\alpha \le m[y]_\alpha$. Thus $\nu([x]_\alpha) \le \nu(m[y]_\alpha) = m \, \nu([y]_\alpha) < \infty$, and it follows from part~\cref{item: finite functional} that $\beta(x) < \infty$.

Define
\[
I_\beta \coloneqq \big\{ [w]_\alpha : w \in \Cu(A)^\ll \cap \ker(\beta) \big\}.
\]
We claim that
\begin{equation} \label{eqn: I_beta}
I_\beta = \big\{ [w]_\alpha \in \Cu(A)_\alpha : w' \in \ker(\beta) \text{ for all } w' \in [w]_\alpha \big\}.
\end{equation}
To see that $I_\beta \subseteq \big\{ [w]_\alpha \in \Cu(A)_\alpha : w' \in \ker(\beta) \text{ for all } w' \in [w]_\alpha \big\}$, fix $w \in \Cu(A)^\ll \cap \ker{\beta}$, and suppose that $w' \in [w]_\alpha$. Then $w' \pathalpha w$, so \cref{lemma: functionals preserve precalpha} implies that $\beta(w') \le \beta(w) = 0$. Hence $w' \in \ker(\beta)$, so $I_\beta \subseteq \big\{ [w]_\alpha \in \Cu(A)_\alpha : w' \in \ker(\beta) \text{ for all } w' \in [w]_\alpha \big\}$. The reverse containment is trivial.

We now show that $I_\beta$ is an ideal of $\Cu(A)_\alpha$. To see that $I_\beta$ is a submonoid of $\Cu(A)_\alpha$, first note that $\beta(0) = 0$, so $[0]_\alpha \in I_\beta$. Fix $[x]_\alpha, [y]_\alpha \in I_\beta$. Then \cref{eqn: I_beta} implies that $\beta(x) = \beta(y) = 0$, so $\beta(x + y) = \beta(x) + \beta(y) = 0$, and hence $[x]_\alpha + [y]_\alpha = [x + y]_\alpha \in I_\beta$. Thus $I_\beta$ is a monoid. To see that $I_\beta$ is hereditary, fix $[x]_\alpha \in \Cu(A)_\alpha$ and $[y]_\alpha \in I_\beta$ such that $[x]_\alpha \le [y]_\alpha$. Then $x, y \in \Cu(A)^\ll$, and \cref{eqn: I_beta} implies that $\beta(y) = 0$. Since $x \pathalpha y$, \cref{lemma: functionals preserve precalpha} implies that $\beta(x) \le \beta(y) = 0$, so $x \in \ker(\beta)$. Thus $[x]_\alpha \in I_\beta$, and so $I_\beta$ is hereditary.

We claim that $I_\beta$ is a regular ideal, as defined in \cref{defn: regular ideal}. To see this, fix $x \in \Cu(A)^\ll$, and suppose that for all $w \in \Cu(A)^{\ll x}$, we have $[w]_\alpha \in I_\beta$. We must show that $[x]_\alpha \in I_\beta$, so we will show that $\beta(x) = 0$. For this, it suffices to show that $\nu([v]_\alpha) = 0$ for all $v \in \Cu(A)^{\ll x}$. Fix $v \in \Cu(A)^{\ll x}$. Then by \cref{lemma: sequentially way-below interpolation}, there exists $w \in \Cu(A)$ such that $v \ll w \ll x$. Since $w \in \Cu(A)^{\ll x}$, we have $[w]_\alpha \in I_\beta$ by our assumption, and hence $\beta(w) = 0$ by \cref{eqn: I_beta}. Since $v \in \Cu(A)^{\ll w}$, we have $\nu([v]_\alpha) \le \beta(w) = 0$, so $\nu([v]_\alpha) = 0$, as required. Therefore, $I_\beta$ is a regular ideal of $\Cu(A)_\alpha$.

We now show that $\beta$ is faithful by way of contradiction. Suppose there exists $y \in \Cu(A) {\setminus} \{0\}$ such that $\beta(y) = 0$. Then by \cref{rem: nonzero sequentially way below}, there exists $x \in \Cu(A) {\setminus} \{0\}$ such that $x \ll y$. Since $\beta$ is order-preserving and $x \le y$, we have $\beta(x) \le \beta(y) = 0$. Thus $x \in \Cu(A)^\ll \cap \ker(\beta)$, so $[x]_\alpha \in I_\beta$. Since $x \ne 0$, \cref{item: [0]_alpha singleton} implies that $[x]_\alpha \ne [0]_\alpha$, so $I_\beta \ne \{[0]_\alpha\}$. Since $\alpha$ is minimal, \cref{prop: minimal means DCS is simple} implies that $\Cu(A)_\alpha$ has no nontrivial regular ideals, so $I_\beta = \Cu(A)_\alpha$. Thus \cref{eqn: I_beta} implies that $\beta(w) = 0$ for all $w \in \Cu(A)^\ll$. Now fix $z \in \Cu(A)^\ll$, and let $f \in \Cu(A)$ such that $z \ll f$. By \cref{lemma: sequentially way-below interpolation}, there exists $w \in \Cu(A)^\ll$ such that $z \ll w \ll f$. Thus, by the definition of $\beta$, we have $\nu([z]_\alpha) \le \beta(w) = 0$. But this contradicts our assumption that $\nu$ is nontrivial, so $\beta$ must be faithful.

Finally, suppose that $A$ is unital. Then $[1_A]_\Cu \in \Cu(A)^\ll {\setminus} \{0\}$ by \cref{rem: [p]_Cu ll [p]_Cu}, and hence $0 < \beta([1_A]_\Cu) < \infty$ because $\beta$ is faithful and is finite on $\Cu(A)^\ll$. Thus, scaling $\beta$ by $\frac{1}{\beta([1_A]_\Cu)}$ gives a faithful invariant functional $\beta'\colon \Cu(A) \to [0,\infty]$ that is finite on $\Cu(A)^\ll$ and satisfies $\beta'([1_A]_\Cu) = 1$.
\end{proof}

\begin{thm} \label{thm: non-paradoxical elements generate functionals}
Let $\alpha\colon S \curvearrowright A$ be a minimal action of a unital inverse semigroup $S$ on a C*\nobreakdash-algebra $A$. If $x \in \Cu(A)^\ll {\setminus} \{0\}$ is completely non-paradoxical, then there exists a faithful invariant functional $\beta_x\colon \Cu(A) \to [0,\infty]$ that is finite on $\Cu(A)^\ll$ and satisfies $\beta_x(x) = 1$.
\end{thm}

\begin{proof}
Suppose that $x \in \Cu(A)^\ll {\setminus} \{0\}$ is completely non-paradoxical. Then by \cref{item: DCS CNP}, $(n+1)[x]_\alpha \not\le n[x]_\alpha$ for all $n \in \N$. Thus, by \cref{thm: Tarski-Wehrung paradoxicality thm}, there is an order-preserving homomorphism $\nu\colon \Cu(A)_\alpha \to [0,\infty]$ such that $\nu([x]_\alpha) = 1$. Since $\alpha$ is minimal and $\nu$ is nontrivial, \cref{thm: OP hom on DCS induces functional} implies that $\nu$ induces a faithful invariant functional $\beta\colon \Cu(A) \to [0,\infty]$ that is finite on $\Cu(A)^\ll$. Since $x \in \Cu(A)^\ll {\setminus} \{0\}$ and $\beta$ is faithful, we have $0 < \beta(x) < \infty$, so the map $\beta_x\colon \Cu(A) \to [0,\infty]$ given by $\beta_x(y) \coloneqq \frac{\beta(y)}{\beta(x)}$ is a faithful invariant functional that is finite on $\Cu(A)^\ll$ and satisfies $\beta_x(x) = 1$.
\end{proof}

The proof of \cref{thm: functionals if and only if non-paradoxical} follows the format and proof of \cite[Proposition~4.8]{Rainone2017}.

\begin{cor} \label{thm: functionals if and only if non-paradoxical}
Let $\alpha\colon S \curvearrowright A$ be an action of a unital inverse semigroup $S$ on a unital C*-algebra $A$. Consider the following statements.
\begin{enumerate}[label=(\arabic*)]
\item \label{item: faithful invariant functional beta_x} For each $x \in \Cu(A)^\ll {\setminus} \{0\}$, there is a faithful invariant functional $\beta_x\colon \Cu(A) \to [0,\infty]$ that is finite on $\Cu(A)^\ll$ and satisfies $\beta_x(x) = 1$.
\item \label{item: faithful invariant unital functional} There exists a faithful invariant functional $\beta\colon \Cu(A) \to [0,\infty]$ that is finite on $\Cu(A)^\ll$ and satisfies $\beta([1_A]_\Cu) = 1$.
\item \label{item: functionals Cu(A)^ll CNP} The monoid $\Cu(A)^\ll$ is completely non-paradoxical.
\item \label{item: functionals any CNP element} There exists a completely non-paradoxical element $x \in \Cu(A)^\ll {\setminus} \{0\}$.
\end{enumerate}
Then \cref{item: faithful invariant functional beta_x} $\implies$ \cref{item: faithful invariant unital functional} $\implies$ \cref{item: functionals Cu(A)^ll CNP} $\implies$ \cref{item: functionals any CNP element}. If $\alpha$ is minimal, then \cref{item: functionals any CNP element} $\implies$ \cref{item: faithful invariant functional beta_x}, and \crefrange{item: faithful invariant functional beta_x}{item: functionals any CNP element} are equivalent.
\end{cor}

\begin{proof}
For \cref{item: faithful invariant functional beta_x} $\implies$ \cref{item: faithful invariant unital functional}, note that $[1_A]_\Cu \in \Cu(A)^\ll {\setminus} \{0\}$ by \cref{rem: [p]_Cu ll [p]_Cu}, so we can just apply item~\cref{item: faithful invariant functional beta_x} to $x = [1_A]_\Cu$. For \cref{item: faithful invariant unital functional} $\implies$ \cref{item: functionals Cu(A)^ll CNP}, suppose that $\beta\colon\Cu(A) \to [0,\infty]$ is a faithful invariant functional that is finite on $\Cu(A)^\ll$ and satisfies $\beta([1_A]_\Cu) = 1$. We prove by contradiction that $\Cu(A)^\ll$ is completely non-paradoxical. Suppose that there exists $x \in \Cu(A)^\ll {\setminus} \{0\}$ that is $(k,l)$-paradoxical for integers $k > l > 0$. Then $kx \pathalpha lx$, so by \cref{lemma: functionals preserve precalpha},
\begin{inequality} \label{ineq: kbeta lbeta contradiction}
k \beta(x) = \beta(kx) \le \beta(lx) = l \beta(x).
\end{inequality}
Since $\beta$ is faithful and is finite on $\Cu(A)^\ll$, we have $0 < \beta(x) < \infty$, and so dividing both sides of \cref{ineq: kbeta lbeta contradiction} by $\beta(x)$ gives $k \le l$, which is a contradiction. Thus $\Cu(A)^\ll$ is completely non-paradoxical. That \cref{item: functionals Cu(A)^ll CNP} $\implies$ \cref{item: functionals any CNP element} is trivial.

Now suppose that $\alpha$ is minimal. To see that \cref{item: functionals any CNP element} $\implies$ \cref{item: faithful invariant functional beta_x}, fix $x \in \Cu(A)^\ll {\setminus} \{0\}$. By assumption there exists $y \in \Cu(A)^\ll {\setminus} \{0\}$ such that $y$ is completely non-paradoxical. By \cref{thm: non-paradoxical elements generate functionals}, there exists a faithful invariant functional $\beta'\colon \Cu(A) \to [0,\infty]$ that is finite on $\Cu(A)^\ll$ and satisfies $\beta'(y) = 1$. Since $0 < \beta'(x) < \infty$, we can scale $\beta'$ by $\frac{1}{\beta'(x)}$ to obtain a faithful invariant functional $\beta_x\colon \Cu(A) \to [0,\infty]$ that is finite on $\Cu(A)^\ll$ and satisfies $\beta_x(x) = 1$.
\end{proof}

\Cref{thm: functionals if and only if non-paradoxical} demonstrates that, for minimal actions, we have the following dichotomy: $\Cu(A)^\ll$ is either completely non-paradoxical, or it contains no completely non-paradoxical elements.

\section{Stably finite crossed products}
\label{sec: stably finite}

Let $\alpha\colon S \curvearrowright A$ be an action of a unital inverse semigroup $S$ on a C*-algebra $A$. In \cref{thm: stable finiteness} below, we characterise stable finiteness of $A \rtimesess S$ under reasonable hypotheses. This result is then used in the proof of the stably finite / purely infinite dichotomy in \cref{sec: purely infinite}. A key ingredient of the proof is the construction of tracial states on the crossed product of $A$ by $S$ from tracial states on $A$ that are invariant under the action of $S$. While every trace $\tau\colon A \to \C$ extends to a normal state $\tau''\colon A'' \to \C$ that preserves properties of the action, this extension is rarely faithful, even if $\tau$ is faithful. This suggests that we should work with the essential crossed product rather than the reduced crossed product in order to retain faithfulness. However, the essential crossed product is isomorphic to a subalgebra of the local multiplier algebra $\Mloc(A)$, which frequently admits no normal functionals (see \cite[Theorem~4.2.17]{SaitoWright2015}). In particular, $\tau''$ does not necessarily induce (in any reasonable manner) a tracial state of $\Mloc(A)$. To resolve this issue, we assume that $A \rtimesred S = A \rtimesess S$ as a hypothesis for part of \cref{thm: stable finiteness} below. A similar strategy is employed in \cite{KMP2025}.

Recall that a \emph{tracial state} of $A$ is a state $\tau\colon A \to \C$ that satisfies $\tau(ab) = \tau(ba)$ for all $a, b \in A$. Our first goal is to extend invariant tracial states of $A$ to suitably invariant tracial states of $A''$. Recall that for each $s \in S$, the isomorphism $\alpha_s\colon A_{s^*s} \to A_{ss^*}$ has a unique weakly continuous surjective extension $\alpha_s''\colon A_{s^*s}'' \to A_{ss^*}''$ (see \cite[Corollary 1.2.26]{AraMathieu2003}). It follows from the uniqueness of these extensions that each $\alpha_s''\colon A_{s^*s}'' \to A_{ss^*}''$ is an isomorphism, and since $A_{s^*s}''$ is an ideal of $A''$ for each $s \in S$ (by \cref{lemma: inclusions of ideals}), it follows that $\alpha''\colon S \curvearrowright A''$ is an action of $S$ on $A''$.

\begin{defn}
Let $\alpha\colon S \curvearrowright A$ be an action of an inverse semigroup $S$ on a C*-algebra $A$. A tracial state $\tau\colon A \to \C$ is \emph{invariant} if $\tau(\alpha_s(a)) = \tau(a)$ for all $s \in S$ and $a \in A_{s^*s}$.
\end{defn}

\begin{lemma} \label{lemma: tracial states extend to bicommutant}
Let $\alpha\colon S \curvearrowright A$ be an action of an inverse semigroup $S$ on a C*-algebra $A$, and for each $s \in S$, let $\alpha_s''\colon A_{s^*s}'' \to A_{ss^*}''$ be the unique weakly continuous extension of $\alpha_s$. Suppose that $\tau\colon A \to \C$ is an invariant tracial state of $A$. Then there exists a unique normal extension $\tau''\colon A'' \to \C$ of $\tau$ such that $\tau''$ is a tracial state and $\tau''$ is invariant in the sense that $\tau''(\alpha_s''(a)) = \tau''(a)$ for all $s \in S$ and $a \in A_{s^*s}''$.
\end{lemma}

\begin{proof}
Let $\pi_\tau\colon A \to B(H_\tau)$ be the GNS representation associated to $\tau$. By \cite[Theorem~5.1.1]{Murphy1990} there is a unique unit cyclic vector $x_\tau \in H_\tau$ such that $\tau(a) = \big(\pi_\tau(a)x_\tau \ipmid x_\tau\big)$. By \cite[Theorem~3.7.7]{Pedersen2018} there is a unique normal extension $\pi_\tau''\colon A'' \to B(H_\tau)$. Then by \cite[Theorem~5.1.7]{Murphy1990}, $\tau''(a) \coloneqq \big(\pi_\tau''(a)x_\tau \ipmid x_\tau\big)$ defines a state of $A''$ that is normal because $\pi_\tau''$ is normal.

We claim that $\tau''$ is tracial and invariant. Let $\pi_\rmU\colon A \to B(H_\rmU)$ be the universal representation of $A$. Then $A'' = \pi_\rmU(A)''$, and $A \cong \pi_\rmU(A)$. Since $A$ is unital, $A''$ is the closure of $A$ in the weak operator topology by \cite[Theorem~4.2.5]{Murphy1990}. Fix $a, b \in A''$. Then there are nets $\{a_i\}_i,\{b_j\}_j \subseteq A$ that converge to $a$ and $b$, respectively, in the weak operator topology. Now
\begin{align*}
\tau''(ab) &= \big(\pi_\tau''(ab)x_\tau \ipmid x_\tau\big) = \big(\pi_\tau''(b)x_\tau \ipmid \pi_\tau''(a^*)x_\tau\big) = \lim_i \lim_j \big(\pi_\tau''(b_j)x_\tau \ipmid \pi_\tau''(a_i^*)x_\tau\big) \\
&= \lim_i \lim_j \big(\pi_\tau(b_j)x_\tau \ipmid \pi_\tau(a_i^*)x_\tau\big) = \lim_i \lim_j \big( \pi_\tau(a_i b_j) x_\tau \ipmid x_\tau \big) = \lim_i \lim_j \tau(a_i b_j) \\
&= \lim_i \lim_j \tau(b_j a_i) = \big( \pi_\tau(b_j a_i) x_\tau \ipmid x_\tau \big) = \lim_i \lim_j \big(\pi_\tau(a_i)x_\tau \ipmid \pi_\tau(b_j^*)x_\tau\big) \\
&= \lim_j \lim_i \big(\pi_\tau''(a_i)x_\tau \ipmid \pi_\tau''(b_j^*)x_\tau\big) = \big(\pi_\tau''(a)x_\tau \ipmid \pi_\tau''(b^*)x_\tau\big) = \big(\pi_\tau''(ba)x_\tau \ipmid x_\tau\big) = \tau''(ab).
\end{align*}
Thus $\tau''$ is a tracial state. For the invariance, fix $s \in S$ and $a \in A_{s^*s}''$. Then there is a net $\{a_i\}_i \subseteq A_{s^*s}$ that converges weakly to $a$. Hence
\[
\tau''(\alpha_s''(a)) = \lim_i \tau''(\alpha_s''(a_i)) = \lim_i \tau(\alpha_s(a_i)) = \lim_i \tau(a_i) = \tau''(a). \qedhere
\]
\end{proof}

We now recall two definitions and a collection of results from \cite{KwasniewskiMeyer2021} in order to clarify the hypotheses in our main theorems that follow.

\begin{defn}
Let $A \hookrightarrow B$ be an inclusion of C*-algebras. We say that \emph{$A$ detects ideals in $B$} if $J \cap A \ne \{0\}$ for any nonzero ideal $J$ of $B$.
\end{defn}

\begin{defn}[{see \cite[Definition~6.1]{KwasniewskiMeyer2021}}]
Let $\alpha\colon S \curvearrowright A$ be an action of a unital inverse semigroup $S$ on a C*-algebra $A$. Then $\alpha$ is \emph{aperiodic} if for all $s \in S$, $a \in I_{s,1}^\perp \cap A_{ss^*}$, $\varepsilon > 0$, and any nonzero hereditary subalgebra $D \subseteq A$, there exists $d \in D^+$ such that $\norm{d} = 1$ and $\norm{d \, a \delta_s \, d} \le \varepsilon$.
\end{defn}

\begin{remark} \label{rem: simplicity and aperiodic actions and detecting ideals}
Let $\alpha\colon S \curvearrowright A$ be an action of a unital inverse semigroup $S$ on a C*-algebra $A$.
\begin{enumerate}[label=(\alph*)]
\item \label{item: aperiodic implies simple iff minimal} If $\alpha$ is non-minimal, then \cite[Theorems~6.5 and 6.6]{KwasniewskiMeyer2021} together imply that $A \rtimesess S$ is not simple. When $\alpha$ is minimal, the precise conditions on $\alpha$ under which $A \rtimesess S$ is simple are not known in general, but if $\alpha$ is aperiodic, then $A \rtimesess S$ is simple if and only if $\alpha$ is minimal.
\item \label{item: simple iff minimal and detects ideals} By \cite[Theorem~6.6]{KwasniewskiMeyer2021}, $A \rtimesess S$ (or, respectively, $A \rtimesred S$) is simple if and only if $\alpha$ is minimal and $A$ detects ideals in $A \rtimesess S$ (or, respectively, in $A \rtimesred S$).
\item \label{item: detecting ideals implies injective quotient map} If $A$ detects ideals in $A \rtimesred S$, then the quotient map $\Lambda\colon A \rtimesred S \to A \rtimesess S$ is injective, because $\ker(\Lambda) \cap A = \{0\}$, and hence $\ker(\Lambda) = \{0\}$.
\item \label{item: aperiodic implies detects ideals} If $\alpha$ is aperiodic, then $A$ detects ideals in $A \rtimesess S$ by \cite[Theorem~6.5(3)]{KwasniewskiMeyer2021}. Thus, if $\alpha$ is aperiodic and $\Lambda\colon A \rtimesred S \to A \rtimesess S$ is injective, then $A$ detects ideals in $A \rtimesred S$.
\end{enumerate}
\end{remark}

\begin{thm} \label{thm: traces reduced}
Let $\alpha\colon S \curvearrowright A$ be an action of a unital inverse semigroup $S$ on a unital C*\nobreakdash-algebra $A$. The weak conditional expectation $\ER\colon A \rtimesfull S \to A''$ defined in \cref{thm: ER existence and properties} induces a positive linear contraction $\tilde{\ER}\colon A \rtimesred S \to A''$ such that $\tilde{\ER}\restr{A \rtimesalg S} = \ER\restr{A \rtimesalg S}$, and there are one-to-one correspondences between
\begin{enumerate}[label=(\alph*)]
\item \label{item: traces reduced invariant} invariant tracial states $\tau\colon A \to \C$;
\item \label{item: traces reduced normal} normal invariant tracial states $\tau''\colon A'' \to \C$; and
\item \label{item: traces reduced expectation} tracial states $\tau_\red\colon A \rtimesred S \to \C$ such that $\tau_\red = \tau''\circ \tilde{\ER}$ for some normal tracial state $\tau''\colon A'' \to \C$.
\end{enumerate}
Moreover, if $A$ is exact, then the above are also in one-to-one correspondence with
\begin{enumerate}[resume, label=(\alph*)]
\item invariant functionals $\beta\colon \Cu(A) \to [0,\infty]$ satisfying $\beta([1_A]_\Cu) = 1$.
\end{enumerate}

In particular, if $A$ is exact, and $\beta\colon \Cu(A) \to [0,\infty]$ is a faithful invariant functional satisfying $\beta([1_A]_\Cu) = 1$, then the induced invariant tracial state $\tau_\beta\colon A \to \C$ is faithful.

Finally, suppose that $A$ detects ideals in $A \rtimesred S$, and let $\tau$ be a faithful invariant tracial state of $A$. Then the singular ideal of $A \rtimesred S$ is trivial, and $\tau_\red$ is faithful on $A \rtimesred S = A \rtimesess S$.
\end{thm}

\begin{proof}
\Cref{lemma: tracial states extend to bicommutant} shows that each invariant tracial state $\tau\colon A \to \C$ induces a unique normal invariant tracial state $\tau''\colon A'' \to \C$. Moreover, if $\tau''\colon A'' \to \C$ is a normal invariant tracial state, then $\tau''\restr{A}$ is an invariant tracial state of $A$.

Since $A \rtimesred S = (A \rtimesfull S) / \NN_\ER$, the existence of the positive linear contraction $\tilde{\ER}$ that coincides with $\ER$ on $A \rtimesalg S$ follows immediately from \cref{thm: ER existence and properties}. Thus, each normal invariant tracial state $\tau''\colon A'' \to \C$ induces a state $\tau_\red = \tau'' \circ \tilde{\ER}$. To see that $\tau_\red$ is tracial, fix $s, t \in S$, $a \in A_{ss^*}$, and $b \in A_{tt^*}$. Recall from \cref{item: 1_s is central} that $1_{ts}$ is central in $A''$. Thus
\begin{align*}
\tau_\red(a\delta_s * b\delta_t) &= \tau''\big(\tilde{\ER}(a \delta_s * b\delta_t)\big) = \tau''\big(\alpha_s(\alpha_{s^*}(a) \, b) \, 1_{st}\big) \\
&= \tau''\big(\alpha_{s^*}''\big(\alpha_s(\alpha_{s^*}(a) \, b) \, 1_{st}\big)\big) \qquad \text{ (by the invariance of $\tau''$}) \\
&= \tau''\big(\alpha_{s^*}(a) \,b \, 1_{ts}\big) \qquad \text{ (since $\alpha_{s^*}''(1_{st}) = 1_{ts}$ by \cref{lemma: I_1st iso to I_1ts})} \\
&= \tau''\big(b \, \alpha_{s^*}(a) \, 1_{ts}\big) \qquad \text{ (since $\tau''$ is tracial and $1_{ts}$ is central in $A''$)} \\
&= \tau''\big(b \, \alpha_{s^*}''(a \, 1_{st})\big) \qquad \text{ (since $\alpha_{s^*}''(1_{st}) = 1_{ts}$ by \cref{lemma: I_1st iso to I_1ts})} \\
&= \tau''\big(b \, \alpha_t''(a \, 1_{st})\big) \qquad \text{ (since $\alpha_{s^*}''(a \, 1_{st}) = \alpha_t''(a \, 1_{st})$ by \cref{lemma: I_1st iso to I_1ts})} \\
&= \tau''\big(\alpha_{tt^*}''(b) \, \alpha_t''(a \, 1_{st})\big) \\
&= \tau''\big(\alpha_t''(\alpha_{t^*}(b) \, a) \, \alpha_t''(1_{st})\big) \\
&= \tau''\big(\alpha_t(\alpha_{t^*}(b) \, a) \, 1_{ts} \big) \qquad \text{ (since $\alpha_t''(1_{st}) = 1_{ts}$ by \cref{lemma: I_1st iso to I_1ts})} \\
&= \tau''\big(\tilde{\ER}(b\delta_t * a\delta_s)\big) = \tau_\red(b \delta_t * a\delta_s),
\end{align*}
and so $\tau_\red$ is tracial.

Now suppose that $\tau_\red\colon A \rtimesred S \to \C$ is a tracial state such that $\tau_\red = \tau'' \circ \tilde{\ER}$ for some normal tracial state $\tau''\colon A'' \to \C$. Then $\tau_\red \restr{A} = \tau'' \circ \tilde{\ER}\restr{A} = \tau''\restr{A}$ is a tracial state of $A$. We claim that $\tau_\red\restr{A}$ is invariant. To see this, fix $s \in S$ and $a \in A_{s^*s}$, and let $\{e_\lambda\}_\lambda$ be an approximate identity for $A_{ss^*}$. By \cref{lemma: alpha implemented by conjugation} we have
\begin{equation} \label{eqn: alpha_s(a) conjugation}
a_s(a) \delta_{ss^*} = \lim_\lambda \big( e_\lambda \delta_s * a \delta_1 * \alpha_{s^*}(e_\lambda) \delta_{s^*} \big).
\end{equation}
Therefore,
\begin{align*}
\tau_\red\restr{A}\big(\alpha_s(a) \delta_1\big) &= \tau_\red\restr{A}\big(\alpha_s(a) \delta_{ss^*}\big) \qquad \text{ (since $\alpha_s(a) \in A_{ss^*}$ and $ss^* \le 1$)} \\
&= \lim_\lambda \, \tau_\red\big( e_\lambda \delta_s * a \delta_1 * \alpha_{s^*}(e_\lambda) \delta_{s^*} \big) \qquad \text{ (by \cref{eqn: alpha_s(a) conjugation})} \\
&= \lim_\lambda \, \tau_\red\big( a \delta_1 * \alpha_{s^*}(e_\lambda) \delta_{s^*} * e_\lambda \delta_s \big) \qquad \text{ (since $\tau_\red$ is tracial)} \\
&= \lim_\lambda \, \tau_\red\big( a \delta_1 * \alpha_{s^*}\big( \alpha_s(\alpha_{s^*}(e_\lambda)) \, e_\lambda \big) \delta_{s^*s} \big) \\
&= \lim_\lambda \, \tau_\red\big( a \delta_1 * \alpha_{s^*}(e_\lambda^2) \delta_{s^*s} \big) \\
&= \lim_\lambda \, \tau_\red\big( a \, \alpha_{s^*}(e_\lambda^2) \, \delta_{s^*s} \big) \\
&= \tau_\red\big( \!\lim_\lambda \, a \, \alpha_{s^*}(e_\lambda^2) \, \delta_{s^*s} \big) \\
&= \tau_\red\restr{A}(a \delta_{s^*s}) \qquad \text{ (since $\big\{ \alpha_{s^*}(e_\lambda^2) \big\}_\lambda$ is an approximate identity for $A_{s^*s}$)} \\
&= \tau_\red\restr{A}(a \delta_1) \qquad \text{ (since $s^*s \le 1$)}.
\end{align*}
Thus $\tau_\red\restr{A}$ is an invariant tracial state of $A$.

Therefore, the correspondences between \cref{item: traces reduced invariant,item: traces reduced normal,item: traces reduced expectation} are given by
\[
\tau \mapsto \tau'' \mapsto \tau_\red = \tau'' \circ \tilde{\ER} \mapsto \tau_\red \restr{A}.
\]

For the ``moreover'' statement, suppose that $A$ is exact. Since normalised quasitraces on unital exact C*\nobreakdash-algebras are tracial states by \cite[Theorem~5.11]{Haagerup2014}, and since quasitraces on $A$ are in one-to-one correspondence with quasitraces on $A \otimes \KK$, \cref{prop: quasitraces and functionals} gives a one-to-one correspondence $\beta \mapsto \tau_\beta$ between functionals $\beta\colon \Cu(A) \to [0,\infty]$ satisfying $\beta([1_A]_\Cu) = 1$ and tracial states $\tau_\beta\colon A \to \C$ such that $\tau_\beta$ is invariant if and only if $\beta$ is invariant. If $\beta$ is faithful, then \cref{prop: quasitraces and functionals}\cref{item: qtf faithfulness} implies that $\tau_\beta$ is also faithful.

Finally, suppose that $A$ detects ideals in $A \rtimesred S$, and let $\tau\colon A \to \C$ be a faithful invariant tracial state. Then the quotient map $\Lambda\colon A \rtimesred S \to A \rtimesess S$ is injective by \cref{rem: simplicity and aperiodic actions and detecting ideals}\cref{item: detecting ideals implies injective quotient map}, so the singular ideal $\ker(\Lambda)$ is trivial, and $A \rtimesred S = A \rtimesess S$. Let $\tau_\red\colon A \rtimesred S \to \C$ be the induced tracial state satisfying $\tau_\red\restr{A} = \tau$. Let $I_\red \coloneqq \{ x \in A \rtimesred S : \tau_\red(x^*x) = 0 \}$, which is an ideal of $A \rtimesred S$. Then $\tau_\red$ is faithful if and only if $I_\red = \{0\}$. Since $A$ detects ideals in $A \rtimesred S$, to see that $I_\red = \{0\}$, it suffices to show that $A \cap I_\red = \{0\}$. Fix $a \in A \cap I_\red$. Since $\tau_\red\restr{A} = \tau$, we have $\tau(a^*a) = \tau_\red(a^*a) = 0$, and hence $a = 0$ because $\tau$ is faithful. Thus $A \cap I_\red = \{0\}$, and so $\tau_\red$ is faithful.
\end{proof}

\begin{remark}
\Cref{thm: traces reduced} is analogous to \cite[Theorem~3.14]{KMP2025} for groupoid C*-algebras. In the proof of \cref{thm: traces reduced} we required $A$ to detect ideals of $A \rtimesred S$ in order to construct a faithful tracial state of $A \rtimesred S$ from a faithful invariant tracial state of $A$. In \cite[Theorem~3.14]{KMP2025}, this condition is not needed. We have been unable to remove this assumption in the noncommutative setting, so it also appears in the main theorem of this section (\cref{thm: stable finiteness}), and in the corresponding result in the commutative setting (\cref{cor: stable finiteness groupoid}). The dichotomy results in \cite{KMP2025} require aperiodicity, which implies the detecting ideals condition (see \cref{rem: simplicity and aperiodic actions and detecting ideals}\cref{item: aperiodic implies detects ideals}).
\end{remark}

We now characterise stable finiteness of $A \rtimesred S = A \rtimesess S$ under certain hypotheses.

\begin{thm} \label{thm: stable finiteness}
Let $\alpha\colon S \curvearrowright A$ be an action of a unital inverse semigroup $S$ on a unital C*-algebra $A$. Consider the following statements.
\begin{enumerate}[label=(\arabic*)]
\item \label{item: A faithful trace} There exists a faithful invariant tracial state of $A$.
\item \label{item: A rtimes S ess faithful trace} There exists a faithful tracial state of $A \rtimesess S$.
\item \label{item: A rtimes S ess stably finite} The C*-algebra $A \rtimesess S$ is stably finite.
\item \label{item: compacts are CNP} The monoid $\Cuc(A)$ is completely non-paradoxical.
\item \label{item: Cu(A)^ll CNP} The monoid $\Cu(A)^\ll$ is completely non-paradoxical.
\item \label{item: any CNP element} There exists $x \in \Cu(A)^\ll {\setminus} \{0\}$ that is completely non-paradoxical.
\item \label{item: OP hm} There exists a nontrivial order-preserving homomorphism $\nu\colon \Cu(A)_\alpha \to [0,\infty]$.
\item \label{item: faithful invariant functional} There exists a faithful invariant functional $\beta\colon \Cu(A) \to [0,\infty]$ that is finite on $\Cu(A)^\ll$ and satisfies $\beta([1_A]_\Cu) = 1$.
\end{enumerate}
The following implications hold.
\begin{enumerate}[label=(\alph*)]
\item \label{item: stable finiteness all} In all cases, \cref{item: A rtimes S ess faithful trace} $\implies$ \cref{item: A rtimes S ess stably finite} $\implies$ \cref{item: compacts are CNP} and \cref{item: Cu(A)^ll CNP} $\implies$ \cref{item: any CNP element} $\implies$ \cref{item: OP hm}.
\item \label{item: stable finiteness detects ideals} If $A$ detects ideals in $A \rtimesred S$, then \cref{item: A faithful trace} $\implies$ \cref{item: A rtimes S ess faithful trace}.
\item \label{item: stable finiteness minimal} If $\alpha$ is minimal, then \cref{item: compacts are CNP} $\implies$ \cref{item: Cu(A)^ll CNP} and \cref{item: OP hm} $\implies$ \cref{item: faithful invariant functional}.
\item \label{item: stable finiteness exact} If $A$ is exact, then \cref{item: faithful invariant functional} $\implies$ \cref{item: A faithful trace}.
\end{enumerate}
In particular, if we assume that $A$ is exact and $A \rtimesred S$ is simple, then \crefrange{item: A faithful trace}{item: faithful invariant functional} are equivalent.
\end{thm}

\begin{proof}
We start by proving the implications in part~\cref{item: stable finiteness all}. That \cref{item: A rtimes S ess faithful trace} $\implies$ \cref{item: A rtimes S ess stably finite} is general theory: any unital C*-algebra that admits a faithful tracial state is stably finite (see, for example, the outline in \cite[Exercise~5.2]{RLL2000}).

To see that \cref{item: A rtimes S ess stably finite} $\implies$ \cref{item: compacts are CNP}, suppose that $A \rtimesess S$ is stably finite. Then by \cref{item: stably finite C*}, every projection in $(A \rtimesess S) \otimes \KK$ is finite. Suppose for contradiction that $\Cuc(A)$ is not completely non-paradoxical. Then there exist integers $k > l > 0$ and an element $x \in \Cuc(A) {\setminus} \{0\}$ that is $(k,l)$-paradoxical, so $kx \pathalpha lx$. Let $\hat{\iota}\colon\Cu(A) \to \Cu(A \rtimesess S)$ be the $\Cu$-morphism from \cref{lemma: Cu-embedding in crossed product is invariant} that is induced by the canonical embedding $A \owns a \mapsto a\delta_1 \in A \rtimesess S$ and has trivial kernel. Then $\hat{\iota}$ maps compact elements to compact elements by \cref{thm: homomorphism induces Cu-morphism}, so $\hat{\iota}(x) \in \Cuc(A \rtimesess S) {\setminus} \{0\}$. Since $kx \pathalpha lx$, \cref{cor: iota hat preserves precalpha} implies that
\[
k \hat{\iota}(x) = \hat{\iota}(kx) \le \hat{\iota}(lx) = l \hat{\iota}(x).
\]
But $k > l$, so $(k - l) \hat{\iota}(x) \in \Cuc(A \rtimesess S) {\setminus} \{0\}$, and we have
\[
(k - l) \hat{\iota}(x) + l \hat{\iota}(x) = k \hat{\iota}(x) \le l \hat{\iota}(x).
\]
Thus $l \hat{\iota}(x)$ is an infinite element of $\Cuc(A \rtimesess S)$, in the sense of \cref{item: infinite element of monoid}. Since $A \rtimesess S$ is stably finite, $\Cuc(A \rtimesess S) \cong V(A \rtimesess S)$ by \cite[Corollary~3.3]{BrownCiuperca2009} (as cited in \cite[Remark~3.5]{GardellaPerera2024}), so $l\hat{\iota}(x)$ is the Cuntz-equivalence class of a projection. But now $A \rtimesess S$ contains an infinite projection (in the sense of \cref{item: infinite element of C*-algebra}), which is a contradiction, because $A \rtimesess S$ is stably finite. Thus $\Cuc(A)$ is completely non-paradoxical.

That \cref{item: Cu(A)^ll CNP} $\implies$ \cref{item: any CNP element} is trivial.

To see that \cref{item: any CNP element} $\implies$ \cref{item: OP hm}, suppose that $x \in \Cu(A)^\ll {\setminus} \{0\}$ is completely non-paradoxical. Then by \cref{item: DCS CNP}, $(n+1)[x]_\alpha \not\le n[x]_\alpha$ for all $n \in \N$. Now \cref{thm: Tarski-Wehrung paradoxicality thm} implies that there exists an order-preserving homomorphism $\nu\colon \Cu(A)_\alpha \to [0,\infty]$ such that $\nu([x]_\alpha) = 1 \in (0,\infty)$, giving \cref{item: OP hm}.

For part~\cref{item: stable finiteness detects ideals}, suppose that $A$ detects ideals in $A \rtimesred S$. Then \cref{thm: traces reduced} gives \cref{item: A faithful trace} $\implies$ \cref{item: A rtimes S ess faithful trace}.

For part~\cref{item: stable finiteness minimal}, suppose that $\alpha$ is minimal. To see that \cref{item: compacts are CNP} $\implies$ \cref{item: Cu(A)^ll CNP}, suppose that $\Cuc(A)$ is completely non-paradoxical. Since $A$ is unital, \cref{rem: [p]_Cu ll [p]_Cu} implies that $[1_A]_\Cu \in \Cuc(A) {\setminus} \{0\}$, so $[1_A]_\Cu$ is completely non-paradoxical. Since $[1_A]_\Cu \in \Cu(A)^\ll {\setminus} \{0\}$, \cref{thm: functionals if and only if non-paradoxical} implies that $\Cu(A)^\ll$ is completely non-paradoxical, giving \cref{item: Cu(A)^ll CNP}.

To see that \cref{item: OP hm} $\implies$ \cref{item: faithful invariant functional}, suppose that $\nu\colon \Cu(A)_\alpha \to [0,\infty]$ is a nontrivial order-preserving homomorphism. Since $\alpha$ is minimal, \cref{thm: OP hom on DCS induces functional} implies that there exists a faithful invariant functional $\beta\colon \Cu(A) \to [0,\infty]$ that is finite on $\Cu(A)^\ll$ and satisfies $\beta([1_A]_\Cu) = 1$, giving \cref{item: faithful invariant functional}.

For part~\cref{item: stable finiteness exact}, assume that $A$ is exact. To see that \cref{item: faithful invariant functional} $\implies$ \cref{item: A faithful trace}, suppose there exists a faithful invariant functional $\beta\colon \Cu(A) \to [0,\infty]$ that is finite on $\Cu(A)^\ll$ and satisfies $\beta([1_A]_\Cu) = 1$. Since $A$ is exact, \cref{thm: traces reduced} implies that $A$ admits a faithful invariant tracial state.

Finally, suppose that $A$ is exact and $A \rtimesred S$ is simple. Then by \cref{rem: simplicity and aperiodic actions and detecting ideals}\cref{item: simple iff minimal and detects ideals}, $\alpha$ is minimal and $A$ detects ideals in $A \rtimesred S$. It follows immediately that \crefrange{item: A faithful trace}{item: faithful invariant functional} are equivalent.
\end{proof}

\section{Purely infinite crossed products and the dichotomy}
\label{sec: purely infinite}

The following theorem is the analogue of \cref{thm: stable finiteness} in the purely infinite case.

\begin{thm} \label{thm: pure infiniteness}
Let $\alpha\colon S \curvearrowright A$ be an action of a unital inverse semigroup $S$ on a unital and exact C*-algebra $A$ such that $A \rtimesred S$ is simple. Then $\alpha$ is minimal, $A$ detects ideals in $A \rtimesred S$, and $A \rtimesred S = A \rtimesess S$. Consider the following statements.
\begin{enumerate}[label=(\arabic*)]
\item \label{item: PIS PIM} The dynamical Cuntz semigroup $\Cu(A)_\alpha$ is purely infinite.
\item \label{item: PIS 21P} Every element of $\Cu(A)^\ll {\setminus} \{0\}$ is $(2,1)$-paradoxical.
\item \label{item: PIS PIC*} The C*-algebra $A \rtimesess S$ is purely infinite.
\item \label{item: PIS NTS} The C*-algebra $A \rtimesess S$ admits no tracial states.
\item \label{item: PIS OPH} There does not exist a nontrivial order-preserving homomorphism $\nu\colon \Cu(A)_\alpha \to [0,\infty]$.
\end{enumerate}
Then \cref{item: PIS PIM} $\implies$ \cref{item: PIS 21P} $\implies$ \cref{item: PIS PIC*} $\implies$ \cref{item: PIS NTS} $\implies$ \cref{item: PIS OPH}. If the dynamical Cuntz semigroup $\Cu(A)_\alpha$ has plain paradoxes, then \cref{item: PIS OPH} $\implies$ \cref{item: PIS PIM}, and \crefrange{item: PIS PIM}{item: PIS OPH} are equivalent.
\end{thm}

\begin{proof}
Since $A \rtimesred S$ is simple, parts~\cref{item: simple iff minimal and detects ideals,item: detecting ideals implies injective quotient map} of \cref{rem: simplicity and aperiodic actions and detecting ideals} together imply that $\alpha$ is minimal, $A$ detects ideals in $A \rtimesred S$, and the singular ideal of $A \rtimesred S$ is trivial, so $A \rtimesred S = A \rtimesess S$.

To see that \cref{item: PIS PIM} $\implies$ \cref{item: PIS 21P}, fix $x \in \Cu(A)^\ll {\setminus} \{0\}$. By assumption, $\Cu(A)_\alpha$ is purely infinite, and hence $[x]_\alpha$ is properly infinite; that is, $2[x]_\alpha \le [x]_\alpha$. Thus, by definition of the partial order on $\Cu(A)_\alpha$, we have $2x \pathalpha x$, so $x$ is $(2,1)$-paradoxical.

To see that \cref{item: PIS 21P} $\implies$ \cref{item: PIS PIC*}, suppose that every element of $\Cu(A)^\ll {\setminus} \{0\}$ is $(2,1)$-paradoxical. Let $\iota\colon A \to A \rtimesess S$ be the canonical embedding $a \mapsto a\delta_1$, and let $\hat{\iota}\colon\Cu(A) \to \Cu(A \rtimesess S)$ be the induced $\Cu$-morphism from \cref{lemma: Cu-embedding in crossed product is invariant}. By \cite[Corollary~6.7]{KwasniewskiMeyer2021}, $A \rtimesess S$ is purely infinite if and only if every element of $A^+ {\setminus} \{0\}$ is infinite in $A \rtimesess S$, in the sense of \cref{item: infinite element of C*-algebra}. Fix $a \in A^+ {\setminus} \{0\}$. By axiom~\cref{item: O2} of \cref{defn: Cu-semigroup}, there is a sequence $\{x_n\}_{n \in \N} \subseteq \Cu(A) {\setminus} \{0\}$ such that $x_n \ll x_{n+1}$ for all $n \in \N$ and $[a]_\Cu = \sup_{n \in \N} (x_n)$. By assumption, each $x_n$ is $(2,1)$-paradoxical, so $2x_n \pathalpha x_n$ for all $n \in \N$. It follows by \cref{cor: iota hat preserves precalpha} that $\hat{\iota}(2x_n) \le \hat{\iota}(x_n)$ for all $n \in \N$. Since $\hat{\iota}$ is a $\Cu$-morphism, it preserves suprema of increasing sequences, and therefore,
\begin{align*}
2[\iota(a)]_\Cu &= 2\hat{\iota}([a]_\Cu) = \hat{\iota}(2[a]_\Cu) = \hat{\iota}\big( 2 \sup_{n \in \N} (x_n) \big) = \sup_{n \in \N} \big(\hat{\iota}(2x_n)\big) \\
&\le \sup_{n \in \N} \big(\hat{\iota}(x_n)\big) = \hat{\iota}\big( \sup_{n \in \N} (x_n) \big) = \hat{\iota}([a]_\Cu) = [\iota(a)]_\Cu.
\end{align*}
Since $\iota$ is injective, $[\iota(a)]_\Cu \ne 0$, so $\iota(a)$ is infinite (in fact, properly infinite) in $A \rtimesess S$. Thus $A \rtimesess S$ is purely infinite.

That \cref{item: PIS PIC*} $\implies$ \cref{item: PIS NTS} is immediate because purely infinite C*-algebras do not admit tracial states (see, for example, \cite[Proposition~V.2.2.29]{Blackadar2006}).

To see that \cref{item: PIS NTS} $\implies$ \cref{item: PIS OPH}, we prove the contrapositive. Suppose there exists a nontrivial order-preserving homomorphism $\nu\colon \Cu(A)_\alpha \to [0,\infty]$. Then since $\alpha$ is minimal, we can apply \cref{thm: OP hom on DCS induces functional} to find an invariant functional $\beta\colon \Cu(A) \to [0,\infty]$ such that $\beta([1_A]_\Cu) = 1$. Then since $A$ is exact and $A$ detects ideals in $A \rtimesred S$, \cref{thm: traces reduced} implies that $A \rtimesred S = A \rtimesess S$ admits a tracial state.

Suppose now that $\Cu(A)_\alpha$ has plain paradoxes. Since $A$ is unital, we have $\Cu(A)_\alpha \ne \{0\}$ by \cref{item: [0]_alpha singleton}. To see that \cref{item: PIS OPH} $\implies$ \cref{item: PIS PIM}, assume that \cref{item: PIS OPH} holds, and fix $x \in \Cu(A)^\ll$. If $x = 0$, then $2[x]_\alpha = [0]_\alpha \le [x]_\alpha$. Suppose that $x \ne 0$. Then there must exist $n \in \N$ such that $(n+1)[x]_\alpha \le n[x]_\alpha$, because if not, then \cref{thm: Tarski-Wehrung paradoxicality thm} would produce a nontrivial order-preserving homomorphism $\nu\colon \Cu(A)_\alpha \to [0,\infty]$ such that $\nu([x]_\alpha) = 1$, contradicting \cref{item: PIS OPH}. Since $\Cu(A)_\alpha$ has plain paradoxes, it follows that $2[x]_\alpha \le [x]_\alpha$. In particular, every element of $\Cu(A)_\alpha {\setminus} \{0\}$ is properly infinite, and thus $\Cu(A)_\alpha$ is purely infinite, giving \cref{item: PIS PIM}.
\end{proof}

\begin{thm} \label{thm: dichotomy result}
Let $\alpha\colon S \curvearrowright A$ be an action of a unital inverse semigroup $S$ on a unital and exact C*-algebra $A$. Suppose that $A \rtimesred S$ is simple, and that the dynamical Cuntz semigroup $\Cu(A)_\alpha$ (as defined in \cref{defn: dynamical cuntz semigroup}) has plain paradoxes. Then $A \rtimesred S$ is either stably finite or purely infinite.
\end{thm}

\begin{proof}
Under these hypotheses, all statements in \cref{thm: stable finiteness,thm: pure infiniteness} are, respectively, equivalent, so the dichotomy is given by the existence or non-existence of a nontrivial order-preserving homomorphism $\nu\colon \Cu(A)_\alpha \to [0,\infty]$.
\end{proof}

\section{Cuntz semigroup retracts and paradoxicality}
\label{sec: retracts}

\Cref{thm: stable finiteness,thm: pure infiniteness,thm: dichotomy result} are powerful, but are stated in terms of the Cuntz semigroup $\Cu(A)$ and its dynamical quotient $\Cu(A)_\alpha$. In general, $\Cu(A)$ is difficult to compute. Even for $A = C_0(X)$, the Cuntz semigroup has only been concretely computed for $X$ with covering dimension $\dim(X) \le 3$ (see, for example, \cite{RobertTikuisis2011, Robert2013b}). This stands in contrast to the explicit type semigroups described in the groupoid literature (for example, in \cite{RS2020, BoenickeLi2020, Ma2022, KMP2025}). Using \cref{thm: stable finiteness,thm: pure infiniteness,thm: dichotomy result} as our guide, we show that we may instead make use of a more computable subobject of $\Cu(A)$ called a \emph{retract}. We begin by defining retracts and showing that the dynamical relations we defined on $\Cu(A)$ in \cref{sec: dynamical Cuntz semigroup} pass naturally to retracts.

In \cref{sec: groupoid C*-algebras} we apply results from \cref{sec: retracts} to the commutative setting where $A = C(X)$ in order to recover the type semigroups associated to C*-algebras of groupoids.

\begin{defn}[{\cite[Definition~3.14]{ThielVilalta2022}}] \label{defn: retract}
Let $R$ and $T$ be $\Cu$-semigroups. We call $R$ a \emph{retract} of $T$ if there exist a $\Cu$-morphism $\rho\colon R \to T$ and a generalised $\Cu$-morphism $\sigma\colon T \to R$ such that $\sigma \circ \rho = \id_R$. In this case, we may also say that $(R,\rho,\sigma)$ is a retract of $T$. We call $R$ a \emph{nonzero} retract if $R \ne \{0\}$.
\end{defn}

\begin{remark} \label{rem: retract properties} Every $\Cu$-semigroup is a retract of itself. Suppose that $(R,\rho,\sigma)$ is retract of a $\Cu$-semigroup $T$. Then the following hold.
\begin{enumerate}[label=(\alph*), ref={\cref{rem: retract properties}(\alph*)}]
\item \label{item: retract maps} Since $\sigma \circ \rho = \id_R$, we have $\rho \circ \sigma\restr{\rho(R)} = \id_{\rho(R)}$, and $\rho$ is injective and $\sigma$ is surjective. It is straightforward to show that $\rho(R)$ is a $\Cu$-semigroup and that $\sigma\restr{\rho(R)}\colon \rho(R) \to R$ is an injective $\Cu$-morphism.
\item \label{item: R^ll and T^ll} The collection $R^\ll \coloneqq \{ r \in R : r \ll x \text{ for some } x \in R \}$ is a submonoid of $R$, and $\sigma(T^\ll)$ is also a submonoid of $R$. Since $\rho$ preserves the sequentially way-below relation $\ll$, $\rho(R^\ll)$ is a submonoid of $T^\ll$, and thus $R^\ll = \sigma(\rho(R^\ll)) \subseteq \sigma(T^\ll)$.
\end{enumerate}
\end{remark}

\begin{defn} \label{defn: invariant retract}
Let $\beta\colon S \curvearrowright T$ be an action of a unital inverse semigroup $S$ on a $\Cu$-semigroup $T$, as defined in \cref{defn: Cu-semigroup action}.
\begin{enumerate}[label=(\alph*)]
\item We say that a retract $(R,\rho,\sigma)$ of $T$ is \emph{invariant} if $\rho(R)$ is invariant under $\beta$, in the sense that $\beta_s(T_{s^*s} \cap \rho(R)) \subseteq \rho(R)$ for each $s \in S$.
\item We say that an invariant retract $(R,\rho,\sigma)$ of $T$ is \emph{strongly invariant} if, for each $s \in S$, $\rho(\sigma(T_{s^*s})) \subseteq T_{s^*s}$ and $\sigma \circ \beta_s \circ \rho \circ \sigma\restr{T_{s^*s}} = \sigma \circ \beta_s$.
\end{enumerate}
\end{defn}

\begin{thm} \label{thm: invariant retract induces action}
Let $\beta\colon S \curvearrowright T$ be an action of a unital inverse semigroup $S$ on a $\Cu$-semigroup $T$, and let $(R,\rho,\sigma)$ be an invariant retract of $T$. Then $\beta$ induces an action $\beta^R\colon S \curvearrowright R$ with ideals $\{ R_{s^*s}\coloneqq\rho^{-1}(T_{s^*s}) \}_{s \in S}$ and $\Cu$-morphisms $\{ \beta^R_s\coloneqq \sigma \circ \beta_s \circ \rho\restr{R_{s^*s}}\colon R_{s^*s} \to R_{ss^*} \}_{s \in S}$ such that $\rho \circ \beta^R_s = \beta_s \circ \rho\restr{R_{s^*s}}$ for each $s \in S$. If $(R,\rho,\sigma)$ is strongly invariant, then $\beta^R_s \circ \sigma\restr{T_{s^*s}} = \sigma \circ \beta_s$ for each $s \in S$.
\end{thm}

\begin{proof}
Fix $s \in S$. Since $T_{s^*s}$ is a submonoid of $T$ and $\rho$ is additive, $\rho^{-1}(T_{s^*s})$ is a submonoid of $R$. To see that $\rho^{-1}(T_{s^*s})$ is hereditary, fix $a \in R$ and $b \in \rho^{-1}(T_{s^*s})$ such that $a \le b$. Then $\rho(a) \le \rho(b)$ because $\rho$ is order-preserving. Since $\rho(b) \in T_{s^*s}$ and $T_{s^*s}$ is hereditary, we have $\rho(a) \in T_{s^*s}$, and so $a \in \rho^{-1}(T_{s^*s})$. Thus $\rho^{-1}(T_{s^*s})$ is hereditary. To see that $\rho^{-1}(T_{s^*s})$ is closed under suprema of increasing sequences, fix an increasing sequence $\{a_n\}_{n \in \N} \subseteq \rho^{-1}(T_{s^*s})$ and let $a = \sup_{n \in \N} (a_n)$. We must show that $a \in \rho^{-1}(T_{s^*s})$. Since $\rho$ is a $\Cu$-morphism, $\{\rho(a_n)\}_{n \in \N}$ is an increasing sequence in $T_{s^*s}$ and $\rho(a) = \sup_{n \in \N} (\rho(a_n))$. Thus, since $T_{s^*s}$ is an ideal of $T$, we have $\rho(a) \in T_{s^*s}$, and hence $a \in \rho^{-1}(T_{s^*s})$. Therefore, $R_{s^*s} \coloneqq \rho^{-1}(T_{s^*s})$ is an ideal of $R$.

Fix $s \in S$, and define $\beta^R_s\coloneqq \sigma \circ \beta_s \circ \rho\restr{R_{s^*s}}$. To see that $\rho \circ \beta^R_s = \beta_s \circ \rho\restr{R_{s^*s}}$, fix $a \in R_{s^*s}$. Then $\rho(a) \in T_{s^*s} \cap \rho(R)$, so $\beta_s(\rho(a)) \in T_{ss^*} \cap \rho(R)$ because $\rho(R)$ is invariant under $\beta$. By \cref{item: retract maps}, we have $\rho \circ \sigma\restr{\rho(R)} = \id_{\rho(R)}$, and therefore,
\[
(\rho \circ \beta^R_s)(a) = (\rho \circ \sigma\restr{\rho(R)} \circ \beta_s \circ \rho\restr{R_{s^*s}})(a) = (\id_{\rho(R)} \circ \beta_s \circ \rho\restr{R_{s^*s}})(a) = (\beta_s \circ \rho\restr{R_{s^*s}})(a).
\]
Thus $\rho \circ \beta^R_s = \beta_s \circ \rho\restr{R_{s^*s}}$. Moreover, since $(\rho \circ \beta^R_s)(a) = (\beta_s \circ \rho\restr{R_{s^*s}})(a) \in T_{ss^*}$, we have $\beta^R_s(a) \in \rho^{-1}(T_{ss^*}) = R_{ss^*}$, so the codomain of $\beta^R_s$ is $R_{ss^*}$.

Since $\beta^R_s$ is a composition of generalised $\Cu$-morphisms, it is itself a generalised $\Cu$-morphism. To see that $\beta^R_s$ is a $\Cu$-morphism, we must show that it preserves the sequentially way-below relation $\ll$. For this, fix $a, b \in R_{s^*s}$ such that $a \ll b$. Suppose that $\{b_n\}_{n \in \N} \subseteq R_{ss^*}$ is an increasing sequence such that $\sup_{n \in \N} (b_n) = \beta^R_s(b)$. Then, since $\rho$ is a $\Cu$-morphism and $\rho \circ \beta^R_s = \beta_s \circ \rho\restr{R_{s^*s}}$, we have
\[
\sup_{n \in \N} (\rho(b_n)) = \rho\big(\beta^R_s(b)\big) = \beta_s(\rho(b)).
\]
Since $\beta_s$ and $\rho$ preserve $\ll$, we have $\beta_s(\rho(a)) \ll \beta_s(\rho(b))$. Therefore, there exists $m \in \N$ such that $\beta_s(\rho(a)) \le \rho(b_m)$. Thus, since $\sigma$ is order-preserving, we have
\[
\beta^R_s(a) = (\sigma \circ \beta_s \circ \rho\restr{R_{s^*s}})(a) \le \sigma(\rho(b_m)) = b_m.
\]
Therefore, $\beta^R_s(a) \ll \beta^R_s(b)$, and $\beta^R_s$ is a $\Cu$-morphism.

We now show that $\beta^R$ satisfies \cref{item: Cu-semigroup action composition}. For this, fix $s, t \in S$. Then, since
\[
T_{(st)^*st} = \beta_t^{-1}(T_{s^*s} \cap T_{tt^*}) \subseteq T_{t^*t} \quad \text{ and } \quad \beta_t \circ \rho\restr{R_{t^*t}} = \rho \circ \beta^R_t,
\]
we have
\begin{align*}
R_{(st)^*st} = \rho^{-1}(T_{(st)^*st}) &= \rho^{-1}\big(\beta_t^{-1}(T_{s^*s} \cap T_{tt^*})\big) = (\beta_t \circ \rho\restr{R_{t^*t}})^{-1}(T_{s^*s} \cap T_{tt^*}) \\
&= (\rho \circ \beta^R_t)^{-1}(T_{s^*s} \cap T_{tt^*}) = (\beta^R_t)^{-1}\big( \rho^{-1}(T_{s^*s} \cap T_{tt^*}) \big) \\
&= (\beta^R_t)^{-1}\big( \rho^{-1}(T_{s^*s}) \cap \rho^{-1}(T_{tt^*}) \big) = (\beta^R_t)^{-1}(R_{s^*s} \cap R_{tt^*}) \subseteq R_{t^*t}.
\end{align*}
It follows that $\rho\restr{R_{s^*s}} \circ \beta^R_t\restr{R_{(st)^*st}} = \beta_t\restr{T_{(st)^*st}} \circ \rho\restr{R_{(st)^*st}}$, and hence
\begin{align*}
\beta^R_s \circ \beta^R_t\restr{R_{(st)^*st}} &= \sigma \circ \beta_s \circ \rho\restr{R_{s^*s}} \circ \beta^R_t\restr{R_{(st)^*st}} \\
&= \sigma \circ \beta_s \circ \beta_t\restr{T_{(st)^*st}} \circ \rho\restr{R_{(st)^*st}} = \sigma \circ \beta_{st} \circ \rho\restr{R_{(st)^*st}} = \beta^R_{st}. \numberthis \label{eqn: beta^R_st}
\end{align*}

Fix $s \in S$. To see that $\beta^R_s$ is bijective, we show that it has inverse $\beta^R_{s^*}$. Note that $\beta_{ss^*} = \id_{T_{ss^*}}$ by \cref{rem: Cu-semigroup action properties}, and thus, taking $t = s^*$ in \cref{eqn: beta^R_st} shows that
\[
\beta^R_s \circ \beta^R_{s^*}\restr{R_{ss^*}} = \beta^R_{ss^*} = \sigma \circ \beta_{ss^*} \circ \rho\restr{R_{ss^*}} = \sigma \circ \id_{T_{ss^*}} \circ \rho\restr{R_{ss^*}} = \sigma \circ \rho\restr{R_{ss^*}} = \id_{R_{ss^*}}.
\]
A similar calculation shows that $\beta^R_{s^*} \circ \beta^R_s\restr{R_{s^*s}} = \id_{R_{s^*s}}$, and thus $\beta^R_s$ is bijective. Finally, we have $R_1 = \rho^{-1}(T_1) = \rho^{-1}(T) = R$, and therefore $\beta^R$ is an action of $S$ on $R$.

Now suppose that $(R,\rho,\sigma)$ is strongly invariant, and fix $s \in S$. Then $\sigma(T_{s^*s}) \subseteq R_{s^*s}$, and $\beta_s^R \circ \sigma\restr{T_{s^*s}} = \sigma \circ \beta_s \circ \rho\restr{R_{s^*s}} \circ \sigma\restr{T_{s^*s}} = \sigma \circ \beta_s$.
\end{proof}

\Cref{thm: invariant retract induces action} shows that any action of an inverse semigroup on a Cuntz semigroup induces an action on an invariant retract, and with this we can define a retracted version of the dynamical Cuntz semigroup.

\begin{defn} \label{defn: retracted alpha-below}
Let $\alpha\colon S \curvearrowright A$ be an action of a unital inverse semigroup $S$ on a C*-algebra $A$. Let $R$ be an invariant retract of $\Cu(A)$, and let $\hat{\alpha}^R\colon S \curvearrowright R$ be the action induced by $\hat{\alpha}\colon S \curvearrowright \Cu(A)$. Fix $x, y \in R$. We say that $x$ is \emph{$\hat{\alpha}^R$-below} $y$, and we write $x \precalpha^R y$, if, for all $f \in R$ such that $f \ll x$, there exist $\{s_i\}_{i=1}^m \subseteq S$ and $\big\{ u_i \in R_{s_i^*s_i} : i \in \{1, \dotsc, m\} \big\}$ such that
\[
f \ll \sum_{i=1}^m u_i \quad \text{ and } \quad \sum_{i=1}^m \hat{\alpha}^R_{s_i}(u_i) \ll y.
\]
We say that $x$ is \emph{$\hat{\alpha}^R$-path-below} $y$, and we write $x \pathalpha^R y$, if there exist $z_1, \dotsc, z_n \in R$ such that $x \precalpha^R z_1 \precalpha^R z_2 \precalpha^R \dotsb \precalpha^R z_n \precalpha^R y$.
\end{defn}

The relations $\precalpha^R$ and $\pathalpha^R$ are preserved by the retract maps in the following sense.

\begin{prop} \label{prop: retracts preserve dynamically below}
Let $\alpha\colon S \curvearrowright A$ be an action of a unital inverse semigroup $S$ on a C*-algebra $A$. Let $(R,\rho,\sigma)$ be an invariant retract of $\Cu(A)$.
\begin{enumerate}[label=(\alph*)]
\item \label{item: precalpha^R implies precalpha} For all $x, y \in R$, we have $x \precalpha^R y \implies \rho(x) \precalpha \rho(y)$.
\item \label{item: le implies precalpha^R} For all $x, y \in R$, we have $x \le y \implies x \precalpha^R y$. In particular, $\precalpha^R$ is reflexive, and $0 \precalpha^R x$ for all $x \in R$.
\end{enumerate}
Suppose that the retract $R$ is strongly invariant.
\begin{enumerate}[resume,label=(\alph*)]
\item \label{item: precalpha implies precalpha^R} For all $x, y \in \Cu(A)$, we have $x \precalpha y \implies \sigma(x) \precalpha^R \sigma(y)$.
\item \label{item: pathalpha implies pathalpha^R} For all $x, y \in \Cu(A)$, we have $x \pathalpha y \implies \sigma(x) \pathalpha^R \sigma(y)$.
\item \label{item: pathalpha^R iff pathalpha} For all $x, y \in R$, we have $x \pathalpha^R y$ if and only if $\rho(x) \pathalpha \rho(y)$.
\end{enumerate}
\end{prop}

\begin{proof}
For part~\cref{item: precalpha^R implies precalpha}, let $x, y \in R$, and suppose that $x \precalpha^R y$. Fix $f \in \Cu(A)$ such that $f \ll \rho(x)$. Choose a sequence $\{x_n\}_{n \in \N} \subseteq R$ such that $x_n \ll x_{n+1}$ for all $n \in \N$ and $\sup_{n \in \N} (x_n) = x$. Since $\rho$ preserves the sequentially way-below relation $\ll$ and suprema of increasing sequences, we have $\rho(x_n) \ll \rho(x_{n+1})$ for all $n \in \N$ and $\sup_{n \in \N} (\rho(x_n)) = \rho(x)$. Hence there exists $m \in \N$ such that $f \le \rho(x_m)$. Since $x_m \ll x$ (by \cref{lemma: ll le properties}) and $x \precalpha^R y$, there exist $\{s_i\}_{i=1}^n \subseteq S$ and $\big \{ u_i \in R_{s_i^*s_i} : i \in \{1, \dotsc, n\} \big\}$ such that
\[
x_m \ll \sum_{i=1}^n u_i \quad \text{ and } \quad \sum_{i=1}^n \hat{\alpha}^R_{s_i}(u_i) \ll y.
\]
Thus, since $\rho$ is a $\Cu$-morphism, we have
\[
f \le \rho(x_m) \ll \sum_{i=1}^n \rho(u_i).
\]
By \cref{thm: invariant retract induces action}, we have $\rho \circ \hat{\alpha}^R_s = \hat{\alpha}_s \circ \rho\restr{R_{s^*s}}$ for all $s \in S$, and so since $\rho$ preserves $\ll$, we have
\[
\sum_{i=1}^n \hat{\alpha}_{s_i}(\rho(u_i)) = \sum_{i=1}^n \rho\big(\hat{\alpha}^R_{s_i}(u_i)\big) = \rho\Big( \sum_{i=1}^n \hat{\alpha}^R_{s_i}(u_i) \Big) \ll \rho(y).
\]
Therefore, $\rho(x) \precalpha \rho(y)$.

For part~\cref{item: le implies precalpha^R}, let $x, y \in R$, and suppose that $x \le y$. Fix $f \in R$ such that $f \ll x$. Then by \cref{lemma: sequentially way-below interpolation}, there exists $u \in R$ such that $f \ll u \ll x$. Thus $\hat{\alpha}^R_1(u) = u \ll x \le y$, so by \cref{lemma: ll le properties}, $\hat{\alpha}^R_1(u) \ll y$. Thus $x \precalpha^R y$. It follows immediately that $\precalpha^R$ is reflexive and $0 \precalpha^R x$ for all $x \in R$.

Now suppose that $R$ is strongly invariant. For part~\cref{item: precalpha implies precalpha^R}, let $x, y \in \Cu(A)$, and suppose that $x \precalpha y$. Fix $g \in R$ such that $g \ll \sigma(x)$. By \cref{lemma: sequentially way-below interpolation}, there exists $w \in R$ such that $g \ll w \ll \sigma(x)$. Choose a sequence $\{x_n\}_{n \in \N} \subseteq \Cu(A)$ such that $x_n \ll x_{n+1}$ for all $n \in \N$ and $\sup_{n \in \N} (x_n) = x$. Since $\sigma$ preserves the partial order $\le$ and suprema of increasing sequences, we have $\sigma(x_n) \le \sigma(x_{n+1})$ for all $n \in \N$ and $\sup_{n \in \N} (\sigma(x_n)) = \sigma(x)$. Hence there exists $m \in \N$ such that $g \ll w \le \sigma(x_m)$. Since $x_m \ll x$ (by \cref{lemma: ll le properties}) and $x \precalpha y$, there exist $\{s_i\}_{i=1}^n \subseteq S$ and $\big\{ u_i \in \Cu(A_{s_i^*s_i}) : i \in \{1, \dotsc, n\} \big\}$ such that
\[
x_m \ll \sum_{i=1}^n u_i \quad \text{ and } \quad \sum_{i=1}^n \hat{\alpha}_{s_i}(u_i) \ll y.
\]
Thus, since $\sigma$ is order-preserving, we have
\[
g \ll w \le \sigma (x_m) \le \sum_{i=1}^n \sigma(u_i).
\]
Since we have assumed that the retract $R$ is strongly invariant, \cref{thm: invariant retract induces action} implies that $\hat{\alpha}^R_s \circ \sigma\restr{\Cu(A_{s^*s})} = \sigma \circ \hat{\alpha}_s$ for all $s \in S$. Thus, since $\sigma$ is order-preserving, we have
\begin{inequality} \label{ineq: sum alpha^R le sigma(y)}
\sum_{i=1}^n \hat{\alpha}^R_{s_i}(\sigma(u_i)) = \sum_{i=1}^n \sigma(\hat{\alpha}_{s_i}(u_i)) = \sigma\Big( \sum_{i=1}^n \hat{\alpha}_{s_i}(u_i) \Big) \le \sigma(y).
\end{inequality}
Since $R$ is strongly invariant, we have $\sigma(u_i) \in \rho^{-1}\big(\Cu(A_{s_i^*s_i})\big) = R_{s_i^*s_i}$ for each $i \in \{1, \dotsc, n\}$, and so by \cref{lemma: sequentially way-below interpolation}, there exist $\big\{ v_i \in R_{s_i^*s_i} : i \in \{1, \dotsc, n\} \big\}$ such that $v_i \ll \sigma(u_i)$ for each $i \in \{1, \dotsc, n\}$, and
\[
g \ll \sum_{i=1}^n v_i \ll \sum_{i=1}^n \sigma(u_i).
\]
For each $i \in \{1, \dotsc, n\}$, $\hat{\alpha}^R_{s_i}$ is a $\Cu$-morphism, and so $\hat{\alpha}^R_{s_i}(v_i) \ll \hat{\alpha}^R_{s_i}(\sigma(u_i))$. Thus, \cref{ineq: sum alpha^R le sigma(y)} implies that
\[
\sum_{i=1}^n \hat{\alpha}^R_{s_i}(v_i) \ll \sum_{i=1}^n \hat{\alpha}^R_{s_i}(\sigma(u_i)) \le \sigma(y).
\]
Therefore, $\sigma(x) \precalpha^R \sigma(y)$.

For part~\cref{item: pathalpha implies pathalpha^R}, let $x, y \in \Cu(A)$, and suppose that $x \pathalpha y$. Then there exist $z_1, \dotsc, z_n \in \Cu(A)$ such that $x \precalpha z_1 \precalpha \dotsb \precalpha z_n \precalpha y$, and so part~\cref{item: precalpha implies precalpha^R} implies that
\[
\sigma(x) \precalpha^R \sigma(z_1) \precalpha^R \dotsb \precalpha^R \sigma(z_n) \precalpha^R \sigma(y).
\]
Therefore, $\sigma(x) \pathalpha^R \sigma(y)$.

For part~\cref{item: pathalpha^R iff pathalpha}, fix $x, y \in R$. Suppose that $x \pathalpha^R y$. Then there exist $w_1, \dotsc, w_k \in R$ such that $x \precalpha^R w_1 \precalpha^R \dotsb \precalpha^R w_k \precalpha^R y$. Thus part~\cref{item: precalpha^R implies precalpha} implies that
\[
\rho(x) \precalpha \rho(w_1) \precalpha \dotsb \precalpha \rho(w_k) \precalpha \rho(y),
\]
and so $\rho(x) \pathalpha \rho(y)$. Now suppose that $\rho(x) \pathalpha \rho(y)$. Then part~\cref{item: pathalpha implies pathalpha^R} implies that $x = \sigma(\rho(x)) \pathalpha^R \sigma(\rho(y)) = y$. Therefore, $x \pathalpha^R y$ if and only if $\rho(x) \pathalpha \rho(y)$.
\end{proof}

\begin{defn} \label{defn: retracted DCS}
Let $\alpha\colon S \curvearrowright A$ be an action of a unital inverse semigroup $S$ on a C*-algebra $A$. Let $R$ be an invariant retract of $\Cu(A)$, and let $\hat{\alpha}^R\colon S \curvearrowright R$ be the action induced by $\hat{\alpha}\colon S \curvearrowright \Cu(A)$. Let $\simalpha^R$ denote the relation on $\sigma\big(\Cu(A)^\ll\big)$ given by
\[
x \simalpha^R y \iff x \pathalpha^R y \,\text{ and }\, y \pathalpha^R x
\]
for $x, y \in \sigma\big(\Cu(A)^\ll\big)$. Since $\pathalpha^R$ is reflexive (by \cref{prop: retracts preserve dynamically below}\cref{item: le implies precalpha^R}) and transitive (by definition), $\simalpha^R$ is an equivalence relation. We denote the $\simalpha^R$-equivalence class of $x \in \sigma\big(\Cu(A)^\ll\big)$ by $[x]_\alpha^R$. An analogous argument to the proofs of \cref{lemma: Cuntz semigroup addition is well-defined,lemma: path order is compatible with addition} shows that, under the operation
\[
[x]_\alpha^R + [y]_\alpha^R \coloneqq [x + y]_\alpha^R,
\]
$R_\alpha \coloneqq \sigma\big(\Cu(A)^\ll\big) / \!\simalpha^R$ is an abelian monoid with identity $0 = [0]_\alpha^R$. Moreover, $R_\alpha$ is preordered with respect to the relation
\[
[x]_\alpha^R \le [y]_\alpha^R \iff x \pathalpha^R y.
\]
We call $R_\alpha$ the \emph{retracted dynamical Cuntz semigroup} of $\alpha\colon S \curvearrowright A$.
\end{defn}

\begin{defn}
Let $\alpha\colon S \curvearrowright A$ be an action of a unital inverse semigroup $S$ on a C*-algebra $A$. Let $(R,\rho,\sigma)$ be a retract of $\Cu(A)$, and let $M$ be a submonoid of $R$.
\begin{enumerate}[label=(\alph*)]
\item Fix $x \in R$.
\begin{enumerate}[label=(\roman*)]
\item For integers $k > l > 0$, we say that $x$ is \emph{$(k,l)$-paradoxical} if $kx \pathalpha^R lx$.
\item We say that $x$ is \emph{completely non-paradoxical} if there are no integers $k > l > 0$ such that $x$ is $(k,l)$-paradoxical.
\end{enumerate}
\item We say that the submonoid $M$ is \emph{completely non-paradoxical} if each $x \in M {\setminus} \{0\}$ is completely non-paradoxical as an element of $R$.
\end{enumerate}
\end{defn}

Using \cref{prop: retracts preserve dynamically below} we can give variants of \cref{thm: stable finiteness,thm: pure infiniteness,thm: dichotomy result} when there are retracts with ``enough'' structure. For the retracted version of \cref{thm: stable finiteness}, we require the retract to be nonzero. For the retracted version of \cref{thm: pure infiniteness}, we require that $[a]_\Cu \in \rho(R)$ for each $a \in A^+$.

In order to prove \cref{thm: retract stable finiteness,thm: retract PIS}, we need the following lemma.

\begin{lemma} \label{lemma: sigma_alpha}
Let $\alpha\colon S \curvearrowright A$ be an action of a unital inverse semigroup $S$ on $A$, let $(R,\rho,\sigma)$ be a strongly invariant retract of $\Cu(A)$, and let $R_\alpha$ be the retracted dynamical Cuntz semigroup. Then the generalised $\Cu$-morphism $\sigma\colon\Cu(A) \to R$ induces a surjective order-preserving homomorphism $\sigma_\alpha\colon\Cu(A)_\alpha \to R_\alpha$ such that $\sigma_\alpha([x]_\alpha) = [\sigma(x)]_\alpha^R$ for each $x \in \Cu(A)^\ll$.
\end{lemma}

\begin{proof}
For all $x, y \in \Cu(A)^\ll$, we have $x \pathalpha y \implies \sigma(x) \pathalpha^R \sigma(y)$ by \cref{prop: retracts preserve dynamically below}\cref{item: pathalpha implies pathalpha^R}, and thus $\sigma_\alpha$ is well-defined. Since $\sigma$ is additive and surjective by \cref{item: retract maps}, $\sigma_\alpha$ is a surjective homomorphism. It follows from \cref{prop: retracts preserve dynamically below}\cref{item: pathalpha implies pathalpha^R} that $\sigma_\alpha$ is order-preserving.
\end{proof}

\begin{thm} \label{thm: retract stable finiteness}
Let $\alpha\colon S \curvearrowright A$ be an action of a unital inverse semigroup $S$ on a unital and exact C*-algebra $A$. Suppose that $A \rtimesred S$ is simple, and let $(R,\rho,\sigma)$ be a nonzero strongly invariant retract of $\Cu(A)$. The following conditions are equivalent to items~\crefrange{item: A faithful trace}{item: faithful invariant functional} in \cref{thm: stable finiteness}.
\begin{enumerate}[label=(\arabic*$'$), start=5]
\item \label{item: R^ll CNP} The monoid $R^\ll$ is completely non-paradoxical.
\item \label{item: retract any CNP element} There exists $x \in \sigma\big(\Cu(A)^\ll\big) {\setminus} \{0\}$ that is completely non-paradoxical.
\item \label{item: retract OP hm} There exists a nontrivial order-preserving homomorphism $\nu\colon R_\alpha \to [0,\infty]$.
\end{enumerate}
\end{thm}

\begin{proof}
The hypotheses ensure that items~\crefrange{item: A faithful trace}{item: faithful invariant functional} in \cref{thm: stable finiteness} are equivalent. We show that
\[
\text{\cref{thm: stable finiteness}\cref{item: Cu(A)^ll CNP}} \implies \text{\cref{item: R^ll CNP}} \implies \text{\cref{item: retract any CNP element}} \implies \text{\cref{item: retract OP hm}} \implies \text{\cref{thm: stable finiteness}\cref{item: OP hm}}.
\]
To see that \cref{thm: stable finiteness}\cref{item: Cu(A)^ll CNP} $\implies$ \cref{item: R^ll CNP}, suppose that $\Cu(A)^\ll$ is completely non-paradoxical. By \cref{item: R^ll and T^ll}, we have $\rho(R^\ll) \subseteq \Cu(A)^\ll$, and hence $\rho(R^\ll)$ is completely non-paradoxical. Fix $x \in R^\ll {\setminus} \{0\}$, and suppose for contradiction that there exist integers $k > l > 0$ such that $x$ is $(k,l)$-paradoxical. Then $kx \pathalpha^R lx$, so \cref{prop: retracts preserve dynamically below}\cref{item: pathalpha^R iff pathalpha} implies that
\[
k \rho(x) = \rho(kx) \pathalpha \rho(lx) = l \rho(x).
\]
Since $\rho(R^\ll)$ is completely non-paradoxical, we must have $\rho(x) = 0$, but this is a contradiction because $\rho$ is injective and $x \ne 0$. Therefore, $R^\ll$ is completely non-paradoxical.

That \cref{item: R^ll CNP} $\implies$ \cref{item: retract any CNP element} is immediate because $R^\ll \subseteq \sigma\big(\Cu(A)^\ll\big)$ by \cref{item: R^ll and T^ll}.

To see that \cref{item: retract any CNP element} $\implies$ \cref{item: retract OP hm}, suppose that $x \in \sigma\big(\Cu(A)^\ll\big) {\setminus} \{0\}$ is completely non-paradoxical. Then for all $n \in \N$, we have $(n+1)x \not\pathalpha^R nx$, and hence $(n+1)[x]_\alpha^R \not\le n[x]_\alpha^R$. Now \cref{thm: Tarski-Wehrung paradoxicality thm} implies that there exists an order-preserving homomorphism $\nu\colon R_\alpha \to [0,\infty]$ such that $\nu([x]_\alpha) = 1 \in (0,\infty)$, giving \cref{item: OP hm}.

To see that \cref{item: retract OP hm} $\implies$ \cref{thm: stable finiteness}\cref{item: OP hm}, suppose that $\nu\colon R_\alpha \to [0,\infty]$ is a nontrivial order-preserving homomorphism. By \cref{lemma: sigma_alpha}, $\sigma\colon \Cu(A) \to R$ induces a surjective order-preserving homomorphism $\sigma_\alpha\colon\Cu(A)_\alpha \to R_\alpha$ such that $\sigma_\alpha([x]_\alpha) = [\sigma(x)]_\alpha^R$ for each $x \in \Cu(A)^\ll$. Thus $\nu \circ \sigma_\alpha\colon \Cu(A)_\alpha \to [0,\infty]$ is an order-preserving homomorphism. Since $\nu$ is nontrivial and $\sigma_\alpha$ is surjective, $\nu \circ \sigma_\alpha$ must also be nontrivial, so \cref{thm: stable finiteness}\cref{item: OP hm} holds.
\end{proof}

\begin{thm} \label{thm: retract PIS}
Let $\alpha\colon S \curvearrowright A$ be an action of a unital inverse semigroup $S$ on a unital and exact C*-algebra $A$. Suppose that $A \rtimesred S$ is simple. Then $\alpha$ is minimal, $A$ detects ideals in $A \rtimesred S$, and $A \rtimesred S = A \rtimesess S$. Suppose that $(R,\rho,\sigma)$ is a strongly invariant retract of $\Cu(A)$ such that $[a]_\Cu \in \rho(R)$ for each $a \in A^+$. Let $R_\alpha$ be the retracted dynamical Cuntz semigroup from \cref{defn: retracted DCS}. Consider the following statements.
\begin{enumerate}[label=(\arabic*)]
\item \label{item: retract PIS PIM} The retracted dynamical Cuntz semigroup $R_\alpha$ is purely infinite.
\item \label{item: retract PIS 21P} Every element of $R^\ll {\setminus} \{0\}$ is $(2,1)$-paradoxical.
\item \label{item: retract PIS PIC*} The C*-algebra $A \rtimesess S$ is purely infinite.
\item \label{item: retract PIS NTS} The C*-algebra $A \rtimesess S$ admits no tracial states.
\item \label{item: retract PIS OPH} There does not exist a nontrivial order-preserving homomorphism $\nu\colon R_\alpha \to [0,\infty]$.
\end{enumerate}
Then \cref{item: retract PIS PIM} $\implies$ \cref{item: retract PIS 21P} $\implies$ \cref{item: retract PIS PIC*} $\implies$ \cref{item: retract PIS NTS} $\implies$ \cref{item: retract PIS OPH}. If the retracted dynamical Cuntz semigroup $R_\alpha$ has plain paradoxes, then \cref{item: retract PIS OPH} $\implies$ \cref{item: retract PIS PIM}, and \crefrange{item: retract PIS PIM}{item: retract PIS OPH} are equivalent.
\end{thm}

\begin{remark}
Although it might appear that when $R_\alpha$ has plain paradoxes, the conditions of \cref{thm: retract PIS} are equivalent to those of \cref{thm: pure infiniteness}, it is not the case that $R_\alpha$ having plain paradoxes implies that $\Cu(A)_\alpha$ does, so we need to be careful. Nevertheless, much of the proof of \cref{thm: retract PIS} is very similar to the corresponding parts of \cref{thm: pure infiniteness}.
\end{remark}

\begin{proof}[Proof of \cref{thm: retract PIS}]
Since $A \rtimesred S$ is simple, parts~\cref{item: simple iff minimal and detects ideals,item: detecting ideals implies injective quotient map} of \cref{rem: simplicity and aperiodic actions and detecting ideals} together imply that $\alpha$ is minimal, $A$ detects ideals in $A \rtimesred S$, and the singular ideal of $A \rtimesred S$ is trivial, so $A \rtimesred S = A \rtimesess S$.

To see that \cref{item: retract PIS PIM} $\implies$ \cref{item: retract PIS 21P}, fix $x \in R^\ll {\setminus} \{0\}$. Then $x \in \sigma\big(\Cu(A)^\ll\big) {\setminus} \{0\}$ by \cref{item: R^ll and T^ll}. By assumption, $R_\alpha$ is purely infinite, and hence $[x]_\alpha^R$ is properly infinite; that is, $2[x]_\alpha^R \le [x]_\alpha^R$. Thus, by definition of the partial order on $R_\alpha$, we have $2x \pathalpha^R x$, so $x$ is $(2,1)$-paradoxical.

To see that \cref{item: retract PIS 21P} $\implies$ \cref{item: retract PIS PIC*}, suppose that every element of $R^\ll {\setminus} \{0\}$ is $(2,1)$-paradoxical. Let $\iota\colon A \to A \rtimesess S$ be the canonical embedding $a \mapsto a\delta_1$, and let $\hat{\iota}\colon\Cu(A) \to \Cu(A \rtimesess S)$ be the induced $\Cu$-morphism from \cref{lemma: Cu-embedding in crossed product is invariant}. By \cite[Corollary~6.7]{KwasniewskiMeyer2021}, $A \rtimesess S$ is purely infinite if and only if every element of $A^+ {\setminus} \{0\}$ is infinite in $A \rtimesess S$, in the sense of \cref{item: infinite element of C*-algebra}. Fix $a \in A^+ {\setminus} \{0\}$. By our hypothesis, we have $[a]_\Cu \in \rho(R) {\setminus} \{0\}$, so there exists $x \in R {\setminus} \{0\}$ such that $\rho(x) = [a]_\Cu$. By axiom~\cref{item: O2} of \cref{defn: Cu-semigroup}, there is a sequence $\{x_n\}_{n \in \N} \subseteq R {\setminus} \{0\}$ such that $x_n \ll x_{n+1}$ for all $n \in \N$ and $x = \sup_{n \in \N} (x_n)$. By assumption, each $x_n$ is $(2,1)$-paradoxical, so $2x_n \pathalpha^R x_n$ for all $n \in \N$. Hence \cref{prop: retracts preserve dynamically below}\cref{item: pathalpha^R iff pathalpha} implies that $\rho(2x_n) \pathalpha \rho(x_n)$ for all $n \in \N$, and it follows by \cref{cor: iota hat preserves precalpha} that $\hat{\iota}(\rho(2x_n)) \le \hat{\iota}(\rho(x_n))$ for all $n \in \N$. Since $\hat{\iota}$ and $\rho$ are $\Cu$-morphisms, $\hat{\iota} \circ \rho$ preserves suprema of increasing sequences, and therefore,
\begin{align*}
2[\iota(a)]_\Cu &= 2\hat{\iota}([a]_\Cu) = 2\hat{\iota}(\rho(x)) = (\hat{\iota} \circ \rho)(2x) = (\hat{\iota} \circ \rho)\big( 2 \sup_{n \in \N} (x_n) \big) = \sup_{n \in \N} \big((\hat{\iota} \circ \rho)(2x_n)\big) \\
&\le \sup_{n \in \N} \big((\hat{\iota} \circ \rho)(x_n)\big) = (\hat{\iota} \circ \rho)\big( \sup_{n \in \N} (x_n) \big) = (\hat{\iota} \circ \rho)(x) = \hat{\iota}([a]_\Cu) = [\iota(a)]_\Cu.
\end{align*}
Since $\iota$ is injective, $[\iota(a)]_\Cu \ne 0$, so $\iota(a)$ is infinite (in fact, properly infinite) in $A \rtimesess S$. Thus $A \rtimesess S$ is purely infinite.

That \cref{item: retract PIS PIC*} $\implies$ \cref{item: retract PIS NTS} is immediate because purely infinite C*-algebras do not admit tracial states (see, for example, \cite[Proposition~V.2.2.29]{Blackadar2006}).

To see that \cref{item: retract PIS NTS} $\implies$ \cref{item: retract PIS OPH}, we prove the contrapositive. Suppose there exists a nontrivial order-preserving homomorphism $\nu\colon R_\alpha \to [0,\infty]$. By \cref{lemma: sigma_alpha}, $\sigma\colon \Cu(A) \to R$ induces a surjective order-preserving homomorphism $\sigma_\alpha\colon\Cu(A)_\alpha \to R_\alpha$ such that $\sigma_\alpha([x]_\alpha) = [\sigma(x)]_\alpha^R$ for each $x \in \Cu(A)^\ll$. Thus $\nu \circ \sigma_\alpha\colon \Cu(A)_\alpha \to [0,\infty]$ is an order-preserving homomorphism. Since $\nu$ is nontrivial and $\sigma_\alpha$ is surjective, $\nu \circ \sigma_\alpha$ must also be nontrivial. Thus, since $\alpha$ is minimal, we can apply \cref{thm: OP hom on DCS induces functional} to find an invariant functional $\beta\colon \Cu(A) \to [0,\infty]$ such that $\beta([1_A]_\Cu) = 1$. Then since $A$ is exact and $A$ detects ideals in $A \rtimesred S$, \cref{thm: traces reduced} implies that $A \rtimesred S = A \rtimesess S$ admits a tracial state.

Suppose now that $R_\alpha$ has plain paradoxes. Since $A$ is unital and $[1_A]_\Cu \in \rho(R) {\setminus} \{0\}$ by assumption, we have $\sigma\restr{\rho(R)}([1_A]_\Cu) \ne 0$ by \cref{item: retract maps}. Since $[1_A]_\Cu \in \Cu(A)^\ll$ by \cref{rem: [p]_Cu ll [p]_Cu}, we have $\sigma([1_A]_\Cu) \in \sigma\big(\Cu(A)^\ll\big) {\setminus} \{0\}$, and a routine argument using \cref{rem: nonzero sequentially way below} then shows that $R_\alpha \ne \{0\}$. To see that \cref{item: retract PIS OPH} $\implies$ \cref{item: retract PIS PIM}, assume that \cref{item: retract PIS OPH} holds, and fix $x \in \sigma\big(\Cu(A)^\ll\big)$. If $x = 0$, then $2[x]_\alpha^R = [0]_\alpha^R \le [x]_\alpha^R$. Suppose that $x \ne 0$. Then there must exist $n \in \N$ such that $(n+1)[x]_\alpha^R \le n[x]_\alpha^R$, because if not, then \cref{thm: Tarski-Wehrung paradoxicality thm} would produce a nontrivial order-preserving homomorphism $\nu\colon R_\alpha \to [0,\infty]$ such that $\nu([x]_\alpha^R) = 1$, contradicting \cref{item: retract PIS OPH}. Since $R_\alpha$ has plain paradoxes, it follows that $2[x]_\alpha^R \le [x]_\alpha^R$. In particular, every element of $R_\alpha {\setminus} \{0\}$ is properly infinite, and thus $R_\alpha$ is purely infinite, giving \cref{item: retract PIS PIM}.
\end{proof}

We now obtain a version of \cref{thm: dichotomy result} that gives the stably finite / purely infinite dichotomy for $A \rtimesess S$ using only the retracted dynamical Cuntz semigroup.

\begin{thm} \label{thm: retract dichotomy}
Let $\alpha\colon S \curvearrowright A$ be an action of a unital inverse semigroup $S$ on a unital and exact C*-algebra $A$. Suppose that $A \rtimesred S$ is simple, and that $(R,\rho,\sigma)$ is a strongly invariant retract of $\Cu(A)$ such that $[a]_\Cu \in \rho(R)$ for each $a \in A^+$. Let $R_\alpha$ be the retracted dynamical Cuntz semigroup from \cref{defn: retracted DCS}. If $R_\alpha$ has plain paradoxes, then $A \rtimesred S$ is either stably finite or purely infinite.
\end{thm}

\begin{proof}
Under these hypotheses, all statements in \cref{thm: retract stable finiteness,thm: retract PIS} are, respectively, equivalent, so the dichotomy is given by the existence or non-existence of a nontrivial order-preserving homomorphism $\nu\colon R_\alpha \to [0,\infty]$.
\end{proof}

\begin{remark}
It is unclear whether there is a ``canonical'' retract one should consider for $\Cu(A)$ of an arbitrary C*-algebra $A$. However, as we discuss in \cref{subsec: groupoid dichotomy}, when $A = C(X)$, there is an ``obvious'' candidate for a retract of $\Cu(A)$ that appears in the literature, corresponding to the Cuntz-equivalence classes of diagonal compact operators. This retract of $\Cu(C(X))$ also corresponds (modulo a completion) to the monoid used to build the type semigroup associated to C*-algebras of groupoids with unit space $X$ (see, for example, \cite{RS2020, BoenickeLi2020, Ma2022, KMP2025}).
\end{remark}

\section{Applications to groupoid C*-algebras}
\label{sec: groupoid C*-algebras}

In this final section, we apply our main theorems to groupoid C*-algebras. The point here is that crossed products $A \rtimesred S$ and $A \rtimesess S$ are groupoid C*-algebras if (and only if) $A$ is commutative. We construct a retracted dynamical Cuntz semigroup of a groupoid C*-algebra and show that it coincides with the type semigroup constructed in (for example) \cite{RS2020, BoenickeLi2020, Ma2022, KMP2025}. As a result, the dichotomy in \cref{cor: dichotomy groupoid} is a special case of \cite[Corollary~6.9]{KMP2025}.

\subsection{Groupoid C*-algebras as crossed products by inverse semigroups}
\label{subsec: groupoid C*s as CPs}

In this section, we give a brief introduction to the groupoid setting, and we formally present the correspondence between groupoid C*-algebras and crossed products of commutative C*-algebras by unital inverse semigroups (see \cref{thm: duality}).

Let $X$ be a locally compact Hausdorff space, and let $S$ be a unital inverse semigroup. An \emph{action $\alpha\colon S \curvearrowright X$} consists of collections of open sets $\{ D_e \subseteq X \}_{e \in E(S)}$ and partial homeomorphisms $\{ \alpha_s\colon D_{s^*s} \to D_{ss^*} \}_{s \in S}$ such that for all $s, t \in S$, $\alpha_{st} = \alpha_s \circ \alpha_t$ wherever composition is defined. Note that if $A \cong C_0(X)$, then it follows from Gelfand duality that an action of $S$ on $X$ induces an action of $S$ on $A$ (as defined in \cref{defn: C*-algebra action}), and vice versa. We will henceforth use the two notions of inverse-semigroup actions interchangeably. If we have an action $S \curvearrowright C_0(X)$, then for each $s \in S$, the spectrum of the ideal $I_{s,1}$ of $C_0(X)$ is the open subset $O_{s,1} \coloneqq \cup_{e \le s, 1} D_e$ of $X$.

Let $\GG$ be an \'etale groupoid with locally compact Hausdorff unit space $X \coloneqq \GGo$ and range and source maps $r,s\colon \GG \to X$. Let $S = S(\GG)$ be the inverse semigroup of open bisections of $\GG$. Then there is an action $\theta\colon S \curvearrowright X$ given by
\[
\theta_U(x) \coloneqq r\big(s\restr{U}^{-1}(x)\big) = UxU^*
\]
for all $U \in S$ and $x \in s(U) = U^*U$. We call $\theta$ the \emph{canonical action of $S$ on $X$}. As discussed on \cite[Page~230]{Exel2008}, $\theta$ induces an action $S \curvearrowright C_0(X)$, which we also denote by $\theta$, given by
\begin{equation} \label{eqn: groupoid action}
\theta_U(f)(x) \coloneqq f\big(\theta_U^{-1}(x)\big) = f\big(s\big(r\restr{U}^{-1}(x)\big)\big) = f(U^*xU)
\end{equation}
for all $U \in S$, $f \in C_0(s(U))$, and $x \in r(U)$. Note that for $U \in S$, we have
\[
C_0(X)_{U^*U} = \dom(\theta_U) = C_0(s(U)) \quad \text{ and } \quad C_0(X)_{UU^*} = \ran(\theta_U) = C_0(r(U)).
\]
Define
\[
\CC_c(\GG) \coloneqq \vecspan\{ f \in C(\GG) : f\restr{U} \in C_c(U) \text{ for some open bisection } U \subseteq \GG \text{ and } f\restr{\GG {\setminus} U} \equiv 0 \}.
\]
Recall (for example, from \cite[Definition~2.2]{KKLRU2021}) that $\CC_c(\GG)$ is dense in the C*-algebras $C_{\max}^*(\GG)$, $C_\red^*(\GG)$, and $C_\ess^*(\GG)$ of $\GG$, and that the identity map on $\CC_c(\GG)$ extends to quotient maps $C_{\max}^*(\GG) \twoheadrightarrow C_\red^*(\GG)$ and $C_\red^*(\GG) \twoheadrightarrow C_\ess^*(\GG)$.

In what follows, let $B_b(X)$ denote the C*-algebra of complex-valued bounded Borel functions on $X$, and let $\MM(X) \subseteq B_b(X)$ be the ideal of such functions with meagre support. We have $\Mloc(C_0(X)) \cong B_b(X) / \MM(X)$ and $B_b(X) \subseteq C_0(X)''$.

\begin{thm} \label{thm: duality}
There is a correspondence between groupoid C*-algebras and the crossed products built from actions of unital inverse semigroups on commutative C*-algebras, in the following sense.

\begin{enumerate}[label=(\alph*)]
\item \label{item: duality gpd to IS} Let $\GG$ be an \'etale groupoid with second-countable locally compact Hausdorff unit space $X \coloneqq \GGo$, let $S = S(\GG)$, and let $\theta\colon S \curvearrowright X$ be the canonical action of $S$ on $X$. Then there is a homomorphism
\begin{equation} \label{map: duality}
\psi_\alg\colon \CC_c(\GG) \to C_0(X) \rtimesalg S, \ \text{ such that } \ \psi_{\alg}(f) = \big(f \circ r\restr{U}^{-1}\big) \delta_U
\end{equation}
for $f \in C_c(U)$ and $U \in S$, that induces isomorphisms $\psi_{\max}, \psi_\red$, and $\psi_\ess$ making the following diagram commute.
\begin{center}
\begin{tikzcd}[scale=50em]
C_{\max}^*(\GG) \arrow[r]{}{\psi_{\max}} \arrow[d,two heads] & C_0(X) \rtimesfull S \arrow[d,two heads] & \\
C_\red^*(\GG) \arrow[r]{}{\psi_\red} \arrow[d,two heads] & C_0(X) \rtimesred S \arrow[d,two heads]{}{\Lambda} \arrow[r]{}{\ER} & B_b(X) \arrow[d,two heads] \\
C_\ess^*(\GG) \arrow[r]{}{\psi_\ess} & C_0(X) \rtimesess S \arrow[r]{}{\EL} & B_b(X) / \MM(X).
\end{tikzcd}
\end{center}
\item \label{item: duality IS to gpd} Let $X$ be a second-countable locally compact Hausdorff space, and let $S$ be a countable unital inverse semigroup acting on $X$. Then there is an \'etale groupoid $\GG$ with unit space $\GGo = X$ and a homomorphism $\psi_\alg$ defined as in \labelcref{map: duality}
that induces isomorphisms $\psi_{\max}, \psi_\red$, and $\psi_\ess$ such that the above diagram commutes.
\end{enumerate}
\end{thm}

\begin{proof}
For part~\cref{item: duality gpd to IS}, note that \cite[Proposition~9.8]{Exel2008} shows that there is an isomorphism
\[
\psi_{\max}\colon C_{\max}^*(\GG) \to C_0(X) \rtimesfull S,
\]
and the proof of \cite[Theorem~9.8]{Exel2008} shows that for $f \in C_c(U)$ and $U \in S$,
\[
\psi_{\alg}(f) = \psi_{\max}(f) = \big(f \circ r\restr{U}^{-1}\big) \delta_U.
\]
We claim that $\psi_{\max}$ descends to an isomorphism of the essential C*-algebras. To see this, let
\[
\ER_{\GG}\colon C_{\max}^*(\GG) \to B_b(X) \subseteq C_0(X)''
\]
denote the usual (weak) conditional expectation given by restriction of functions to the unit space $X$, and let
\[
\EL_{\GG}\colon C_{\max}^*(\GG) \to \Mloc(C_0(X))
\]
be the map $a \mapsto \ER_{\GG}(a) + \MM(X)$. Let $\NN_{\EL_\GG} \coloneqq \{ a \in C_{\max}^*(\GG) : \EL_\GG(a^*a) = 0 \}$. Since $C_\ess^*(\GG) = C_{\max}^*(\GG) / \NN_{\EL_\GG}$ and $C_0(X) \rtimesess S = \big(C_0(X) \rtimes_{\max} S\big) / \NN_\EL$, it suffices to show that $\psi_{\max}\big(\NN_{\EL_\GG}\big) = \NN_\EL$. We first show that $\EL_\GG = \EL \circ \psi_{\max}$. For this, fix $U \in S$. Then $I_{U,1} = C_0(U \cap X)$ and $i_U = 1_{U \cap X} + \MM(X)$. Let $f \in C_c(U)$. Since $C_{\max}^* (\GG)$ is equal to the closed span (with respect to the maximal norm) of such functions, it suffices to show that $\EL_\GG(f) = \EL\big(\psi(f)\big)$. Define $f^r \coloneqq f \circ r\restr{U}^{-1}$. Then $f^r \in C_c(UU^*)$ and $f = f^r \circ r\restr{U}$. Using \cref{thm: EL existence and properties} for the second equality, we get
\[
\EL\big(\psi_{\max}(f)\big) = \EL(f^r \delta_U) = f^r i_U = f^r \big(1_{U \cap X} + \MM(X)\big) = f\restr{X} + \MM(X) = \EL_\GG(f).
\]
Thus $\EL_\GG = \EL \circ \psi_{\max}$. To see that $\psi\big(\NN_{\EL_\GG}\big) \subseteq \NN_\EL$, fix $a \in \NN_{\EL_\GG}$. Then
\[
\EL\big(\psi_{\max}(a)^*\psi_{\max}(a)\big) = \EL\big(\psi_{\max}(a^*a)\big) = \EL_\GG(a^*a) = 0,
\]
and so $\psi_{\max}(a) \in \NN_\EL$ by the definition of $\NN_\EL$. For the reverse containment, fix $b \in \NN_\EL$. Then $\EL(b^*b) = 0$. Since $\psi_{\max}$ is an isomorphism, there is a unique element $a \in C_{\max}^*(\GG)$ such that $b = \psi_{\max}(a)$, and we have
\[
\EL_\GG(a^*a) = \EL\big(\psi_{\max}(a^*a)\big) = \EL\big(\psi_{\max}(a)^*\psi_{\max}(a)\big) = \EL(b^*b) = 0.
\]
Thus $\psi_{\max}\big(\NN_{\EL_\GG}\big) = \NN_\EL$, proving the claim. Therefore, $\psi_{\max}$ descends to an isomorphism $\psi_\ess\colon C_\ess^*(\GG) \to C_0(X) \rtimesess S$. By a similar argument using that $\ER \circ \psi_{\max} = \ER_\GG$, $\psi_{\max}$ also descends to an isomorphism $\psi_\red\colon C_\red^*(\GG) \to C_0(X) \rtimesred S$. That the diagram commutes is clear: the left and middle vertical arrows are induced by the identity maps on $\CC_c(\GG)$ and $C_0(X)\rtimes_{\alg}S$ respectively, and the right-most vertical arrow is the canonical quotient map.

For part~\cref{item: duality IS to gpd}, note that by \cite[Theorem~9.8]{Exel2008}, there exists a groupoid $\GG$ with unit space $X$ such that $C_{\max}^*(\GG) \cong C_0(X) \rtimes_{\max} S$, and the isomorphism, when restricted to $\CC_c(\GG)$, satisfies \labelcref{map: duality}. The rest of the argument follows as in the proof of part~\cref{item: duality gpd to IS}.
\end{proof}

\begin{remark}
\Cref{thm: duality} is known to hold in greater generality and in particular with no second-countability assumptions; for details, see \cite{BKM2025, KwasniewskiMeyer2021}. The condition of the space $X$ being second-countable is equivalent to the C*-algebra $C_0(X)$ being separable. While this separability assumption is not necessary in order to apply our results \cref{thm: retract stable finiteness,thm: retract PIS,thm: retract dichotomy}, we have chosen to assume it throughout \cref{sec: groupoid C*-algebras} as it simplifies some of our arguments, while still allowing us to show how our results apply to a large class of groupoid C*-algebras.
\end{remark}

\subsection{A dichotomy for the essential C*-algebra of a groupoid}
\label{subsec: groupoid dichotomy}

To describe our results on stable finiteness and pure infiniteness and our dichotomy theorem in the context of C*\nobreakdash-algebras of groupoids with second-countable compact Hausdorff unit space $X$, we first need to understand the associated Cuntz semigroup $\Cu(C(X))$. As mentioned earlier, $\Cu(C(X))$ has only been computed for low-dimensional spaces $X$. To overcome this, we restrict our attention to a well-behaved retract and apply the results of \cref{sec: retracts}.

We begin by recalling some results about Cuntz comparison in $C(X)$. For $f \in C(X)$, we define $\osupp(f) \coloneqq \{ x \in X : f(x) \ne 0 \}$, and we call this set the \emph{open support} of $f$.

\begin{lemma}[{see \cite[Propositions~2.2 and~4.4]{GardellaPerera2024}}] \label{lemma: Cuntz comparison for functions}
Let $X$ be a compact Hausdorff space, and fix $f,g \in C(X)^+$. Then
\begin{enumerate}[label=(\alph*), ref={\cref{lemma: Cuntz comparison for functions}(\alph*)}]
\item $f \preceq g $ if and only if $\osupp(f) \subseteq \osupp(g)$;
\item \label{item: Cuntz equivalence for functions} $f \sim g $ if and only if $\osupp(f) = \osupp(g)$; and
\item $f \ll g $ if and only if there is a compact set $K \subseteq X$ such that $\osupp(f) \subseteq K \subseteq \osupp(g)$.
\end{enumerate}
\end{lemma}

\begin{defn}
Define $\bar{\N} \coloneqq \N \cup \{\infty\}$. Let $X$ be a compact metrisable space. The collection of \emph{lower-semicontinuous functions} on $X$, denoted by $\Lsc(X,\bar{\N})$, is the set of functions $F\colon X \to \bar{\N}$ that are lower-semicontinuous in the sense that for each $i \in \N$, the set $F^{-1}(\{i, i+1, \dotsc \} \cup \{\infty\})$ is open in $X$. (Note that we do not require $F^{-1}(\infty)$ to be open.)
\end{defn}

Every function $F \in \Lsc(X,\bar{\N})$ has a unique \emph{normal form} such that
\[
F = \sum_{i=1}^\infty 1_{U_i}, \quad \text{ where } U_i = F^{-1}(\{ i, i+1, \dotsc \} \cup \{\infty\}).
\]
Note that $U_{i+1} \subseteq U_i$ for each $i \ge 1$. We equip $\Lsc(X,\bar{\N})$ with pointwise addition and partial order given by
\[
F \le H \text{ if and only if } F(x) \le H(x) \text{ for all } x \in X;
\]
equivalently, $F = \sum_{i=1}^\infty 1_{U_i} \le H = \sum_{i=1}^\infty 1_{V_i}$ if and only if $U_i \subseteq V_i$ for each $i \ge 1$.

It is shown in \cite[Corollary~4.19]{Vilalta2023} that with this structure, $\Lsc(X,\bar{\N})$ is a $\Cu$-semigroup. The sequentially way-below relation in $\Lsc(X,\bar{\N})$ is given by $F = \sum_{i=1}^\infty 1_{U_i} \ll H = \sum_{i=1}^\infty 1_{V_i}$ if and only if $\im(F)$ is a finite subset of $\N$ and for each $i \ge 1$ there exists a compact set $K_i \subseteq X$ such that $U_i \subseteq K_i \subseteq V_i$. Therefore, $F \in \Lsc(X,\bar{\N})^\ll$ if and only if $\im(F) \subseteq [0,N]$ for some $N \in \N$.

Suppose now that $X$ is a second-countable compact Hausdorff space. Then in particular, $X$ is metrisable. We will follow the argument in \cite[Corollary~4.8]{ThielVilalta2022} to show that $\Lsc(X,\bar{\N})$ is a retract of $\Cu(C(X))$. Note that for any nonempty open set $U \subseteq X$, we can apply the strong form of Urysohn's lemma to find some $f_U \in C(X)^+$ with $\osupp(f_U) = U$. Moreover, for any $g_U \in C(X)^+$ with $\osupp(g_U) = U$, we have $[g_U]_\Cu = [f_U]_\Cu$ by \cref{item: Cuntz equivalence for functions}. It is straightforward to show that the assignment $\rho(1_U) \coloneqq [f_U]_\Cu$ extends to a $\Cu$-morphism $\rho\colon \Lsc(X,\bar{\N}) \to \Cu(C(X))$. For all $f, g \in C(X,\KK)^+ \cong (C(X) \otimes \KK)^+$ such that $[f]_\Cu = [g]_\Cu$, we have $\rank(f(x)) = \rank(g(x))$ for all $x \in X$. Given $f \in C(X,\KK)^+$, we define $\sigma([f]_\Cu)(x) \coloneqq \rank(f(x))$. It is straightforward to check that $\sigma\colon \Cu(C(X)) \to \Lsc(X,\bar{\N})$ is a generalised $\Cu$-morphism and that $\osupp(\sigma([f]_\Cu)) \subseteq \osupp(f)$ for each $f \in C(X,\KK)^+$.

\begin{prop} \label{prop: retract Lcs}
Let $X$ be a second-countable compact Hausdorff space. Let
\[
\rho\colon \Lsc(X,\bar{\N}) \to \Cu(C(X)) \quad \text{ and } \quad \sigma\colon \Cu(C(X)) \to \Lsc(X,\bar{\N})
\]
be the $\Cu$-morphism and generalised $\Cu$-morphism (respectively) defined above. Then $\Lsc(X,\bar{\N})$ is a retract of $\Cu(C(X))$, and $\sigma\big(\Cu(C(X))^\ll\big) = \Lsc(X,\bar{\N})^\ll$.
\end{prop}

\begin{proof}
Fix an open set $U \subseteq X$, and choose $f_U \in C(X)^+$ such that $\osupp(f_U) = U$. Recall that there is an embedding $C(X) \hookrightarrow C(X,\KK)$ that maps $f_U$ to the function $X \owns x \mapsto f_U(x) \, e_{1,1} \in \KK$. Thus, for each $x \in X$,
\[
\sigma\big(\rho(1_U)\big)(x) = \sigma([f_U]_\Cu)(x) = \rank(f_U(x)) = \rank(e_{1,1}) \, 1_{\osupp(f_U)}(x) = 1_U(x).
\]
It follows that $\sigma \circ \rho = \id_{\Lsc(X,\bar{\N})}$, and hence $\Lsc(X,\bar{\N})$ is a retract of $\Cu(C(X))$.

We now show that $\sigma\big(\Cu(C(X))^\ll\big) = \Lsc(X,\bar{\N})^\ll$. By \cref{item: R^ll and T^ll}, we have $\Lsc(X,\bar{\N})^\ll \subseteq \sigma\big(\Cu(C(X))^\ll\big)$. For the reverse containment, fix $f \in C(X,\KK)^+ \cong \big(C(X) \otimes \KK)^+$ such that $[f]_\Cu \in \Cu(A)^\ll$. Then $[f]_\Cu \ll [g]_\Cu$ for some $g \in C(X,\KK)^+$. Let $\{g_n\}_{n \in \N} \subseteq C(X,\KK)^+$ be an increasing sequence such that $g_n \in C(X,M_n(\C))^+$ for each $n \in \N$ and $g = \lim_{n \to \infty} g_n$. By \cite[Lemma~2.57]{Thiel2017}, $\big\{[g_n]_\Cu\big\}_{n \in \N} \subseteq \Cu(C(X))$ is an increasing sequence and $\sup_{n \in \N} \big([g_n]_\Cu\big) = [g]_\Cu$. Since $[f]_\Cu \ll [g]_\Cu$, there exists $m \in \N$ such that $[f]_\Cu \le [g_m]_\Cu$. Since $\sigma$ is order-preserving and $g_m \in C(X,M_m(\C))^+$, it follows that for all $x \in X$,
\[
\sigma([f]_\Cu)(x) \le \sigma([g_m]_\Cu)(x) = \rank(g_m(x)) \le m.
\]
Therefore, $\sigma([f]) \in \Lsc(X,\bar{\N})^\ll$, and so $\sigma\big(\Cu(C(X))^\ll\big) \subseteq \Lsc(X,\bar{\N})^\ll$.
\end{proof}

\begin{remark}
When $X$ has low covering dimension we can say more about $\Lsc(X,\bar{\N})$ and $\Cu(C(X))$. Let $X$ be a second-countable compact Hausdorff space with $\dim(X) \le 1$ (or $\dim X \le 2$ with an additional cohomological assumption). Then by \cite[Theorem~1.1]{Robert2013b}, the map $\sigma\colon \Cu(C(X)) \to \Lsc(X,\bar{\N})$ is an isomorphism (that is, a bijective $\Cu$-morphism). On the other hand, when $\dim(X) \ge 3$, $\sigma$ is not necessarily injective (or a $\Cu$-morphism).

If we remove the second-countability assumption on $X$ (so that $X$ is just a compact Hausdorff space) and place no restrictions on the covering dimension of $X$, then we need to consider a different collection of functions to obtain a retract of $\Cu(C(X))$. Consider the collection $\Lsc_\sigma(X,\bar{\N})$ of functions $F\colon X \to \bar{\N}$ such that for each $i \in \N$, the set $F^{-1}(\{i, i+1, \dotsc \} \cup \{\infty\})$ is $\sigma$-compact and open in $X$. This is a $\Cu$-semigroup (see \cite[Proposition~2.13]{ElliottIm2024}), and a similar argument to \cite[Corollary~4.8]{ThielVilalta2022} would show that it is a retract of $\Cu(C_0(X))$. (This result is referenced above \cite[Proposition~2.13]{ElliottIm2024}, but does not seem to have appeared in the literature.)
\end{remark}

We now show that the retract $\Lsc(X,\bar{\N})$ is strongly invariant (in the sense of \cref{defn: invariant retract}) under the action $\hat{\theta}\colon S(\GG) \curvearrowright \Cu(C(X))$ induced by the canonical action $\theta\colon S(\GG) \curvearrowright C(X)$, and that $\Lsc(X,\bar{\N})$ has the appropriate structure for us to apply the results of \cref{sec: retracts}.

\begin{prop} \label{prop: Lsc strongly invariant retract}
Let $\GG$ be an \'etale groupoid with second-countable compact Hausdorff unit space $X$, and let $\theta\colon S(\GG) \curvearrowright C(X)$ be the canonical action. Define $\rho\colon \Lsc(X,\bar{\N}) \to \Cu(C(X))$ and $\sigma\colon \Cu(C(X)) \to \Lsc(X,\bar{\N})$ as above. Then $\Lsc(X,\bar{\N})$ is a strongly invariant retract of $\Cu(C(X))$, and $[f]_\Cu \in \rho\big(\Lsc(X,\bar{\N})\big)$ for each $f \in C(X)^+$. The action $\hat{\theta}^\Lsc\colon S(\GG) \curvearrowright \Lsc(X,\bar{\N})$ (defined as in \cref{thm: invariant retract induces action}) satisfies $\hat{\theta}^\Lsc_U(F) = F \circ \theta_U^{-1}$ for all $U \in S(\GG)$ and $F \in \Lsc(X,\bar{\N})$. In particular, given $U \in S(\GG)$ and any open set $V \subseteq s(U)$, we have $\hat{\theta}^\Lsc_U(1_V) = 1_{r(UV)} = 1_{UVU^*}$.
\end{prop}

\begin{proof}
For each $f \in C(X)^+$, we have $[f]_\Cu = \rho(1_{\osupp(f)}) \in \rho\big(\Lsc(X,\bar{\N})\big)$.

Let $\hat{\theta}\colon S(\GG) \curvearrowright \Cu(C(X))$ be the action induced by $\theta\colon S(\GG) \curvearrowright C(X)$, as defined on \cpageref{page: alpha hat s}. To see that $\rho\big(\Lsc(X,\bar{\N})\big)$ is invariant under $\hat{\theta}$, fix $U \in S(\GG)$, and recall that $C(X)_{U^*U} = C_0(s(U))$. Fix $a \in \Cu\big(C_0(s(U))\big) \cap \rho\big(\Lsc(X,\bar{\N})\big)$, and let $F$ be the unique element of $\Lsc(X,\bar{\N})$ such that $\rho(F) = a$. Then $\osupp(F) \subseteq s(U)$. For each $i \ge 1$, let $V_i \coloneqq F^{-1}(\{ i, i+1, \dotsc \} \cup \{\infty\})$, and choose $f_{V_i} \in C(X)^+$ such that $\osupp(f_{V_i}) = V_i \subseteq s(U)$. Then
\[
F = \sum_{i=1}^\infty 1_{V_i} \quad \text{ and } \quad a = \rho(F) = \sum_{i=1}^\infty \rho(1_{V_i}) = \sum_{i=1}^\infty [f_{V_i}]_\Cu.
\]
By \cref{eqn: groupoid action}, for each $i \ge 1$ and $x \in X$, we have $\theta_U(f_{V_i})(x) = f_{V_i}(U^*xU)$, and hence $\osupp(\theta_U(f_{V_i})) = U V_i U^* = r(U V_i)$, which is an open subset of $X$ because the range and multiplication maps of \'etale groupoids are open. By \cref{item: Cuntz equivalence for functions}, functions in $C(X)^+$ are Cuntz equivalent if and only if they have the same open support, so it follows that
\begin{equation} \label{eqn: theta hat U}
\hat{\theta}_U([f_{V_i}]_\Cu) = [\theta_U(f_{V_i})]_\Cu = \rho(1_{r(U V_i)})
\end{equation}
for each $i \ge 1$. Therefore,
\[
\hat{\theta}_U(a) = \hat{\theta}_U(\rho(F)) = \sum_{i=1}^\infty \hat{\theta}_U\big(\rho(1_{V_i})\big) = \sum_{i=1}^\infty \hat{\theta}_U([f_{V_i}]_\Cu) = \sum_{i=1}^\infty \rho(1_{r(U V_i)}) = \rho\Big( \sum_{i=1}^\infty 1_{r(U V_i)} \Big),
\]
and thus $\rho\big(\Lsc(X,\bar{\N})\big)$ is invariant under $\hat{\theta}$. So $\Lsc(X,\bar{\N})$ is an invariant retract of $\Cu(C(X))$, and hence \cref{thm: invariant retract induces action} implies that there is an action $\hat{\theta}^\Lsc\colon S(\GG) \curvearrowright \Lsc(X,\bar{\N})$ such that for each $U \in S(\GG)$,
\[
\Lsc(X,\bar{\N})_{U^*U} = \rho^{-1}\big(\Cu\big(C_0(s(U))\big)\big) \quad \text{ and } \quad \hat{\theta}^\Lsc_U = \sigma \circ \hat{\theta}_U \circ \rho\restr{\Lsc(X,\bar{\N})_{U^*U}}.
\]
Fix $U \in S(\GG)$, and let $V$ be an open subset of $s(U)$. Then $1_V \in \Lsc(X,\bar{\N})_{U^*U}$. Choose a function $f_V \in C(X)^+$ such that $\osupp(f_V) = V$. Then \cref{eqn: theta hat U} implies that
\begin{equation} \label{eqn: theta hat U 1_V}
\hat{\theta}^\Lsc_U(1_V) = \sigma\big(\hat{\theta}_U(\rho(1_V))\big) = \sigma\big(\hat{\theta}_U([f_V]_\Cu)\big) = \sigma\big(\rho(1_{r(UV)})\big) = 1_{r(UV)} = 1_{UVU^*}.
\end{equation}
Now fix $F \in \Lsc(X,\bar{\N})_{U^*U}$, and write
\[
F = \sum_{i=1}^\infty 1_{V_i}, \quad \text{ where } V_i = F^{-1}(\{ i, i+1, \dotsc \} \cup \{\infty\}) \subseteq s(U).
\]
Then $\osupp\!\big(\hat{\theta}^\Lsc_U(F)\big) \subseteq UU^* = r(U)$, and \cref{eqn: theta hat U 1_V} implies that for all $x \in r(U)$, we have
\begin{equation} \label{eqn: theta hat Lsc U of F}
\hat{\theta}^\Lsc_U(F)(x) = \sum_{i=1}^\infty \hat{\theta}^\Lsc_U(1_{V_i})(x) = \sum_{i=1}^\infty 1_{U V_i U^*}(x) = \sum_{i=1}^\infty 1_{V_i}(U^*xU) = F(U^*xU) = F\big(\theta_U^{-1}(x)\big).
\end{equation}
Thus $\hat{\theta}^\Lsc_U(F) = F \circ \theta_U^{-1}$.

We now show that $\Lsc(X,\bar{\N})$ is strongly invariant. For this, fix $U \in S(\GG)$. To see that $\rho\big(\sigma(C(X)_{U^*U})\big) \subseteq C(X)_{U^*U}$, fix $f \in C_0(s(U), \KK)^+ \cong \big( C_0(s(U)) \otimes \KK \big)^+$, and let $F \coloneqq \sigma([f]_\Cu)$. Recall that $\osupp(F) \subseteq \osupp(f) \subseteq s(U)$, and write
\[
F = \sum_{i=1}^\infty 1_{V_i}, \quad \text{ where } V_i = F^{-1}(\{ i, i+1, \dotsc \} \cup \{\infty\}) \subseteq s(U).
\]
For each $i \ge 1$, choose $f_{V_i} \in C(X)^+$ such that $\osupp(f_{V_i}) = V_i \subseteq s(U)$. Then
\[
\rho\big(\sigma([f]_\Cu)\big) = \rho(F) = \sum_{i=1}^\infty \rho(1_{V_i}) = \sum_{i=1}^\infty [f_{V_i}]_\Cu \in \Cu\big(C_0(s(U))\big) = \Cu(C(X))_{U^*U}.
\]
Therefore, $\rho\big(\sigma(C(X)_{U^*U})\big) \subseteq C(X)_{U^*U}$. To see that $\sigma \circ \hat{\theta}_U \circ \rho \circ \sigma\restr{\Cu(C(X))_{U^U*}} = \sigma \circ \hat{\theta}_U$, it suffices by the definition of $\hat{\theta}^\Lsc$ to show that $\hat{\theta}^\Lsc_U \circ \sigma\restr{\Cu(C(X))_{U^*U}} = \sigma \circ \hat{\theta}_U$. For this, fix $g \in C_0(s(U),\KK)^+ \cong \big( C_0(s(U)) \otimes \KK \big)^+$. Then $[g]_\Cu \in \Cu(C(X))_{U^*U}$, and $\hat{\theta}_U([g]_\Cu) = [\theta^\KK_U(g)]_\Cu$, where $\theta^\KK_U = \theta_U \otimes \id_\KK$. Fix $x \in UU^* = r(U)$. Using \cref{eqn: groupoid action} for the second equality and \cref{eqn: theta hat Lsc U of F} for the final equality, we see that
\[
\sigma\big(\hat{\theta}_U([g]_\Cu)\big)(x) = \rank\!\big(\theta^\KK_U(g)(x)\big) = \rank\!\big(g(U^*xU)\big) = \sigma([g]_\Cu)(U^*xU) = \hat{\theta}^\Lsc_U\big(\sigma([g]_\Cu)\big)(x),
\]
as required. Therefore, $\Lsc(X,\bar{\N})$ is a strongly invariant retract of $\Cu(C(X))$.
\end{proof}

We now use \cref{prop: Lsc strongly invariant retract} to give an alternative characterisation of the $\hat{\theta}^\Lsc$-below relation $\preceq_\theta^\Lsc$ defined in \cref{defn: retracted alpha-below}.

\begin{lemma} \label{lemma: theta^Lsc-below characterisation}
Let $\GG$ be an \'etale groupoid with second-countable compact Hausdorff unit space $X$, and let $\theta\colon S(\GG) \curvearrowright C(X)$ be the canonical action. Fix $F,H \in \Lsc(X,\bar{\N})$. Then $F \preceq_\theta^\Lsc H$ if and only if, for every $L \in \Lsc(X,\bar{\N})$ such that $L \ll F$, there exist open bisections $\{U_i\}_{i=1}^n \subseteq S(\GG)$ and open sets $\big\{ V_i \subseteq s(U_i) : i \in \{1, \dotsc, n\} \big\}$ such that
\begin{inequality} \label{ineq: theta^Lsc-below 1_{V_i}}
L \ll \sum_{i=1}^n 1_{V_i} \quad \text{ and } \quad \sum_{i=1}^n \hat{\theta}^\Lsc_{U_i}(1_{V_i}) = \sum_{i=1}^n 1_{r(U_i V_i)} \ll H.
\end{inequality}
\end{lemma}

\begin{proof}
Suppose that $F \preceq_\theta^\Lsc H$. Fix $L \in \Lsc(X,\bar{\N})$ such that $L \ll F$. By \cref{defn: retracted alpha-below}, there exist $\{W_i\}_{i=1}^m \subseteq S(\GG)$ and $\big\{ K_i \in \Lsc(X,\bar{\N})_{W_i^*W_i} : i \in \{1, \dotsc, m\} \big\}$ such that
\[
L \ll \sum_{i=1}^m K_i \quad \text{ and } \quad \sum_{i=1}^m \hat{\theta}^\Lsc_{U_i}(K_i) \ll H.
\]
For each $i \in \{1, \dotsc, m\}$, we have $\hat{\theta}^\Lsc_{U_i}(K_i) \ll H$ by \cref{lemma: ll le properties}, and hence $\im\!\big(\hat{\theta}^\Lsc_{U_i}(K_i)\big)$ is a finite subset of $\N$. Thus we can replace each $K_i$ with a finite sum of characteristic functions on open subsets of $X$, giving the desired result. The reverse implication is trivial.
\end{proof}

\begin{remark}
In \cref{lemma: theta^Lsc-below characterisation} we have simplified the $\preceq_\theta^\Lsc$ relation by replacing arbitrary elements of $\Lsc(X,\bar{\N})$ with characteristic functions on open subsets of $X$. This is possible because any function $L \in \Lsc(X,\bar{\N})$ that satisfies $L \ll H$ for some $H \in \Lsc(X,\bar{\N})$ must take finitely many values in $\N$. Furthermore, in \cref{lemma: theta^Lsc-below characterisation}, each $U_i V_i$ is a bisection of $\GG$ with $s(U_i V_i) = V_i$, because $V_i \subseteq s(U_i)$. Thus \cref{ineq: theta^Lsc-below 1_{V_i}} shows that $\preceq_\theta^\Lsc$ coincides with the relations described \cite[Definition~4.3]{Ma2022} and \cite[Definition~4.16]{KMP2025}.
\end{remark}

It follows from \cite[Lemma~4.13]{KMP2025} that $\Lsc(X,\bar{\N})$ has almost refinement, and hence $\preceq_\theta^\Lsc$ is transitive by \cref{lemma: almost refinement implies precalpha transitive}. Thus we do not need to consider the path-below relation $\preceq_{\theta,\rmpath}^\Lsc$. Let $\sim_\theta^\Lsc$ denote the relation on $\Lsc(X,\bar{\N})^\ll$ given by
\[
F \sim_\theta^\Lsc H \iff F \preceq_\theta^\Lsc H \,\text{ and }\, H \preceq_\theta^\Lsc F.
\]
Then $\sim_\theta^\Lsc$ is an equivalence relation. For each $F \in \Lsc(X,\bar{\N})^\ll$, we write $[F]_\theta^\Lsc$ for the $\sim_\theta^\Lsc$-equivalence class of $F$.

\begin{defn} \label{defn: dynamical retract}
Let $\GG$ be an \'etale groupoid with second-countable compact Hausdorff unit space $X$, and let $\theta\colon S(\GG) \curvearrowright C(X)$ be the canonical action. We define the \emph{dynamical retract of $\GG$} to be the preordered abelian monoid $\Lsc(\GG) \coloneqq \Lsc(X,\bar{\N})^\ll / \!\!\sim_\theta^\Lsc$. (Since $\sigma\big(\Cu(C(X))^\ll\big) = \Lsc(X,\bar{\N})^\ll$ by \cref{prop: retract Lcs}, this is the retracted dynamical Cuntz semigroup with respect to the canonical action $\theta$.)
\end{defn}

Using the dynamical retract $\Lsc(\GG)$ of \cref{defn: dynamical retract}, the following are straightforward corollaries of \cref{thm: retract stable finiteness,thm: retract PIS}, respectively.

\begin{cor} \label{cor: stable finiteness groupoid}
Let $\GG$ be an \'etale groupoid whose unit space $X$ is second-countable, compact, and Hausdorff. Suppose that $C_\red^*(\GG)$ is simple. Then $C_\red^*(\GG) = C_\ess^*(\GG)$, and the following statements are equivalent.
\begin{enumerate}[label=(\arabic*)]
\item There exists a faithful invariant tracial state of $C(X)$.
\item There exists a faithful tracial state of $C_\ess^*(\GG)$.
\item The C*-algebra $C_\ess^*(\GG)$ is stably finite.
\item The monoid $\Lsc(X,\bar{\N})^\ll$ is completely non-paradoxical.
\item There exists $F \in \Lsc(X,\bar{\N})^\ll {\setminus} \{0\}$ that is completely non-paradoxical.
\item There exists a nontrivial order-preserving homomorphism $\nu\colon \Lsc(\GG) \to [0,\infty]$.
\item There exists a faithful invariant functional $\beta\colon \Cu(C(X)) \to [0,\infty]$ that is finite on $\Cu(C(X))^\ll$ and satisfies $\beta([1_X]_\Cu) = 1$.
\end{enumerate}
\end{cor}

\begin{cor} \label{cor: pure infiniteness groupoid}
Let $\GG$ be an \'etale groupoid whose unit space $X$ is second-countable, compact, and Hausdorff. Suppose that $C_\red^*(\GG)$ is simple. Then $C_\red^*(\GG) = C_\ess^*(\GG)$. Consider the following statements.
\begin{enumerate}[label=(\arabic*)]
\item \label{item: PIS groupoid PIM} The dynamical retract $\Lsc(\GG)$ is purely infinite.
\item \label{item: PIS groupoid 21P} Every element of $\Lsc(X,\bar{\N})^\ll {\setminus} \{0\}$ is $(2,1)$-paradoxical.
\item \label{item: PIS groupoid PIC*} The C*-algebra $C_\ess^*(\GG)$ is purely infinite.
\item \label{item: PIS groupoid NTS} The C*-algebra $C_\ess^*(\GG)$ admits no tracial states.
\item \label{item: PIS groupoid OPH} There does not exist a nontrivial order-preserving homomorphism $\nu\colon \Lsc(\GG) \to [0,\infty]$.
\end{enumerate}
Then \cref{item: PIS groupoid PIM} $\implies$ \cref{item: PIS groupoid 21P} $\implies$ \cref{item: PIS groupoid PIC*} $\implies$ \cref{item: PIS groupoid NTS} $\implies$ \cref{item: PIS groupoid OPH}. If the dynamical retract $\Lsc(\GG)$ has plain paradoxes, then \cref{item: PIS groupoid OPH} $\implies$ \cref{item: PIS groupoid PIM}, and \crefrange{item: PIS groupoid PIM}{item: PIS groupoid OPH} are equivalent.
\end{cor}

We end the paper with an application of \cref{thm: retract dichotomy}, which is a special case of \cite[Corollary~6.9]{KMP2025}.

\begin{cor} \label{cor: dichotomy groupoid}
Let $\GG$ be an \'etale groupoid whose unit space $X$ is second-countable, compact, and Hausdorff. Suppose that $C_\red^*(\GG)$ is simple, and that $\Lsc(\GG)$ has plain paradoxes. Then $C_\ess^*(\GG)$ is either stably finite or purely infinite.
\end{cor}

\begin{remark}
Let $\GG$ be a topological groupoid. Recall that the \emph{isotropy} of $\GG$ is the subgroupoid $\Iso(\GG) \coloneqq \{ \gamma \in \GG : r(\gamma) = s(\gamma) \}$. We say that $\GG$ is \emph{topologically free} if $\Iso(\GG) {\setminus} \GGo$ has empty interior. If $\GG$ is an \'etale groupoid whose unit space $X$ is locally compact and Hausdorff, then $\GG$ is topologically free if and only if the canonical action $\theta\colon S(\GG) \curvearrowright C(X)$ is aperiodic (by \cite[Theorem~6.13]{KwasniewskiMeyer2021}), and $\GG$ is minimal if and only if $\theta$ is minimal.

The condition in \cref{cor: stable finiteness groupoid,cor: pure infiniteness groupoid,cor: dichotomy groupoid} that $C_\red^*(\GG)$ is simple holds in particular when $\GG$ is minimal and topologically free and the singular ideal of $C_\red^*(\GG)$ is trivial (so that $C_\red^*(\GG) = C_\ess^*(\GG)$; see \cref{rem: simplicity and aperiodic actions and detecting ideals}\cref{item: aperiodic implies simple iff minimal}). The singular ideal and useful criteria for its triviality have been intensively studied in the literature; see, for instance, \cite{BGHL2025, Hume2025}.
\end{remark}

\vspace{2ex}
\bibliographystyle{amsplain}
\makeatletter\renewcommand\@biblabel[1]{[#1]}\makeatother
\bibliography{references}
\vspace{2ex}

\end{document}